\documentclass[AMS,STIX1COL]{WileyNJD-v2}
\usepackage{moreverb}
\usepackage{bm}
\usepackage{mathtools}
\usepackage{enumitem}
\usepackage{diagbox}

\newcommand\BibTeX{{\rmfamily B\kern-.05em \textsc{i\kern-.025em b}\kern-.08em
		T\kern-.1667em\lower.7ex\hbox{E}\kern-.125emX}}

\articletype{Original Article}%

\received{30 December, 2021}
\revised{18 January, 2023; 4 September, 2023}
\accepted{6 September, 2023}

\begin{document}
	
	\title{Nonlocality, Nonlinearity, and Time Inconsistency in Stochastic Differential Games}

\author[1]{Qian Lei}

\author[1]{Chi Seng Pun*}


\authormark{Lei and Pun}

\address[1]{\orgdiv{School of Physical and Mathematical Sciences}, \orgname{Nanyang Technological University}, \orgaddress{\country{Singapore}}}



\corres{*Chi Seng Pun, 21 Nanyang Link, Singapore 637371. \email{cspun@ntu.edu.sg}}


\abstract[Abstract]{This paper studies the well-posedness of a class of nonlocal fully nonlinear parabolic systems, which nest the equilibrium Hamilton--Jacobi--Bellman (HJB) systems that characterize the time-consistent Nash equilibrium point of a stochastic differential game (SDG) with time-inconsistent (TIC) preferences. The nonlocality of the parabolic systems stems from the flow feature (controlled by an external temporal parameter) of the systems. This paper proves the existence and uniqueness results as well as the stability analysis for the solutions to such systems. We first obtain the results for the linear cases for an arbitrary time horizon and then extend them to the quasilinear and fully nonlinear cases under some suitable conditions. {\color{black}Two examples of TIC SDG are provided to illustrate financial applications} with global solvability.
	Moreover, with the well-posedness results, we establish a general multidimensional Feynman--Kac formula in the presence of nonlocality (time inconsistency).

}

\jnlcitation{\cname{%
	\author{Q. Lei} and
	\author{C. S. Pun}} (\cyear{2023}), 
\ctitle{Nonlocality, Nonlinearity, and Time Inconsistency in Stochastic Differential Games}, \cjournal{Math. Finance}, \cvol{2023;00:1--47}.}

\maketitle


\section{Introduction} \label{sec:intro}
The aim of this paper is to study the well-posedness of nonlocal fully nonlinear higher-order systems for the unknown $\bm{u}$ of the form 
\begin{equation} \label{Nonlocal fully nonlinear system}
\left\{
\begin{array}{rcl}
	\bm{u}_s(t,s,y) & = & \bm{F}\big(t,s,y,\left(\partial_I\bm{u}\right)_{|I|\leq 2r}(t,s,y),  \left(\partial_I\bm{u}\right)_{|I|\leq 2r}(s,s,y)\big), \\
	\bm{u}(t,0,y) & = & \bm{g}(t,y),\hfill 0\leq s\leq t\leq T,\quad y\in\mathbb{R}^d,
\end{array}
\right. 
\end{equation} 
where $\bm{u},\bm{F},\bm{g}$ are $m$-dimensional real-valued, $r$ is a positive integer, the nonlinearity $\bm{F}$ could be nonlinear with respect to its all arguments, and both $s$ and $y$ are dynamical variables while $t$ should be considered as an external parameter. 
Here, $I=(i_1,\ldots,i_j)$ is a multi-index with $j=|I|$, and $\partial_I\bm{u}:=\frac{\partial^{|I|}\bm{u}}{\partial y_{i_1}\cdots\partial y_{i_j}}$. To clarify, the system \eqref{Nonlocal fully nonlinear system} consists of $m$ coupled $\mathbb{R}$-valued nonlocal fully nonlinear equations $\bm{u}^a$ $(a=1,\ldots,m)$: 
\begin{equation} \label{eq:indnonlocaleq}   
\left\{
\begin{array}{lr}
	\bm{u}^a_s(t,s,y)=\bm{F}^a\Big(t,s,y,\bm{u}^1(t,s,y),\ldots,\bm{u}^m(t,s,y),\frac{\partial}{\partial y_1}\bm{u}^1(t,s,y),\ldots,\frac{\partial^{2r}}{\partial y_d\cdots\partial y_d}\bm{u}^m(t,s,y), \\
	\qquad\qquad\qquad\qquad\qquad \bm{u}^1(s,s,y),\ldots,\bm{u}^m(s,s,y),\frac{\partial}{\partial y_1}\bm{u}^1(s,s,y),\ldots,\frac{\partial^{2r}}{\partial y_d\cdots\partial y_d}\bm{u}^m(s,s,y)\Big),  \\
	\bm{u}^a(t,0,y)=\bm{g}^a(t,y),\hfill 0\leq s\leq t\leq T,\quad y\in\mathbb{R}^d, \quad a=1,\ldots,m,  
\end{array}
\right. 
\end{equation}
where the superscripts $a$ of $\bm{u},\bm{F},\bm{g}$ represent the $a$-th entry of the corresponding vector functions. The systems above characterize a series of systems indexed by $t$, which are connected via their dependence on $\left(\partial_I\bm{u}\right)_{|I|\leq 2r}(s,s,y)$. The diagonal dependence, referring to that $\left(\partial_I\bm{u}\right)_{|I|\leq 2r}$ of $\bm{F}$ are evaluated not only at $(t,s,y)$ but also at $(s,s,y)$, directly results in nonlocality. When $r=1$ and $m=1$, the system \eqref{Nonlocal fully nonlinear system} or \eqref{eq:indnonlocaleq} is reduced to nonlocal fully nonlinear parabolic partial differential equations (PDEs) studied in \cite{Lei2023}.

The diagonal dependence is an inevitable consequence when we look for an equilibrium solution for a time-inconsistent (TIC) dynamic choice problem. The TIC problem was first seriously studied in \cite{Strotz1955} from an economic perspective, in which a consistent planning of control policies, characterized by a Nash equilibrium (NE), is considered as a suitable solution; see \cite{Pollak1968}. With the popularity of behavioral economics, it has been recognized that the time inconsistency (also abbreviated as TIC) or dynamic inconsistency is prevalent when we consider behavioral factors in dynamic choice problems; see \cite{Thaler1981} for some empirical evidence, \cite{Laibson1997,Frederick2002} for non-exponential discounting, \cite{Kahneman1979,Tversky1992} for (cumulative) prospect theory. However, a rigorous treatment in continuous-time settings was only available a decade ago; see \cite{He2022} for a review. \cite{Basak2010,Pun2018a} adopt a recursive approach, originally suggested in \cite{Strotz1955}, to derive a time-consistent (TC) portfolio strategy for the mean-variance (MV) investor in continuous time. The MV analysis, pioneered by \cite{Markowitz1952}, is free of TIC in the static setting (single period) but the dynamic choice problem under the MV criterion is TIC due to the nonlinearity of the variance operator. Subsequently, \cite{Bjoerk2014,Bjoerk2017} establishes heuristic analytical frameworks for discrete- and continuous-time TIC stochastic control problems, which can successfully address the state-dependence issue (a type of TIC) of risk aversion in portfolio selection \cite{Bjoerk2014a}, but they leave behind many open problems, including the existence and uniqueness of the solution to their deduced Hamilton--Jacobi--Bellman (HJB) equation system.

Using a discretization approach, \cite{Yong2012,Wei2017} derive a so-called equilibrium HJB equation, which accords with the HJB system in \cite{Bjoerk2017}, to characterize the equilibrium solutions to TIC stochastic control problems. The equilibrium HJB equation is a nonlocal PDE, whose nonlocality comes from the diagonal dependence, and is a special case of \eqref{Nonlocal fully nonlinear system}. \cite{Wei2017} established the existence and uniqueness results for the equation in a quasilinear setting, which requires the linear dependence on the second-order derivative at local point $(t,s,y)$ and the removal of the second-order derivative at diagonal $(s,s,y)$, such that the diffusion term of the state process was restricted to be uncontrolled. From the perspective of stochastic differential equations (SDEs), \cite{Wang2019,Hamaguchi2020,Wang2020,Hamaguchi2021,Wang2021} made attempts to the TIC problem with a flow of forward-backward SDEs (FBSDEs) or backward stochastic Volterra integral equation (BSVIE) but their results are still subject to the same restriction as in the PDE theory or limitation to the first-order dependence. However, their work has converted the key open problem in \cite{Bjoerk2017} to another open problem in PDE or SDE theory. Recently, \cite{Lei2023} has shown the existence and uniqueness of a nonlocal fully nonlinear parabolic PDE in a small-time setting.

Following \cite{Lei2023}, this paper extends the well-posedness results from a nonlocal fully nonlinear second-order PDE to a system of coupled nonlocal fully nonlinear higher-order PDEs, from small-time (in the sense of maximally defined regularity) to global settings, {\color{black}and from the conventional space of bounded functions to a weighted space with exponential growth functions}. Similarly, the essential difficulty for constructing a desired contraction (to use fixed-point arguments) is the presence of the highest-order diagonal term $\left(\partial_{I}\bm{u}\right)_{|I|=2r}(s,s,y)$.
To see this, we discuss an intuitive attempt heuristically, from which we outline the distinct feature of our problem. To show the existence and uniqueness of \eqref{Nonlocal fully nonlinear system}, it is intuitive to consider a mapping from $\bm{u}$ to $\bm{w}$ that satisfies
\begin{equation} \label{Mapping from u to w} 
\left\{
\begin{array}{lr}
	\bm{w}_s(t,s,y)=\bm{F}\big(t,s,y,\left(\partial_I\bm{w}\right)_{|I|\leq 2r}(t,s,y),  \left(\partial_I\bm{u}\right)_{|I|\leq R}(s,s,y)\big), \\
	\bm{w}(t,0,y)=\bm{g}(t,y),\hfill 0\leq s\leq t\leq T,\quad y\in\mathbb{R}^d.
\end{array}
\right. 
\end{equation}
Thanks to the classical theory of parabolic systems \cite{Solonnikov1965,Ladyzhanskaya1968,Eidelman1969}, the mapping is well-defined. By replacing the intractable diagonal term with a known vector-valued function $\bm{u}$, the well-posedness of the higher-order system \eqref{Mapping from u to w} parameterized by $t$ promises the existence and uniqueness of the solution $\bm{w}$. Moreover, if $R=2r$, it is clear that the fixed point solves the original system \eqref{Nonlocal fully nonlinear system}. However, for the case of $R=2r$, since the input $\bm{u}$ is of the same order of the output $\bm{w}$, it is not immediate to show the contraction. The curse has limited all aforementioned works to a restricted case of $R=1$ and $r=1$ except for \cite{Lei2023} that extends the study to the case of $R=2$ and $r=1$. This paper leverages on the techniques developed in \cite{Lei2023} to further extend the study to $R=2r$ with an arbitrary positive integer $r$.

The mathematical extension achieved in this paper has two immediate implications, namely the establishments of mathematical foundation of TIC stochastic differential games (SDGs) and a general multidimensional Feynman--Kac formula under the framework of nonlocality.
\begin{enumerate}
\item Some specific types of TIC SDGs are studied in \cite{Wei2018,Lazrak2023} for zero-sum games and \cite{Wang2021b} for nonzero-sum games but they still left behind the existence and uniqueness results. In this paper, we will show the relevance of the system \eqref{Nonlocal fully nonlinear system} by introducing a general formulation of nonzero-sum TIC SDGs. Under some regularity assumptions as in \cite{Friedman1972,Friedman1976,Bensoussan2000}, we yield the parabolic systems for the TIC SDGs as a special case of \eqref{Nonlocal fully nonlinear system}.
\item The extended version of Feynman--Kac formula paves a new path for the SDE theory to study a flow of the multidimensional second-order FBSDEs (or 2FBSDEs), which is also called the multidimensional second-order BSVIEs (or 2BSVIEs), where both forward and backward SDEs are multidimensional. The 2BSDEs were first introduced in \cite{Cheridito2007} to provide a probabilistic interpretation of a fully nonlinear parabolic PDE. 
\end{enumerate}

To clarify, our paper considers only the TIC caused by the initial-\textit{time}-dependence in the control/game problems, and thus the parabolic systems of our interest \eqref{Nonlocal fully nonlinear system} only involve the nonlocality with a two-time-variable structure. It is noteworthy that
the initial-\textit{state}-dependence and nonlinearity of conditional expectations also form the sources of TIC and there exist similar arguments to convert the control/game problem into a parabolic PDE systems with nonlocality in state; see \cite{Bjoerk2017,Landriault2018,Hu2012,Hu2017,Hernandez2023,Hernandez2021,Yan2019}. We do not attempt the initial-state-dependence in this paper as it {\color{black}poses} technical challenges. Its key difference from our consideration is that the state variable is multidimensional and unrestricted, whereas the time variable is naturally bounded especially in a finite-time framework.

This paper contributes to the theories of PDE, SDE, and SDG, especially for the treatment of nonlocality in the multidimensional setting. Specifically,
\begin{description}
\item[Section \ref{sec:sdg} (SDG aspect)] formulates TIC SDGs that incorporate with TIC behavioral factors, which facilitate developments of many studies in financial economics including robust stochastic controls and games under relative performance concerns. We heuristically derive the associate equilibrium HJB systems and reveal its relation with the TIC SDGs. Our focus is then placed on the well-posedness of such nonlocal systems as it serves as the prerequisite of using its solution to characterize the solution to the TIC SDGs. Noteworthy is that our study allows the diffusion of the state process to be controllable{\color{black}, which breaks through the existing bottleneck of time-inconsistent stochastic control problems.}
\item[Section \ref{sec:mainresults} (PDE aspect)] presents our main results of well-posedness of nonlocal higher-order systems in linear, quasilinear, and fully nonlinear settings individually. Our results generalize the existing studies while potential extensions are discussed. {\color{black}To our best knowledge, our well-posedness results in a larger function space that accommodate more complex research objects over a longer time horizon open the frontier of the existing literature on nonlocal PDEs/systems.}  
\item[Section \ref{Analysis of equilbirum HJB systems} (SDG and PDE)] analyzes the solvability of the equilibrium HJB systems in Section \ref{sec:sdg} with the well-posedness results in Section \ref{sec:mainresults}. Moreover, we illustrate {\color{black}two financial examples of TIC SDG that are globally solvable.} 
\item[Section \ref{sec:stochrep} (SDE and PDE)] provides a \textit{nonlocal Feynman--Kac formula} linking the solution to a flow of multidimensional 2FBSDEs to that of a nonlocal fully nonlinear parabolic system.
\item[Section \ref{sec:conclusion}] concludes.
\end{description}

\section{Nonzero-Sum Time-Inconsistent Stochastic Differential Games} \label{sec:sdg}
In this section, we follow the frameworks of \cite{Friedman1972,Friedman1976,Bensoussan2000} to formulate general $m$-player nonzero-sum TIC SDGs, where preferences and utility functions for each player are time-varying.

Let $\left(\Omega,\mathcal{F},\mathbb{F},\mathbb{P}\right)$ be a completed filtered probability space on which a $k$-dimensional standard Brownian motion $\{\bm{W}(\tau)\}_{\tau\geq 0}$ with the natural filtration $\mathbb{F}=\left\{\mathcal{F}_\tau\right\}_{\tau\geq 0}$ augmented by all the $\mathbb{P}$-null sets in $\mathcal{F}$ is well-defined. Let $\{\bm{X}(\tau)\}_{\tau\in[s,T]}$ be the controlled $d$-dimensional state process driven by the forward SDE (FSDE): 
\begin{equation} \label{State equation}
	\left\{
	\begin{array}{rcl}
		d\bm{X}(\tau) & = & b\big(\tau,\bm{X}(\tau),\bm{\alpha}(\tau)\big)d\tau+\sigma\big(\tau,\bm{X}(\tau),\bm{\alpha}(\tau)\big)d\bm{W}(\tau), \quad \tau\in[s,T], \\
		\bm{X}(s) & = & \bm{\xi}, \quad \bm{\xi}\in\mathcal{L}^2_{\mathcal{F}_s}(\Omega;\mathbb{R}^d),
	\end{array}
	\right.
\end{equation}
where $\mathcal{L}^2_{\mathcal{F}_s}(\Omega;\mathbb{R}^d)$ is the set of $\mathbb{R}^d$-valued, $\mathcal{F}_s$-measurable, and square-integrable random variables
and $\bm{\alpha}(\cdot):[s,T]\times\Omega\to U$ with $U\subseteq\mathbb{R}^p$ is the aggregated control process that consists of all $m$ players' controls characterized by $\{\bm{\alpha}^a\}_{a=1}^m$, i.e. $\bm{\alpha}=((\bm{\alpha}^1)^\top,\ldots,(\bm{\alpha}^m)^\top)^\top$, with $\bm{\alpha}^a(\cdot):[s,T]\times\Omega\to U^a\subseteq\mathbb{R}^{p^a}$ and $\sum_{a=1}^mp^a=p$. Here, $k,d,m,p^a$ are arbitrary positive integers. Hereafter, we follow the notations in \cite{Han2021} for the control policies that the left subscript and superscript denote the time bounds of truncated control policies, i.e. ${}_{lb}^{ub}\bm{\alpha}=\{\bm{\alpha}_\tau\}_{\tau\in [lb,ub]}$ (they are suppressed when they are 0 and $T$, respectively), while the right subscript indicates the control at specific time point. Denote by $\bm{\alpha}^{-a}$ the aggregated controls except for $\bm{\alpha}^a$ such that $\bm{\alpha}$ consists of $\bm{\alpha}^a$ and $\bm{\alpha}^{-a}$ for any $a$, denoted by $\bm{\alpha}=\bm{\alpha}^a \oplus \bm{\alpha}^{-a}$. {\color{black}Moreover, it is useful to introduce $\mathbb{X}^{s,\bm{\xi},\bm{\alpha}}_\tau$ (or $\mathbb{X}_\tau$ for short) the set of reachable states at time $\tau$ from the time-state $(s,\bm{\xi})$ with the strategy $\bm{\alpha}$, which is defined by 
	$$\mathbb{X}^{s,\bm{\xi},\bm{\alpha}}_\tau:=\mathrm{Int}\overline{\mathbb{X}}^{s,\bm{\xi},\bm{\alpha}}_\tau\cup\big\{y\in\partial\overline{\mathbb{X}}^{s,\bm{\xi},\bm{\alpha}}_\tau:\mathbb{P}\big(\bm{X}(\tau)\in\partial\overline{\mathbb{X}}^{s,\bm{\xi},\bm{\alpha}}_\tau\cap B(y,\delta)\big)>0~\forall\delta>0\big\},$$
	where $\overline{\mathbb{X}}^{s,\bm{\xi},\bm{\alpha}}_\tau$ is the support of the distribution of $X(\tau)$ of \eqref{State equation}, the interior and the boundary of which in $\mathbb{R}^d$ are denoted by $\mathrm{Int}\overline{\mathbb{X}}^{s,\bm{\xi},\bm{\alpha}}_\tau$ and $\partial\overline{\mathbb{X}}^{s,\bm{\xi},\bm{\alpha}}_\tau$, respectively, and $B(y,\delta)$ denotes the ball centered at $y$ with radius $\delta$. We refer the readers to Section 3 of \cite{He2021} for more details. Next, let} $\{(\bm{Y}(\tau),\bm{Z}(\tau))\}_{\tau\in[s,T]}\equiv\{(\bm{Y}(\tau;s,\bm{\xi},{}_s\bm{\alpha}),\bm{Z}(\tau;s,\bm{\xi},{}_s\bm{\alpha}))\}_{\tau\in[s,T]}$ be the adapted solution (see \cite[Proposition 3.3]{Ma1999} for the solvability) to the following backward SDE (BSDE): 
\begin{equation} \label{Backward SDE} 
	\left\{
	\begin{array}{rcl}
		d\bm{Y}(\tau) & = & -\bm{h}\big(s,\tau,\bm{X}(\tau),\bm{\alpha}(\tau),\bm{Y}(\tau),\bm{Z}(\tau)\big)d\tau+\bm{Z}(\tau)d\bm{W}(\tau), \quad \tau\in[s,T], \\
		\bm{Y}(T) & = & \bm{g}\big(s,\bm{X}(T)\big),
	\end{array}
	\right.
\end{equation}
where $\bm{X}$ satisfies \eqref{State equation} and for $\Psi=\bm{Y},\bm{h}$, or $\bm{g}$, $\Psi=(\Psi^1,\ldots,\Psi^m)^\top$ and $\bm{Z}$ is $\mathbb{R}^{m\times k}$-valued. Equations \eqref{State equation} and \eqref{Backward SDE} jointly form forward-backward SDEs (FBSDEs). 

We presume that each Player $a$ ($a=1,\ldots,m$) aims to choose her control $\bm{\alpha}^a$ to minimize the following cost functional:
\begin{equation} \label{Cost functional}
	\bm{J}^a\left(s,\bm{\xi};{}_s\bm{\alpha}^a \oplus {}_s\bm{\alpha}^{-a}\right):=\bm{Y}^a(s;s,\bm{\xi},{}_s\bm{\alpha}).
\end{equation}
With the similar arguments in \cite{Karoui1997}, it turns out that under some mild conditions, the cost functional $\bm{J}^a$ may be expressed as
\begin{equation*} 
	\bm{J}^a\left(s,\bm{\xi};{}_s\bm{\alpha}^a \oplus {}_s\bm{\alpha}^{-a}\right)=\mathbb{E}\left[\left.\int^T_s \bm{h}^a\big(s,\tau,\bm{X}(\tau),\bm{\alpha}(\tau),\bm{Y}(\tau),\bm{Z}(\tau)\big)d\tau
	+\bm{g}^a\big(s,\bm{X}(T)\big)\right|\mathcal{F}_s\right].
\end{equation*}
Note that when $m=1$, the problem is reduced to the TIC problem with recursive cost functional considered in \cite{Wei2017,Yan2019}. Moreover, if both $\bm{h}$ and $\bm{g}$ are independent of the initial time $s$, then it is further reduced to a TC problem with a recursive utility, considered in \cite{Karoui2001}. Furthermore, when $\bm{h}$ depends on neither $s$ nor $(\bm{Y}(\cdot),\bm{Z}(\cdot))$, the cost functional reduces to the classical one; see \cite{Yong1999}.

A typical example of such initial-time-dependent cost functionals adopts non-exponential or hyperbolic discounting factors; see \cite{Laibson1997,Frederick2002}. For illustration, we assume a Markovian framework and that all the coefficient and objective functions, $b:[0,T]\times\mathbb{R}^d\times U\to\mathbb{R}^d$, $\sigma:[0,T]\times\mathbb{R}^d\times U\to\mathbb{R}^{d\times k}$, $\bm{h}^a:\nabla[0,T]\times\mathbb{R}^d\times U\times\mathbb{R}^m\times(\mathbb{R}^d)^m\to\mathbb{R}$ and $\bm{g}^a:[0,T]\times\mathbb{R}^d\to\mathbb{R}$ are deterministic, where $\nabla[0,T]:=\{(\tau_1,\tau_2)\in[0,T]^2:~0\le \tau_1\le \tau_2 \le T\}$. Moreover, we define the set of all admissible control processes on $[s,T]$ as follows:
\begin{equation*}
	{}^{s^\prime}_s\bm{\mathcal{A}}=\left\{\bm{\alpha}:[s,s^\prime]\times\Omega\to U:\bm{\alpha}(\cdot) \text{~is~}\mathbb{F}\text{-progressively measurable with~} \mathbb{E}\int^{s^\prime}_s|\bm{\alpha}(\tau)|^2d\tau<\infty\right\}.
\end{equation*}
Similarly, we define the admissible set ${}^T_s\bm{\mathcal{A}}^a$ for each player $a$ by replacing $U$ with $U^a$. Under some mild conditions (see \cite[Proposition 3.3]{Ma1999}), for any $(s,\bm{\xi})\in[0,T]\times\mathcal{L}^2_{\mathcal{F}_s}(\Omega;\mathbb{R}^d)$ and ${}^T_s\bm{\alpha}\in{}^T_s\bm{\mathcal{A}}$, the controlled FBSDEs \eqref{State equation}-\eqref{Backward SDE} admit a unique $\mathbb{F}$-adapted solution $\{\bm{X}(\tau),\bm{Y}(\tau),\bm{Z}(\tau)\}_{\tau\in[s,T]}$.  

The $m$-player game is formed, attributed to the common state processes and the recursion of the cost functionals on the aggregated $(\bm{Y}(\tau),\bm{Z}(\tau))$. Each player wants to minimize her own cost functional, naturally resulting in a Nash equilibrium (NE) point. However, since the cost functions $\bm{h}^a$ and $\bm{g}^a$ in \eqref{Cost functional} are dependent on the initial time $s$, we will observe TIC of the decision-making. In other words, the NE point found at time $t$ may not be the NE point when we evaluate again the SDG \eqref{State equation} with \eqref{Cost functional} at time $s>t$. To deal with the TIC, we introduce the concept of \textit{time-consistent NE (TC-NE)} point below, in line with the initiative of \cite{Strotz1955}.

\subsection{Time-Consistent Nash Equilibrium Point} \label{sec:TCNE}
Heuristically, we are treating the TIC SDGs as ``games in subgames" while the similar concept is first proposed in \cite{Pun2018} for robust TIC stochastic controls, where the problem is recast as a (two-player) nonzero-sum TIC SDG played by the agent and the nature. A TC-NE point of TIC SDG \eqref{State equation}-\eqref{Backward SDE} with \eqref{Cost functional} finds the NE point over $[s,T]$ given that the players adopt the predetermined NE points over $[s+\epsilon,T]$ for a small time elapse $\epsilon>0$ and $s\in[0,T]$. In light of this search, the NE points identified backwardly are subgame perfect equilibrium (SPE). Note that the SPE concept is concerned about the (aggregated) controls across time and it implies the so-called time consistency of the NE points. We give the formal definition of TC-NE point as follows.

\begin{definition}[Time-Consistent Nash Equilibrium (TC-NE) Point] \label{def:TCNEpt}
	Let $U$ be a non-empty set of $\mathbb{R}^p$ and $U^a\subseteq\mathbb{R}^{p^a}$ for $a=1,\ldots,m$ be the control set for the player $a$. A continuous map $\overline{\bm{\alpha}}:[0,T]\times\mathbb{R}^d\to U$ is called a closed-loop TC-NE point of the nonzero-sum TIC SDG \eqref{State equation}-\eqref{Backward SDE} with \eqref{Cost functional} if the following two conditions hold: 
	\begin{enumerate}
		\item For any {\color{black}$(t,y)\in[0,T]\times\mathbb{R}^d$}, the state equation 
		\begin{equation*} 
			\left\{
			\begin{array}{rcl}
				d\overline{\bm{X}}{\color{black}^{t,y}}(\tau) & = & b\big(\tau,\overline{\bm{X}}{\color{black}^{t,y}}(\tau),\overline{\bm{\alpha}}(\tau,\overline{\bm{X}}{\color{black}^{t,y}}(\tau))\big)d\tau+\sigma\big(\tau,\overline{\bm{X}}{\color{black}^{t,y}}(\tau),\overline{\bm{\alpha}}(\tau,\overline{\bm{X}}{\color{black}^{t,y}}(\tau))\big)d\bm{W}(\tau), \quad \tau\in[{\color{black}t},T], \\
				\overline{\bm{X}}{\color{black}^{t,y}}({\color{black}t}) & = & y, \quad y\in\mathbb{R}^d,
			\end{array}
			\right.
		\end{equation*}
		admits a unique solution $\{\overline{\bm{X}}{\color{black}^{t,y}}(\tau)\}_{\tau\in[{\color{black}t},T]}$; 
		\item For any $\big(a,s,{\color{black}\bm{\alpha}^a,x}\big)\in\{1,\ldots,m\}\times[{\color{black}t},T)\times {\color{black}U^a\times\mathbb{X}^{t,y,\{\overline{\alpha}(\tau,\overline{\bm{X}}^{t,y}(\tau))\}_{\tau\in[t,s]}}_s}$, let {\color{black}$\{\widetilde{\bm{X}}^{s,x}(\tau)\}_{\tau\in[s,T]}$} solves 
		\begin{equation*}
			\left\{
			\begin{array}{rcll}
				d{\color{black}\widetilde{\bm{X}}^{s,x}(\tau)} & = & b\big(\tau,{\color{black}\widetilde{\bm{X}}^{s,x}(\tau)},{\color{black}\bm{\alpha}^a}\oplus\overline{\bm{\alpha}}^{-a}(\tau,{\color{black}\widetilde{\bm{X}}^{s,x}(\tau)})\big)d\tau+\sigma\big(\tau,{\color{black}\widetilde{\bm{X}}^{s,x}(\tau)},{\color{black}\bm{\alpha}^a}\oplus\overline{\bm{\alpha}}^{-a}(\tau,{\color{black}\widetilde{\bm{X}}^{s,x}(\tau)})\big)d\bm{W}(\tau), & \tau\in[s,s+\epsilon), \\
				d{\color{black}\widetilde{\bm{X}}^{s,x}(\tau)} & = & b\big(\tau,{\color{black}\widetilde{\bm{X}}^{s,x}(\tau)},\overline{\bm{\alpha}}(\tau,{\color{black}\widetilde{\bm{X}}^{s,x}(\tau)})\big)d\tau+\sigma\big(\tau,{\color{black}\widetilde{\bm{X}}^{s,x}(\tau)},\overline{\bm{\alpha}}(\tau,{\color{black}\widetilde{\bm{X}}^{s,x}(\tau)})\big)d\bm{W}(\tau), & \tau\in[s+\epsilon,T], \\
				{\color{black}\widetilde{\bm{X}}^{s,x}(s)} & = & {\color{black}x}, &
			\end{array}
			\right. 
		\end{equation*} 
		then the following inequality holds: 
		\begin{equation} \label{Local optimality} 
			\underset{\epsilon\downarrow 0}{\underline{\lim}}\frac{\bm{J}^a\Big(s,{\color{black}x};{}_s\widetilde{\bm{\alpha}}^{a,\epsilon{\color{black},\bm{\alpha}^a}}\oplus\{\overline{\bm{\alpha}}^{-a}(\tau,{\color{black}\widetilde{\bm{X}}^{s,x}(\tau)})\}_{\tau\in[s,T]}\Big)-\bm{J}^a\Big(s,{\color{black}x};\{\overline{\bm{\alpha}}(\tau,\overline{\bm{X}}{\color{black}^{s,x}}(\tau))\}_{\tau\in[s,T]}\Big)}{\epsilon}\geq 0, 
		\end{equation} 
		where  
		\begin{equation} \label{PiecewiseStra}
			{}_s\widetilde{\bm{\alpha}}^{a,\epsilon{\color{black},\bm{\alpha}^a}}(\tau):={\color{black}\bm{\alpha}^a}\cdot\bm{1}_{[s,s+\epsilon)}(\tau)\otimes\overline{\bm{\alpha}}^a\big(\tau,{\color{black}\widetilde{\bm{X}}^{s,x}(\tau)}\big)\cdot\bm{1}_{[s+\epsilon,T]}(\tau) = \left\{
			\begin{array}{lr}
				{\color{black}\bm{\alpha}^a}, \hfill \tau\in[s,s+\epsilon), \\
				\overline{\bm{\alpha}}^a\big(\tau,{\color{black}\widetilde{\bm{X}}^{s,x}(\tau)}\big), \quad \tau\in[s+\epsilon,T]. 
			\end{array}
			\right. 
		\end{equation} 
	\end{enumerate} 
	Furthermore, $\{\overline{\bm{X}}{\color{black}^{t,y}}(\tau)\}_{\tau\in[{\color{black}t},T]}$ and $\bm{V}^a(t,y)\equiv \bm{J}^a\big(t,y;\{\overline{\bm{\alpha}}(\tau,\overline{\bm{X}}{\color{black}^{t,y}}(\tau))\}_{\tau\in[{\color{black}t},T]}\big)$ for $a=1,\ldots,m$ and $t\in[0,T]$ are called the TC-NE state process and the TC-NE value functions, respectively.  
\end{definition}

\begin{remark} \label{VariousES}
	{\color{black}For the local optimality condition \eqref{Local optimality} and the piecewise-defined strategy \eqref{PiecewiseStra} in Definition \ref{def:TCNEpt}, \cite{He2021} conducts in-depth studies on the choice of reference points $x$ and perturbations $\bm{\alpha}^a$ in a small time period of the length $\epsilon$. It summarizes a variety of similar but different concepts of closed-loop equilibrium strategies in the existing literature (see, for instance, \cite{Ekeland2006,Ekeland2010,Basak2010,Bjoerk2010,Bjoerk2014a,Dai2021} where the perturbations of \eqref{PiecewiseStra} are chosen from a set of all constant strategies; \cite{Ekeland2012,Ekeland2008,Bjoerk2017} where the perturbed strategies of \eqref{PiecewiseStra} are constructed by pasting two feasible deterministic feedback strategies). One main result of \cite{He2021} is to show the equilibrium strategy is independent of whether the alternative strategies are constant or deterministic strategies. In other words, $\bm{\alpha}^a$ can be taken as a deterministic feedback strategy $\bm{\alpha}^a=\bm{\alpha}^a(\cdot,\cdot)$, which would facilitate later analyses in Subsection \ref{sec:derHJB}. Another key contribution of \cite{He2021} is to elaborate the set of $x$ and to show the advantage of replacing the whole set $\mathbb{R}^d$ with the set $\mathbb{X}_s$ of reachable states in \eqref{Local optimality}. We assume that $\mathbb{X}_s=\mathbb{R}^d$ for all $s$ throughout our paper, except for Subsection \ref{Sec: PowerU} where we consider the power-utility model with $\mathbb{X}_s=(0,\infty)$. In addition to the closed-loop strategies, the existing literature also define so-called open-loop equilibrium policies (see, \cite{Hu2012,Hu2017,Yan2019}).         
		In this paper, we handle TIC problem by the means of closed-loop strategies within a game-theoretical framework.} When $m=1$, the TC-NE point of the nonzero-sum TIC SDG is reduced to the SPE of the corresponding TIC control problem; see \cite{Wei2017,Yong2012,Yan2019}. When the TIC sources are eliminated, it is clear that the local optimality described by \eqref{Local optimality} agrees with the conventional dynamic optimality.  
\end{remark}

The inequality \eqref{Local optimality} implies that each player is locally optimal in minimizing the cost functional $\bm{J}^a$ over $[s,s+\epsilon)$ in a proper sense and no player can do better by unilaterally changing their strategy. The basic idea is illustrated in Figure \ref{fig:equilibrium strategy}, which also clarifies the notations we used.
\begin{figure}[!ht]
	\centering
	\includegraphics[width=0.4\textwidth]{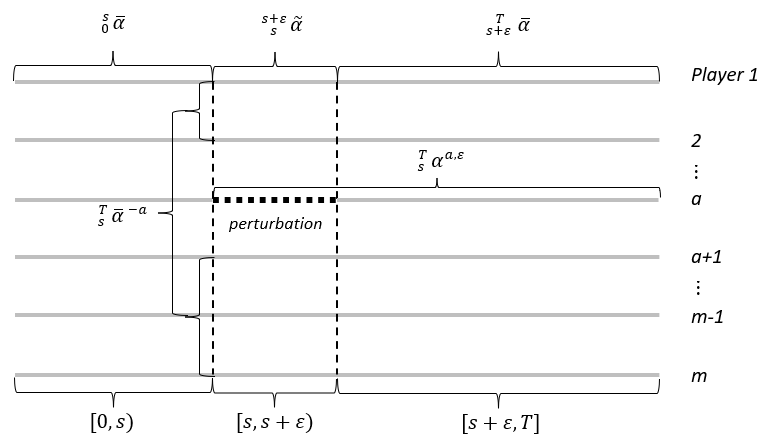}
	\caption{Time-Consistent Nash Equilibrium (TC-NE) Point} 
	\label{fig:equilibrium strategy}
\end{figure}


In the next subsection, we shall characterize the TC-NE point as well as the TC-NE value function with a differential equation approach. While a single equilibrium Hamilton--Jacobi--Bellman (HJB) equation is used to characterize the SPE of the TIC control problem in \cite{Bjoerk2017,Wei2017}, it can be imagined that a system of equilibrium HJB equations is needed for our case. Prior to its derivation,
we first introduce some notations and make an assumption as with \cite{Friedman1972,Friedman1976,Bensoussan2000}.

For $(t,s,y,\alpha^a\oplus \alpha^{-a},u,p,q^a)\in \nabla[0,T]\times\mathbb{R}^d\times U\times\mathbb{R}^m\times(\mathbb{R}^d)^m\times\mathbb{S}^{d}$, $a=1,\ldots,m$, where $\mathbb{S}^d$ is the set of all $d\times d$ symmetric matrices, we denote the Hamiltonian by a $\mathbb{R}^m$-valued function $\bm{\mathcal{H}}(t,s,y,\alpha,u,p,q)=(\bm{\mathcal{H}}^1(t,s,y,\alpha^1,\alpha^{-1},u,p,q^1),\cdots,\bm{\mathcal{H}}^m(t,s,y,\alpha^m,\alpha^{-m},u,p,q^m))^\top$ with $q=\{q^a\}_{a=1}^m$ and $p=\{p^a\}_{a=1}^m$, and $\bm{\mathcal{H}}^a$ defined by 
\begin{equation} \label{Hamiltonian} 
	\begin{split}
		\bm{\mathcal{H}}^a(t,s,y,\alpha^a,\alpha^{-a},u,p,q^a) & =\frac{1}{2}\mathrm{tr}\left[q^a\cdot (\sigma\sigma^\top)\big(s,y,\alpha^a\oplus\alpha^{-a}\big)\right]+(p^a)^\top b\big(s,y,\alpha^a\oplus\alpha^{-a}\big) \\
		& \qquad +\bm{h}^a\big(t,s,y,\alpha^a\oplus\alpha^{-a},u,\{(p^a)^\top\sigma\big(s,y,\alpha^a\oplus\alpha^{-a}\big)\}^m_{a=1}\big).
	\end{split}
\end{equation} 
\begin{assumption}[Generalized minimax condition] \label{assumption} 
	There exist functions $\bm{\phi}^a(t,s,y,u,p,q):~\nabla[0,T]\times\mathbb{R}^d\times\mathbb{R}^m\times(\mathbb{R}^d)^m\times(\mathbb{S}^d)^m\to\mathbb{R}^{p^a}$ for $a=1,\ldots,m$ with needed regularity such that 
	\begin{enumerate}
		\item for any $(t,s,y,u,p,q)\in \nabla[0,T]\times\mathbb{R}^d\times\mathbb{R}^m\times(\mathbb{R}^d)^m\times(\mathbb{S}^d)^m$, $\bm{\phi}^a(t,s,y,u,p,q)\in U^a$; 
		\item for any $(t,s,y,u,p,q)\in \nabla[0,T]\times\mathbb{R}^d\times\mathbb{R}^m\times(\mathbb{R}^d)^m\times(\mathbb{S}^d)^m$, 
		\begin{equation*}
			\min\limits_{\alpha^a\in U^a}\bm{\mathcal{H}}^a\big(t,s,y,\alpha^a,\bm{\phi}^{-a}(t,s,y,u,p,q),u,p,q^a\big)= \bm{\mathcal{H}}^a\big(t,s,y,\bm{\phi}^a(t,s,y,u,p,q),\bm{\phi}^{-a}(t,s,y,u,p,q),u,p,q^a\big).
		\end{equation*}  
	\end{enumerate}
\end{assumption} 
\noindent This generalized minimax condition has implied the existence of the NE point at each $(t,s)\in \nabla[0,T]$ in the sense that all $\phi^a$ are found simultaneously. It is desirable as we are discussing about a general setting and it is normally equivalent to model assumptions on $b,\sigma,$ and ${\bm h}^a$. 

\subsection{Heuristic Derivation of Equilibrium HJB System} \label{sec:derHJB}
In this subsection, we derive the system of equilibrium HJB equations, characterizing the TC-NE point in Definition \ref{def:TCNEpt}, from which we reveal that
it is a special case of our nonlocal parabolic system \eqref{Nonlocal fully nonlinear system}. Since the focus of our paper is on the well-posedness of \eqref{Nonlocal fully nonlinear system} and the nested HJB system rather than the latter's origination, a heuristic derivation in the similar fashion of \cite{Bjoerk2014,Bjoerk2017,Bjoerk2010} will be in place. For simplicity, we show only where the nonlocal terms $\left(\partial_I\bm{u}\right)_{|I|\leq 2}(s,s,y)$ come from and the linkage with the classical HJB equations. For a rigorous derivation, one can follow the discretization approach in \cite{Yong2012,Wei2017} {\color{black}or a rigorous argument in \cite{He2021}} to derive the nonlocal parabolic system but it is too lengthy and thus not adopted here. 

In light of the methodology in \cite{Bjoerk2014,Bjoerk2017,Bjoerk2010}, there are three main steps (\textit{Step 1}-\textit{Step 3}) to obtain the equilibrium HJB system of a multiplayer nonzero-sum TIC SDGs. {\color{black}For the sake of simplification of the heuristic derivation, we assume that $\bm{\xi}=y\in\mathbb{R}^d$ and $\bm{h}=0$ in \eqref{State equation}-\eqref{Cost functional}, and adopt the deterministic feedback-type controls throughout this subsection. Next, let us} consider  
\begin{equation*}
	\bm{J}^a(s,y;{\color{black}\bm{\alpha}(\cdot,\cdot)})=\mathbb{E}_{s,y}\left[\bm{g}^a\left(s,\bm{X}(T;s,y,{\color{black}\bm{\alpha}(\cdot,\cdot)})\right)\right], \quad a=1,\ldots,m, 
\end{equation*}  
where $\mathbb{E}_{s,y}$ is the conditional expectation under $\bm{X}(s)=y$ and $\bm{X}(\cdot;s,y,{\color{black}\bm{\alpha}(\cdot,\cdot)})$ (or $\bm{X}^{\bm{\alpha}}(\cdot)$ for short) is the unique adapted solution to \eqref{State equation} on $[s,T]$ with {\color{black}$\bm{\alpha}(\cdot,\cdot):[0,T]\times\mathbb{R}^d\to U$} and $\bm{X}(s)=y$. {\color{black}The set of feasible feedback strategies is denoted by $\mathbb{U}$, which can be roughly understood as the class of deterministic functions that are regular enough to promise the well-posedness of the state process \eqref{State equation} and the cost functional \eqref{Cost functional}. We refer the readers to Definition 2.2 of \cite{Bjoerk2017} and Definition 2.1 of \cite{He2021} for more details.}

\begin{definition} \label{Definition of u}
	{\color{black}Given a feasible feedback strategy $\bm{\alpha}\in\mathbb{U}$}, we define $\bm{u}(t,s,y;\bm{\alpha}):\nabla[0,T]\times\mathbb{R}^d\to\mathbb{R}^m$ by
	\begin{equation} \label{eq:u}
		\bm{u}(t,s,y;\bm{\alpha})=\mathbb{E}_{s,y}\left[\bm{g}\left(t,\bm{X}^{\bm{\alpha}}(T)\right)\right], 
	\end{equation}
	i.e., its component $\bm{u}^a(t,s,y;\bm{\alpha})=\mathbb{E}_{s,y}\left[\bm{g}^a\left(t,\bm{X}^{\bm{\alpha}}(T)\right)\right]$ for $a=1,\ldots,m$. 
\end{definition}

For any $t\in[0,T]$ and {\color{black}$\bm{\alpha}\in\mathbb{U}$}, the process $\bm{u}^a(t,s,\bm{X}^{\bm{\alpha}}(s);\bm{\alpha})$ is a martingale and \eqref{eq:u} satisfies 
\begin{equation} \label{eq:uPDE}
	\left\{
	\begin{array}{rcl}
		\mathcal{A}^{\bm{\alpha}} \bm{u}^a(t,s,y;\bm{\alpha}) & = & 0, \qquad\qquad t\leq s\leq T, \\
		\bm{u}(t,T,y;\bm{\alpha}) & = & \bm{g}(t,y), \hfill y\in\mathbb{R}^d,  
	\end{array} 
	\right. 
\end{equation} 
where $\mathcal{A}^{\bm{\alpha}}$ is the controlled infinitesimal generator of the FSDE \eqref{State equation}: 
\begin{equation*}
	\mathcal{A}^{\bm{\alpha}}=\frac{\partial}{\partial s}+\frac{1}{2}\sum^d_{i,j=1}(\sigma\sigma^\top)_{ij}\big(s,y,\bm{\alpha}\big)\frac{\partial^2}{\partial y_i\partial y_j}+\sum^d_{i=1}b_i\big(s,y,\bm{\alpha}\big)\frac{\partial}{\partial y_i}. 
\end{equation*}

Similar to the classical dynamic programming principle in \cite{Yong1999}, we need to first derive a recursive relation between cost functionals/value functions evaluated at two different initial points $(s,y)$ and $(s+\epsilon,\bm{X}^{\bm{\alpha}}(s+\epsilon))$. Then, by sending the mesh size of the time interval partition $\epsilon$ to zero, a nonlocal system of parabolic type is derived, through which a closed-loop TC-NE point can be identified and the TC-NE value function can be obtained.

\begin{description}
	\item[\textbf{\underline{\textit{Step 1: The Recursion for Cost functionals $\bm{J(s,y;\alpha)}$}}.}] From the Markovian structure and Definition \ref{Definition of u}, we have $\mathbb{E}_{s+\epsilon,\bm{X}^{\bm{\alpha}}_{s+\epsilon}}\left[\bm{g}^a(s+\epsilon,\bm{X}^{\bm{\alpha}}(T))\right]=\bm{u}^a\left(s+\epsilon,s+\epsilon,\bm{X}^{\bm{\alpha}}(s+\epsilon);\bm{\alpha}\right)$, which yields $\bm{J}^a(s+\epsilon,\bm{X}^{\bm{\alpha}}(s+\epsilon);\bm{\alpha})=\bm{u}^a\left(s+\epsilon,s+\epsilon,\bm{X}^{\bm{\alpha}}(s+\epsilon);\bm{\alpha}\right)$. Taking conditional expectation at $(s,y)$ on both sides of the latter equation, we have 
	\begin{equation*} \label{Recursion for J 1}
		\mathbb{E}_{s,y}\left[\bm{J}^a(s+\epsilon,\bm{X}^{\bm{\alpha}}(s+\epsilon);\bm{\alpha})\right]=\bm{J}^a(s,y;\bm\alpha)+\mathbb{E}_{s,y}\left[\bm{u}^a\left(s+\epsilon,s+\epsilon,\bm{X}^{\bm{\alpha}}(s+\epsilon);\bm{\alpha}\right)\right]-\mathbb{E}_{s,y}\left[\bm{g}^a\left(s,\bm{X}^{\bm{\alpha}}(T)\right)\right].
	\end{equation*}
	Moreover, by the tower rule of conditional expectations in the last term,
	we obtain the recursive equation for $\bm{J}^a(s,y;\bm{\alpha})$ as follows: 
	\begin{equation}
		\mathbb{E}_{s,y}\left[\bm{J}^a(s+\epsilon,\bm{X}^{\bm{\alpha}}(s+\epsilon);\bm{\alpha})\right]= \bm{J}^a(s,y;\bm\alpha)+\mathbb{E}_{s,y}\left[\bm{u}^a\left(s+\epsilon,s+\epsilon,\bm{X}^{\bm{\alpha}}(s+\epsilon);\bm{\alpha}\right)\right]-\mathbb{E}_{s,y}\left[\bm{u}^a\left(s,s+\epsilon,\bm{X}^{\bm{\alpha}}(s+\epsilon);\bm{\alpha}\right)\right]. \label{Recursion for J} 
	\end{equation}
	
	\item[\textbf{\underline{\textit{Step 2: The Recursion for TC-NE Value functions $\bm{V(s,y)}$}}.}] Based on \eqref{Recursion for J}, we aim to derive a recursive equation for $\bm{V}^a$. We first define {\color{black}a perturbed feedback strategy $\widetilde{\bm{\alpha}}^{a,s,\epsilon,\bm{\alpha}^a}(\tau,y)$ such that $\widetilde{\bm{\alpha}}^{a,s,\epsilon,\bm{\alpha}^a}(\tau,y):=\bm{\alpha}^a(\tau,y)$ for $\tau\in[s,s+\epsilon)$ and $\widetilde{\bm{\alpha}}^{a,s,\epsilon,\bm{\alpha}^a}(\tau,y):=\overline{\bm{\alpha}}^a(\tau,y)$ for $\tau\in[s+\epsilon,T]$, where $\bm{\alpha}^a$ is an arbitrary element in $\mathbb{U}^a$ that consists of the $a$-th component of feasible controls in $\mathbb{U}$ and $\overline{\bm{\alpha}}$ can be viewed as a candidate equilibrium strategy. Note that $\widetilde{\bm{\alpha}}^{a,s,\epsilon,\bm{\alpha}^a}(\tau,y)$ is a function rather than a process ${}_s\widetilde{\bm{\alpha}}^{a,\epsilon{\color{black},\bm{\alpha}^a}}$ in Definition \ref{def:TCNEpt} while they have similar roles. Noteworthy is that the perturbed strategy $\widetilde{\bm{\alpha}}^{a,s,\epsilon,\bm{\alpha}^a}(\tau,y)$ is constructed with two feedback strategies $\bm{\alpha}^a\in\mathbb{U}^a$ and $\overline{\bm{\alpha}}^a$ rather than by pasting a constant strategy $\bm{\alpha}^a\in U^a$ and a feedback one $\overline{\bm{\alpha}}^a$ (as in \eqref{PiecewiseStra}). However, Remark \ref{VariousES} illustrates that the slight difference does not affect our characterization of the equilibrium point and its associated HJB equations/systems.} 
	Then, Definition \ref{def:TCNEpt} implies that for $a=1,\ldots,m$,
	\begin{eqnarray}  
		\bm{J}^a(s+\epsilon,\bm{X}{\color{black}^{\widetilde{\bm{\alpha}}}}(s+\epsilon);{\color{black}\widetilde{\bm{\alpha}}^{a,s,\epsilon,\bm{\alpha}^a}\oplus\overline{\bm{\alpha}}^{-a}}) & = & \bm{V}^a(s+\epsilon,\bm{X}{\color{black}^{\widetilde{\bm{\alpha}}}}(s+\epsilon)), \label{V and u 1} \\
		\bm{u}^a(t,s+\epsilon,\bm{X}{\color{black}^{\widetilde{\bm{\alpha}}}}(s+\epsilon);{\color{black}\widetilde{\bm{\alpha}}^{a,s,\epsilon,\bm{\alpha}^a}\oplus\overline{\bm{\alpha}}^{-a}}) & = & \bm{u}^a(t,s+\epsilon,\bm{X}{\color{black}^{\widetilde{\bm{\alpha}}}}(s+\epsilon)), \label{V and u 2}  
	\end{eqnarray}
	where {\color{black}$\bm{X}^{\widetilde{\bm{\alpha}}}(\cdot)$ represents $\bm{X}(\cdot;s,y,\widetilde{\bm{\alpha}}^{a,s,\epsilon,\bm{\alpha}^a}\oplus\overline{\bm{\alpha}}^{-a}))$ and} 
	the function $\bm{u}(t,\cdot,\cdot)$ is defined by \eqref{eq:u} with $\bm{\alpha}$ replaced by $\overline{\bm{\alpha}}$. {\color{black}Next, inspired by the discrete setting of TIC stochastic control problem in \cite{Bjoerk2014,Bjoerk2010}, it is anticipated from \eqref{Local optimality} that 
		\begin{equation} \label{Sentence}
			\text{``}\bm{J}^a\big(s,y; {\color{black}\widetilde{\bm{\alpha}}^{a,s,\epsilon,\bm{\alpha}^a}\oplus\overline{\bm{\alpha}}^{-a}}\big)\geq \bm{V}^a(s,y)\text{~~for~~} \forall \bm{\alpha}^a\in \mathbb{U}^a \text{~~with the equality holds when~~} \bm{\alpha}^a(s,y)=\overline{\bm{\alpha}}^a(s,y).\text{''}
		\end{equation} 
		However, for a continuous-time model, this statement is not always true since it is still possible that $\bm{J}^a\big(s,y; {\color{black}\widetilde{\bm{\alpha}}^{a,s,\epsilon,\bm{\alpha}^a}\oplus\overline{\bm{\alpha}}^{-a}}\big)<\bm{V}^a(s,y)$ for sufficiently small $\epsilon$ and certain $\bm{\alpha}^a\in \mathbb{U}^a$; see \cite{Bjoerk2017,He2021}. Hence, the following analyses of \eqref{InformalJE} and \eqref{Pre-Recursion V} are rather heuristic, while they are included in our derivation of equilibrium HJB systems as they inspire us on how to investigate continuous-time TIC problems via the lens of discrete-time setting; see \cite{Bjoerk2017,Bjoerk2010} for the similar heuristic arguments.
		For a formal and rigorous proof, readers are suggested to refer to Theorem 3.3 of \cite{He2021} and Section 4 of \cite{Wei2017}. Note that no matter whether the argument is formal or not, one always obtains the same HJB equations.}

Next, we find that \eqref{Sentence} and \eqref{Recursion for J} indicate that 
\begin{equation} \label{InformalJE}
	\begin{split}
		&\inf\limits_{\bm{\alpha}^a~\in~\mathbb{U}^a}\Big\{\mathbb{E}_{s,y}\left[\bm{J}^a\left(s+\epsilon,\bm{X}^{\widetilde{\bm{\alpha}}}(s+\epsilon);{\color{black}\widetilde{\bm{\alpha}}^{a,s,\epsilon,\bm{\alpha}^a}\oplus\overline{\bm{\alpha}}^{-a}}\right)\right]-\bm{V}^a(s,y) \\
		& \quad\quad -\mathbb{E}_{s,y}\left[\bm{u}^a\left(s+\epsilon,s+\epsilon,\bm{X}^{\widetilde{\bm{\alpha}}}(s+\epsilon);{\color{black}\widetilde{\bm{\alpha}}^{a,s,\epsilon,\bm{\alpha}^a}\oplus\overline{\bm{\alpha}}^{-a}}\right)\right]+\mathbb{E}_{s,y}\left[\bm{u}^a\left(s,s+\epsilon,\bm{X}^{\widetilde{\bm{\alpha}}}(s+\epsilon);{\color{black}\widetilde{\bm{\alpha}}^{a,s,\epsilon,\bm{\alpha}^a}\oplus\overline{\bm{\alpha}}^{-a}}\right)\right]\Big\}=0, 
	\end{split}
\end{equation}
By the expressions of \eqref{V and u 1}-\eqref{V and u 2}, we obtain the following recursion for $V$: 
\begin{equation} \label{Pre-Recursion V}
	\begin{split}
		& \inf\limits_{\bm{\alpha}^a~\in~\mathbb{U}^a}\Big\{\mathbb{E}_{s,y}\left[\bm{V}^a(s+\epsilon,\bm{X}^{\widetilde{\alpha}}(s+\epsilon))\right]-\bm{V}^a(s,y) \\
		& \qquad \qquad \quad -\big(\mathbb{E}_{s,y}\left[\bm{u}^a\left(s+\epsilon,s+\epsilon,\bm{X}^{\widetilde{\bm{\alpha}}}(s+\epsilon)\right)\right]-\mathbb{E}_{s,y}\left[\bm{u}^a\left(s,s+\epsilon,\bm{X}^{\widetilde{\bm{\alpha}}}(s+\epsilon)\right)\right]\big)\Big\}=0,
	\end{split}
\end{equation}
the first line of which can be approximated by $\mathbb{E}_{s,y}\left[\bm{V}^a(s+\epsilon,\bm{X}^{\widetilde{\alpha}}(s+\epsilon))\right]-\bm{V}^a(s,y)\approx\bm{\mathcal{A}}^{\widetilde{\bm{\alpha}}} \bm{V}^a(s,y)\epsilon$ and the second line of which can be expressed as:
\begin{equation*}  
	\begin{split}
		&\mathbb{E}_{s,y}\left[\bm{u}^a\left(s+\epsilon,s+\epsilon,\bm{X}^{\widetilde{\bm{\alpha}}}(s+\epsilon)\right)\right]-\mathbb{E}_{s,y}\left[\bm{u}^a\left(s,s+\epsilon,\bm{X}^{\widetilde{\bm{\alpha}}}(s+\epsilon)\right)\right] \\
		=&~\mathbb{E}_{s,y}\left[\bm{u}^a\left(s+\epsilon,s+\epsilon,\bm{X}^{\widetilde{\bm{\alpha}}}(s+\epsilon)\right)\right]-\bm{u}^a(s,s,y)-\left(\mathbb{E}_{s,y}\left[\bm{u}^a\left(s,s+\epsilon,\bm{X}^{\widetilde{\bm{\alpha}}}(s+\epsilon)\right)\right]-\bm{u}^a(s,s,y)\right) \\
		\approx&~\big[\bm{\mathcal{A}}^{\widetilde{\bm{\alpha}}} \bm{u}^a(s,s,y)-\big(\bm{\mathcal{A}}^{\widetilde{\bm{\alpha}}} \bm{u}^a(t,s,y)\big)\big|_{t=s}\big]\epsilon
	\end{split}
\end{equation*}

\item[\textbf{\underline{\textit{Step 3: Equilibrium HJB System}}.}] Letting $\epsilon\to 0$ in \eqref{Pre-Recursion V} gives a deterministic system:
\begin{equation*}  \label{HJB system 2}
	\inf\limits_{\alpha^a\in U^a}\Big\{A^\alpha \bm{V}^a(s,y)-A^\alpha \bm{u}^a(s,s,y)+\left.\left(A^\alpha\bm{u}^a(t,s,y)\right)\right|_{t=s}\Big\} = 0, \quad {\color{black}a=1,2,\cdots,m,}  
\end{equation*}
with boundary conditions $\bm{V}(T,y)=\bm{g}(T,y)$, where 
$$
A^{\alpha}=\frac{\partial}{\partial s}+\frac{1}{2}\sum^d_{i,j=1}(\sigma\sigma^\top)_{ij}\big(s,y,\alpha\big)\frac{\partial^2}{\partial y_i\partial y_j}+\sum^d_{i=1}b_i\big(s,y,\alpha\big)\frac{\partial}{\partial y_i}
$$
for any $\alpha\in U$. Note that the key difference between the operators $\mathcal{A}^{\bm \alpha}$ and $A^{\alpha}$ is that the former corresponds to a {\color{black}function ${\bm \alpha}\in\mathbb{U}$} while the latter corresponds to a point $\alpha\in U$. By the generalized minimax condition in Assumption \ref{assumption} and noting that $\bm{V}^a(s,\cdot)=\bm{u}^a(s,s,\cdot)$, we know that the infimum above is achievable and the minimum is expressed by
\begin{equation} \label{Equilibrium strategy} 
	\alpha^{*a}(s,y)=\bm{\phi}^a(s,s,y,\bm{u}(s,s,y),\bm{u}_y(s,s,y),\bm{u}_{yy}(s,s,y)), \quad a=1,\ldots,m.
\end{equation} 
By the earlier discussion in Step 2, we must have $\overline{\bm \alpha}^a(\cdot,\cdot)=\alpha^{*a}(\cdot,\cdot)$ to form a closed-loop TC-NE point. With the representation of $\overline{\bm \alpha}$, we can then solve for $\bm{u}^a(\cdot,\cdot,\cdot)$ from \eqref{eq:uPDE} with ${\bm \alpha}$ replaced by $\overline{\bm \alpha}$, which is the equilibrium HJB system we look for, but we present only the system for the general case below to save space. It is noteworthy that even we can focus on solving $\bm{u}$ hereafter, the derivation uses the definition of ${\bm V}$ and thus we require $\bm{u}(t,s,y)$ to be first-order differentiable in $t$. 

\item[\underline{General Case}.]
{\color{black}Even if the running cost functional ${\bm h}$ is non-zero} and ${\bm \xi}$ is a random variable, the heuristic derivation above is almost identical except for more tedious expressions, i.e., we can obtain a generalized HJB system:
\begin{equation} \label{HJB system}
	\inf\limits_{\alpha^a\in U^a}\Big\{A^\alpha \bm{V}^a(s,y)-A^\alpha \bm{u}^a(s,s,y)+\left.\left(A^\alpha\bm{u}^a(t,s,y)\right)\right|_{t=s}+\bm{h}^a\big(s,s,y,\alpha,\bm{u}(s,s,y),\bm{u}^\top_y(s,s,y)\sigma(s,y,\alpha)\big)\Big\} = 0
\end{equation}
and the same expression of $\overline{\bm \alpha}^a$ in \eqref{Equilibrium strategy}. Plugging \eqref{Equilibrium strategy} into a modified version of \eqref{eq:uPDE} (with $\bm{h}$ as a non-homogeneous term), we obtain
\begin{equation*} 
	\left\{
	\begin{array}{lr}
		\bm{u}^a_s(t,s,y) +  \frac{1}{2}\mathrm{tr}\Big[(\sigma\sigma^\top)\big(s,y,\bm{\phi}(s,s,y,\bm{u}(s,s,y),\bm{u}_y(s,s,y),\bm{u}_{yy}(s,s,y))\big)\bm{u}^a_{yy}(t,s,y)\Big] \\
		\qquad
		+  \Big\langle b\big(s,y,\bm{\phi}(s,s,y,\bm{u}(s,s,y),\bm{u}_y(s,s,y),\bm{u}_{yy}(s,s,y))\big),\bm{u}^a_y\Big\rangle \\
		\qquad 
		+  \bm{h}^a\big(t,s,y,\bm{\phi}(s,s,y,\bm{u}(s,s,y),\bm{u}_y(s,s,y),\bm{u}_{yy}(s,s,y)), \\
		\qquad\qquad   
		\bm{u}(t,s,y),\bm{u}_y(t,s,y)^\top\sigma\big(s,y,\bm{\phi}(s,s,y,\bm{u}(s,s,y),\bm{u}_y(s,s,y),\bm{u}_{yy}(s,s,y))\big)\big) =  0, \\
		\bm{u}(t,T,y)  =  \bm{g}(t,y), \quad 0\leq t\leq s\leq T, \quad y\in\mathbb{R}^d, \quad a=1,\ldots,m.   
	\end{array}
	\right. 
\end{equation*}

Similar to the notations $\bm{\mathcal{H}}$ and $\bm{\mathcal{H}}^a$, for $(t,s,y,u,p,q,l,m,n)\in \nabla[0,T]\times\mathbb{R}^d\times\mathbb{R}^m\times(\mathbb{R}^d)^m\times(\mathbb{S}^{d})^m\times\mathbb{R}^m\times(\mathbb{R}^d)^m\times(\mathbb{S}^{d})^m$, we denote the nonlinearity $\bm{H}$ by a $\mathbb{R}^m$-valued function $\bm{H}(t,s,y,u,p,q,l,m,n)=(\bm{H}^1(t,s,y,u,p,q,l,m,n),\cdots,\bm{H}^m(t,s,y,u,p,q,l,m,n))^\top$ with $\bm{H}^a(t,s,y,u,p,q,l,m,n)=\bm{\mathcal{H}}^a(t,s,y,\phi(s,s,y,l,m,n),u,p,q^a)$ for $a=1,2,\cdots,m$. Then, we can simplify the system above as a nonlocal parabolic system for ${\bm u}$ of the form 
\begin{equation} \label{Equilibrium HJB system}  
	\left\{
	\begin{array}{l}
		\bm{u}_s(t,s,y)+\bm{H}\big(t,s,y,\bm{u}(t,s,y),\bm{u}_y(t,s,y),\bm{u}_{yy}(t,s,y),\bm{u}(s,s,y),\bm{u}_y(s,s,y),\bm{u}_{yy}(s,s,y)\big) =  0, \\
		\bm{u}(t,T,y)=\bm{g}(t,y), \quad 0\leq t\leq s\leq T, \quad y\in\mathbb{R}^d.   
	\end{array}
	\right. 
\end{equation}
\end{description}

\begin{remark}
Though the derivations above are heuristic, we can similarly develop the system-version results of \cite{Yong2012,Wei2017,Yan2019}, following their discretization approach {\color{black}or the argument in Theorem 3.3 of \cite{He2021}}, and we will also end up with the equilibrium HJB equation/system \eqref{Equilibrium HJB system} as well as a closed-loop TC-NE point \eqref{Equilibrium strategy} and the associated TC-NE value functions $\bm{V}(s,y):=\bm{u}(s,s,y)$ in the sense of Definition \ref{def:TCNEpt}. The relationship between \eqref{HJB system} and \eqref{Equilibrium HJB system} is discussed in \cite{Hernandez2021}. {\color{black}While it is not the focus of this paper, we summarize the mathematical claims/conjectures about the connection between solutions to \eqref{Equilibrium HJB system} and TIC SDGs in Section \ref{sec:relation}. Moreover, an interesting fact is that we only need to study the HJB system \eqref{Equilibrium HJB system} in the set of reachable states rather than in $\mathbb{R}^d$; see Section 2.3 of \cite{He2021} and the example in Subsection \ref{Sec: PowerU}. }
\end{remark}

From the derivations in this subsection, the inclusion of the nonlocal terms $\left(\partial_I\bm{u}\right)_{|I|\leq 2}(s,s,y)$ is rationalized: the characterizations of $\alpha^{*}$ (or $\overline{\bm \alpha}$) and ${\bm u}$ are coupled. It is also easy to see that in the TC case (independent of the initial time point), i.e., $\bm{u}(t,s,y)=\bm{u}(s,y)=\bm{V}(s,y)$, the HJB system \eqref{HJB system} or equilibrium HJB system \eqref{Equilibrium HJB system} reduce to the classicial ones (see \cite{Yong1999}). Proposition \ref{Equivalence between systems} below tells us that \eqref{Equilibrium HJB system} is a special case of \eqref{Nonlocal fully nonlinear system} with $r=1$. Moreover, a closely related topic is robust TIC stochastic controls via the formulation of nonzero-sum TIC SDGs. From which, we can observe many solvable examples in finance and insurance; see \cite{Pun2018,Han2021,Han2022,Yan2020,Lei2020}. By carefully choosing cost functionals of nonzero-sum SDGs, we can model the relative performance concerns among multiple agents in decision-making and thus our theory can extend the related works of \cite{Espinosa2015,Pun2016,Pun2016a,Lacker2019} by introducing TIC or behavioral factors.

\subsection{The Relation between the Equilibrium HJB System and TIC SDGs} \label{sec:relation}
The derivation of the equilibrium HJB system alone does not justify its mathematical connection with the stochastic control/game problem \eqref{State equation}-\eqref{Cost functional}. We need to show two aspects of the connection, namely \textbf{sufficiency} and \textbf{necessity}, constituting two conjectures below:
\begin{description}
\item[(\textbf{Sufficiency/Verification theorem}):] \textit{the solutions to \eqref{Equilibrium strategy} and \eqref{Equilibrium HJB system} indeed give a TC-NE point and a TC-NE value function}. Mathematically, we assume that $\bm{u}(t,s,y)\in C^{1,1,2}$ solves \eqref{Equilibrium HJB system} and that the infimum of \eqref{HJB system} is attained for every $(s,y)$ in the sense of NE. Then, the minimizer of $\bm{\mathcal{H}}$ (under Assumption \ref{assumption}),
given in \eqref{Equilibrium strategy}, 
is a closed-loop TC-NE point and the function $\bm{V}(s,y):=\bm{u}(s,s,y)$ is the TC-NE value function as in Definition \ref{def:TCNEpt};
\item[(\textbf{Necessity}):] \textit{every TC-NE point must minimize the Hamiltonian associated to TIC problem \eqref{State equation}-\eqref{Cost functional} and the corresponding value function solves the HJB system \eqref{HJB system}}. Mathematically, we assume that there exist a closed-loop TC-NE point $\overline{\bm{\alpha}}(s,y)$ and the corresponding value function $\bm{V}(s,y)\in C^{1,2}$ and we define $\bm{u}(t,s,y):=\overline{\bm{Y}}^t(s;s,y,\overline{\bm{\alpha}})$, where $\overline{\bm{Y}}^t$ comes from $\{(\overline{\bm{X}}(\tau;s,y,\overline{\bm{\alpha}}),\overline{\bm{Y}}^t(\tau;s,y,\overline{\bm{\alpha}}),\overline{\bm{Z}}^t(\tau;s,y,\overline{\bm{\alpha}})\}_{\tau\in[s,T]}$ being the adapted solution of the family of FBSDEs parameterized by $t$:
\begin{equation*} 
	\left\{
	\begin{array}{rcl}
		d\overline{\bm{X}}(\tau) & = & b\big(\tau,\overline{\bm{X}}(\tau),\overline{\bm{\alpha}}(\tau,\overline{\bm{X}}(\tau))\big)d\tau+\sigma\big(\tau,\overline{\bm{X}}(\tau),\overline{\bm{\alpha}}(\tau,\overline{\bm{X}}(\tau))\big)d\bm{W}(\tau),~\tau\in[s,T], \\
		d\overline{\bm{Y}}^t(\tau) & = & -\bm{h}\big(t,\tau,\overline{\bm{X}}(\tau),\overline{\bm{\alpha}}(\tau,\overline{\bm{X}}(\tau)),\overline{\bm{Y}}^t(\tau),\overline{\bm{Z}}^t(\tau)\big)d\tau+\overline{\bm{Z}}^t(\tau)d\bm{W}(\tau),~\tau\in[s,T], \\
		\overline{\bm{X}}(s) & = & y, \quad \overline{\bm{Y}}^t(T) ~ = ~ \bm{g}\big(t,\overline{\bm{X}}(T)\big), \quad (t,s)\in\nabla[0,T], \quad y\in\mathbb{R}^d. 
	\end{array}
	\right.
\end{equation*}
Then, $(\bm{V}(s,y),\bm{u}(t,s,y))$ solves \eqref{HJB system} and \eqref{Equilibrium HJB system} while $\overline{\bm{\alpha}}(s,y)$ realizes the infimum of \eqref{HJB system}.
\end{description}


By the derivations in the previous subsection and using the similar arguments in \cite{Bjoerk2017}, it is easy to establish the verification theorem for the Markovian setting while we omit the straightforward proof. It should be noted that the equilibrium HJB equation and the verification theorem for the non-Markovian setting are attempted in \cite{Hernandez2023}. The necessity issue is a difficult problem, while we refer the readers to the latest progress alone this line, such as \cite{Lindensjoe2019,Hernandez2023,He2021,Hamaguchi2021} and a comprehensive literature review of the field \cite{He2022}. Specifically, \cite{Hernandez2023} proves the necessity for the scalar case (stochastic control problem) in a general non-Markovian setting. With a more mathematically rigorous definition of the SPE/TC solution (similar to Definition \ref{def:TCNEpt}) and a discretization approach, \cite{Wei2017,Yan2019,Wang2022,Wang2021} show intuitively the desired mathematical connection between TIC stochastic control problem and the associated equilibrium HJB equation, given that the latter is well-posed. Even we consider TIC SDGs with higher dimensions, one could expect that the sufficiency and the necessity above are provable.

However, they are not the focus of this paper while the well-posedness of the equilibrium HJB system appears to be the core of the concerns above. It is noteworthy that the solvability of TIC SDGs has not been provided by the works on the sufficiency and necessity. The desired mathematical connection between the equilibrium HJB equations/systems and TIC stochastic controls/SDGs is meaningless if the equilibrium HJB equations/systems are not well-posed; see the assumptions in the two conjectures above. Moreover, though some works have proved the sufficiency and necessity, they uniformly assumed that the volatility $\sigma$ of \eqref{State equation} is free of control such that the nonlocal second-order term in \eqref{Equilibrium HJB system} vanishes; see \cite{Wei2017,Hernandez2023}. Hence, our well-posedness results, which get rid of this bottleneck, benefit the studies on more general problems of TIC SDGs as well as the previous studies on TIC stochastic controls. While there may not be a specific order of studying the sufficiency, the necessity, and the well-posedness of the nonlocal parabolic system, this paper addresses the last one and based on which, the other two are relatively simple given the extensive related studies in the literature. By establishing the existence, uniqueness, stability, and regularities of solutions of \eqref{Equilibrium HJB system}, our PDE results directly imply the existence and uniqueness of TIC problems where both the drift and the volatility are controlled.

\section{Well-posedness of Nonlocal Parabolic System} \label{sec:mainresults}
In this section, we present our main results about the well-posedness issues of the nonlocal higher-order systems \eqref{Nonlocal fully nonlinear system}. The overall idea is to first study the case with a linear operator, which will then be used to infer the well-posedness results for the system \eqref{Nonlocal fully nonlinear system} with a general nonlinear operator, together with the linearization method. All the proofs are deferred to \ref{app:pf}.

While the SDG and the Feynman--Kac formula are usually formulated in a backward setting, we first show the equivalence between the solvabilities of the nonlocal backward (terminal-value) problems, 
\begin{equation} \label{Backward system} 
\left\{
\begin{array}{lr}
	\bm{u}_s(t,s,y)+\bm{F}\big(t,s,y,\left(\partial_I\bm{u}\right)_{|I|\leq 2r}(t,s,y),  \left(\partial_I\bm{u}\right)_{|I|\leq 2r}(s,s,y)\big)=0, \\
	\bm{u}(t,T,y)=\bm{g}(t,y),\hfill 0\leq t\leq s\leq T,\quad y\in\mathbb{R}^d,
\end{array}
\right. 
\end{equation}
and the nonlocal forward (initial-value) problems \eqref{Nonlocal fully nonlinear system}.
There are a few noteworthy differences between the two systems: first, for the backward problem \eqref{Backward system}, if we move the $\bm{F}$ to the right-hand side, we will have a negative sign $-\bm{F}$, compared to the forward problem \eqref{Nonlocal fully nonlinear system}; second,
the ordering between $t$ and $s$ are the opposite of one another. The symmetry between \eqref{Nonlocal fully nonlinear system} and \eqref{Backward system} is shown in Figure \ref{fig:backward problems}. Notation-wise, we use the time region $\Delta[0,T]:=\{(\tau_1,\tau_2)\in[0,T]^2:~0\le \tau_2\le \tau_1 \le T\}$ for forward problems to distinguish from $\nabla[0,T]$ for backward problems.

\begin{figure}[!ht]
\centering
\includegraphics[width=0.3\textwidth]{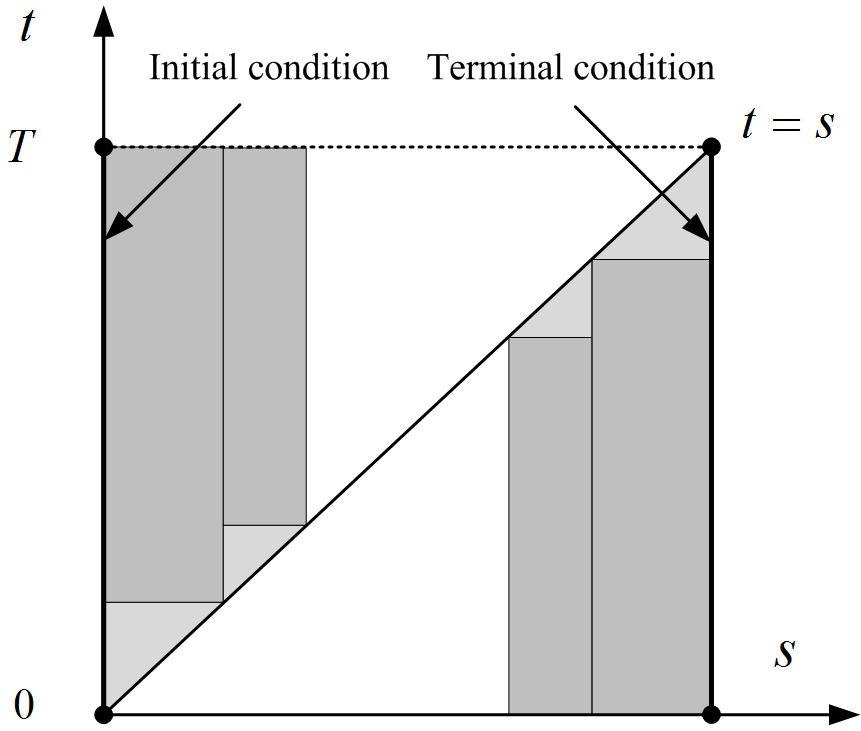}
\caption{Symmetry between forward and backward problems}
\label{fig:backward problems} 
\end{figure}

\begin{proposition}\label{Equivalance between forward probelms and backward problems} 
The solvabilities of Problems \eqref{Nonlocal fully nonlinear system} and  \eqref{Backward system} are equivalent.
\end{proposition}

Given Proposition \ref{Equivalance between forward probelms and backward problems}, we only study the forward problem \eqref{Nonlocal fully nonlinear system} in this section as it can simplify the notations. To this end, we introduce some norms and the induced Banach spaces for the problems of our interest. {\color{black} The solvability of \eqref{Nonlocal fully nonlinear system} will be first investigated in the usual space of bounded and continuous functions in Subsection \ref{sec:normandBspace}-\ref{sec:nonlinear}. Subsequently, in order to meet practical needs for more financial applications, it is necessary to extend the well-posedness results in an exponentially weighted space of growth functions in Subsection \ref{Sec:WeightedNormsSpaces}. }

\subsection{Norms and Banach Spaces} \label{sec:normandBspace}
For a $m$-dimensional real-valued array $\bm{x}=(\bm{x}^1,\bm{x}^2,\cdots,\bm{x}^m)$, $|\bm{x}|:=\left(\sum^m_{i=1}(\bm{x}^i)^2\right)^\frac{1}{2}$.
Given $0\leq a\leq b\leq T$, we denote by $C([a,b]\times\mathbb{R}^d;\mathbb{R}^m)$ the set of all the continuous and bounded $\mathbb{R}^m$-valued functions in $[a,b]\times\mathbb{R}^d$ endowed with the supremum norm $|\cdot|^{\infty}_{[a,b]\times\mathbb{R}^d}:=\sup_{[a,b]\times\mathbb{R}^d}|\cdot|$. Wherever no confusion arises, we write $|\cdot|^\infty$ instead of $|\cdot|^{\infty}_{[a,b]\times\mathbb{R}^d}$. 
Then, we revisit the definition of ``parabolic" H\"{o}lder spaces, which is commonly adopted in the studies of local parabolic equations, including \cite{Eidelman1969,Lei2023}. Let $C^{\frac{l}{2r},l}({[a,b]\times\mathbb{R}^d};\mathbb{R})$ be the Banach space of the functions $\varphi(s,y)$ such that $\varphi(s,y)$ is continuous in $[a,b]\times\mathbb{R}^d$, its derivatives of the form $\partial^h_s\partial^j_y\varphi$ for $2rh+j<l$ exist, and it has a finite norm defined by
\begin{equation*}
|\varphi|^{(l)}_{[a,b]\times\mathbb{R}^d}= \sum_{k\leq\lfloor l\rfloor}\sum_{2rh+j=k}\left|\partial^h_s\partial^j_y\varphi\right|^\infty+\sum_{2rh+j=\lfloor l\rfloor}\big\langle \partial^h_s\partial^j_y\varphi\big\rangle^{(l-\lfloor l\rfloor)}_y+\sum_{0<l-2rh-j<2r}\big\langle \partial^h_s\partial^j_y\varphi\big\rangle^{\big(\frac{l-2rh-j}{2r}\big)}_s, 
\end{equation*}
where $r$ is always a positive integer, $l$ is a non-integer positive number and $\lfloor \cdot\rfloor$ is the floor function, $\partial^h_s\partial^j_y\varphi$ represents the $d^j$-dimensional array, the entries of which are the \textit{j}-th-order mixed partial derivatives of $\frac{\partial^{h}\varphi}{\partial s\cdots\partial s}$ in $y$, i.e. $\frac{\partial^{h+j}\varphi}{\partial s\cdots\partial s\partial y_{i_1}\cdots\partial y_{i_j}}$.
Moreover, for $0<\alpha<1$ {\color{black}and $\rho_0>0$}, 
\begin{equation*}
\langle\varphi\rangle^{(\alpha)}_y:=\sup\limits_{\begin{subarray}{c} {\color{black}s\in[a,b],y,y^\prime\in\mathbb{R}^d} \\ {\color{black}0<|y-y^\prime|\leq \rho_0} \end{subarray}}\frac{|\varphi(s,y)-\varphi(s,y^\prime)|}{|y-y^\prime|^\alpha}, ~ \langle\varphi\rangle^{(\alpha)}_s:=\sup\limits_{\begin{subarray}{c} {\color{black}s,s^\prime\in[a,b],y\in\mathbb{R}^d} \\ {\color{black}0<|s-s^\prime|\leq \rho_0} \end{subarray}}\frac{|\varphi(s,y)-\varphi(s^\prime,y)|}{|s-s^\prime|^\alpha}.
\end{equation*}
{\color{black}The defined norms depend on $\rho_0$ but indeed for different $\rho_0>0$ they are equivalent. Hence, we suppress the dependence on $\rho_0$ will be not noted unless otherwise specified.} Moreover, wherever no confusion arises, we do not distinguish between $|\varphi|^{(l)}_{[a,b]\times\mathbb{R}^d}$ and $|\varphi|^{(l)}_{\mathbb{R}^d}$ for functions $\varphi(y)$ independent of $s$. 

Now, we are ready to define the norms and Banach spaces for nonlocal systems of unknown vector-valued functions $\bm{u}(t,s,y)=(\bm{u}^1(t,s,y),\bm{u}^2(t,s,y),\cdots,\bm{u}^m(t,s,y))^\top$. For any $t$ and $\delta$ such that $0\leq t\leq\delta\leq T$, we introduce the following norms:
\begin{eqnarray*}
~[\bm{u}]^{(l)}_{[0,\delta]} & := &  \sup\limits_{t\in[0,\delta]}\left\{|\bm{u}(t,\cdot,\cdot)|^{(l)}_{[0,t]\times\mathbb{R}^d}\right\}, \\
\|\bm{u}\|^{(l)}_{[0,\delta]} & := &  \sup\limits_{t\in[0,\delta]}\left\{|\bm{u}(t,\cdot,\cdot)|^{(l)}_{[0,t]\times\mathbb{R}^d}+|\bm{u}_t(t,\cdot,\cdot)|^{(l)}_{[0,t]\times\mathbb{R}^d}\right\},
\end{eqnarray*}
where $\left|\bm{u}(t,\cdot,\cdot)\right|^{(l)}_{[0,t]\times\mathbb{R}^d}:=\sum\limits_{a\leq m}\left[|\bm{u}^a(t,\cdot,\cdot)|^{(l)}_{[0,t]\times\mathbb{R}^d}\right]$. Then, these norms induce the following spaces, respectively,  
\begin{eqnarray*}
\bm{\Theta}^{(l)}_{[0,\delta]} & := & \left\{\bm{u}(\cdot,\cdot,\cdot)\in C(\Delta[0,\delta]\times\mathbb{R}^d;\mathbb{R}^m):~[\bm{u}]^{(l)}_{[0,\delta]}<\infty\right\}, \\
\bm{\Omega}^{(l)}_{[0,\delta]} & := & \left\{\bm{u}(\cdot,\cdot,\cdot)\in C(\Delta[0,\delta]\times\mathbb{R}^d;\mathbb{R}^m):~\Vert\bm{u}\|^{(l)}_{[0,\delta]}<\infty\right\},
\end{eqnarray*}
where $C(\Delta[0,\delta]\times\mathbb{R^d};\mathbb{R}^m)$ is the set of all continuous and bounded $\mathbb{R}^m$-valued functions defined in $\{0\leq s\leq t\leq \delta\}\times\mathbb{R}^d$. It is easy to see that both
$\bm{\Theta}^{(l)}_{[0,\delta]}$ and $\bm{\Omega}^{(l)}_{[0,\delta]}$ are Banach spaces. The definitions above leverage not only the order relation between $t$ and $s$ but also the sufficient regularities in all arguments.


\subsection{Nonlocal Linear Higher Order Parabolic Systems} \label{sec:linear}
Let $\bm{L}$ be a family of nonlocal, linear, and strongly elliptic operator of order $2r$, whose $a$-th entry, $\left(\bm{L}\bm{u}\right)^a$, $a=1,\ldots,m$, takes the form
\begin{equation} \label{eq:Loperator}
\left(\bm{L}\bm{u}\right)^a(t,s,y):=\sum\limits_{|I|\leq 2r,b\leq m}A^{aI}_b(t,s,y)\partial_I\bm{u}^b(t,s,y)+\sum\limits_{|I|\leq 2r,b\leq m}B^{aI}_b(t,s,y)\partial_I\bm{u}^b(s,s,y),
\end{equation}
where the nonlocality stems from the presence of $\partial_I\bm{u}(s,s,y)$ and the strong ellipticity condition implies that there exists some $\lambda>0$ such that
\begin{eqnarray}
(-1)^{r-1}\sum_{a,b,|I|=2r}A^{aI}_b(t,s,y)\xi_{i_1}\cdots\xi_{i_{2r}}v^av^b & \geq & \lambda|\xi|^{2r}|v|^2, \label{Ellipticity 1} \\
(-1)^{r-1}\sum_{a,b,|I|=2r}\left(A^{aI}_b+B^{aI}_b\right)(t,s,y)\xi_{i_1}\cdots\xi_{i_{2r}}v^av^b & \geq & \lambda|\xi|^{2r}|v|^2 \label{Ellipticity 2} 
\end{eqnarray} 
uniformly for any $(t,s)\in\Delta[0,T]$, $y,\xi\in\mathbb{R}^d$, and $v\in\mathbb{R}^m$. Next we consider a nonlocal linear system: 
\begin{equation} \label{Nonlocal linear system}  
\left\{
\begin{array}{rcl}
\bm{u}_s(t,s,y) & = & \left(\bm{L}\bm{u}\right)(t,s,y)+\bm{f}(t,s,y), \\
\bm{u}(t,0,y) & = & \bm{g}(t,y),\qquad \qquad \hfill 0\leq s\leq t\leq T,\quad y\in\mathbb{R}^d.
\end{array}
\right. 
\end{equation}
where all coefficients $A^{aI}_b$ and $B^{aI}_b$ belong to $\bm{\Omega}^{{(\alpha)}}_{[0,T]}$. Moreover, the inhomogeneous term $\bm{f}\in\bm{\Omega}^{(\alpha)}_{[0,T]}$ and the initial condition $\bm{g}\in\bm{\Omega}^{(2r+\alpha)}_{[0,T]}$.  

Suppose that $\bm{u}(t,s,y)$ is differentiable with respect to $t$, then by differentiating \eqref{Nonlocal linear system} with respect to $t$, the derivative $\frac{\partial\bm{u}}{\partial t}$ satisfies 
\begin{equation} \label{Differenate nonlocal linear system with respect to t}
\left\{
\begin{array}{rcl}
\left(\frac{\partial\bm{u}}{\partial t}\right)^a_s(t,s,y) & = & \sum\limits_{|I|\leq 2r,b\leq m}A^{aI}_b(\cdot)\partial_I \left(\frac{\partial\bm{u}}{\partial t}\right)^b(t,s,y)+\sum\limits_{|I|\leq 2r,b\leq m}\left(\frac{\partial A}{\partial t}\right)^{aI}_b(\cdot)\partial_I \bm{u}^b(t,s,y) \\
&& +\sum\limits_{|I|\leq 2r,b\leq m}\left(\frac{\partial B}{\partial t}\right)^{aI}_b(\cdot)\partial_I \bm{u}^b(s,s,y)+\bm{f}^a_t(\cdot), \hfill a=1,\ldots,m,     \\
\left(\frac{\partial\bm{u}}{\partial t}\right)(t,0,y) & = & \bm{g}_t(t,y), \hfill 0\leq s\leq t\leq T,~y\in\mathbb{R}^d.
\end{array}
\right.
\end{equation} 

By taking advantage of the integral representations:
\begin{equation*}
\partial_I\bm{u}^b(t,s,y)-\partial_I\bm{u}^b(s,s,y)=\int^t_s\partial_I\left(\frac{\partial\bm{u}}{\partial t}\right)^b(\theta,s,y)d\theta, \quad \text{for } |I|\leq 2r,~b\leq m, 
\end{equation*}
it is clear that $\left(\bm{u},\frac{\partial\bm{u}}{\partial t}\right)$, denoted by $(\bm{u},\bm{v})$, satisfies the following system of $2m$ equations: 
\begin{equation} \label{Induced nonlocal linear PDE system}  
\left\{
\begin{array}{rcl}
\bm{u}^a_s(t,s,y) & = & \sum\limits_{|I|\leq 2r,b\leq m}\left(A+B\right)^{aI}_b(\cdot)\partial_I\bm{u}^b(t,s,y)-\sum\limits_{|I|\leq 2r,b\leq m}B^{aI}_b(\cdot)\int^t_s\partial_I \bm{v}^b(\theta,s,y)d\theta \\
&& +\bm{f}^a(\cdot), \hfill a=1,\ldots,m, \\
\bm{v}^a_s(t,s,y) & = & \sum\limits_{|I|\leq 2r,b\leq m}A^{aI}_b(\cdot)\partial_I \bm{v}^b(t,s,y)+\sum\limits_{|I|\leq 2r,b\leq m}\left(\frac{\partial A}{\partial t}+\frac{\partial B}{\partial t
}\right)^{aI}_b(\cdot)\partial_I \bm{u}^b(t,s,y) \\
&& -\sum\limits_{|I|\leq 2r,b\leq m}\left(\frac{\partial B}{\partial t}\right)^{aI}_b(\cdot)\int^t_s \partial_I\bm{v}^b(\theta,s,y)d\theta+\bm{f}^a_t(\cdot), \hfill a=1,\ldots,m, \\
\left(\bm{u},\bm{v}\right)(t,0,y) & = & \left(\bm{g},\bm{g}_t\right)(t,y), \hfill 0\leq s\leq t\leq T,~y\in\mathbb{R}^d.  
\end{array}
\right.
\end{equation} 

The following lemma reveals that problems \eqref{Nonlocal linear system} and \eqref{Induced nonlocal linear PDE system} are equivalent.
\begin{lemma} \label{Equivalence between systems}
\begin{enumerate}
\item If $\bm{u}$ is a solution of \eqref{Nonlocal linear system}, then $(\bm{u},\bm{u}_t)$ solves \eqref{Induced nonlocal linear PDE system}. 
\item Conversely, if \eqref{Induced nonlocal linear PDE system} admits a solution pair $(\bm{u},\bm{v})$, then $\bm{u}$ solves \eqref{Nonlocal linear system}. 
\end{enumerate}
\end{lemma}

With Lemma \ref{Equivalence between systems}, it makes sense for us to shift our focus to the well-posedness of \eqref{Induced nonlocal linear PDE system}.
\begin{theorem} \label{Well-posedness of u and v}
Let $\bm{L}$ be the nonlocal, linear, and strongly elliptic operator of order $2r$ defined in \eqref{eq:Loperator} with all coefficients belonging to $\bm{\Omega}^{(\alpha)}_{[0,T]}$. If $\bm{f}\in\bm{\Omega}^{(\alpha)}_{[0,T]}$ and $\bm{g}\in\bm{\Omega}^{(2r+\alpha)}_{[0,T]}$, then \eqref{Induced nonlocal linear PDE system} admits a unique solution pair $(\bm{u},\bm{v})\in\bm{\Theta}^{(2r+\alpha)}_{[0,T]}\times\bm{\Theta}^{(2r+\alpha)}_{[0,T]}$ in $\Delta[0,T]\times\mathbb{R}^d$.   
\end{theorem}

Next, thanks to the equivalence between \eqref{Nonlocal linear system} and \eqref{Induced nonlocal linear PDE system}, we will establish the global well-posedness of nonlocal linear systems. Moreover, we will derive a Schauder-type estimate of solutions of \eqref{Nonlocal linear system}. It not only justifies the stability of the solutions to \eqref{Nonlocal linear system} with respect to the data $(\bm{f},\bm{g})$, but also establishes a foundation for the further analysis of nonlocal fully nonlinear systems \eqref{Nonlocal fully nonlinear system} in the next section.

\begin{theorem} \label{Schauder estimates} 
Suppose that all coefficient functions and $\bm{f}$ of \eqref{Nonlocal linear system} belong to $\bm{\Omega}^{{(\alpha)}}_{[0,T]}$ and assume that $\bm{g}\in\bm{\Omega}^{{(2r+\alpha)}}_{[0,T]}$.Then the nonlocal linear system \eqref{Nonlocal linear system} admits a unique solution $\bm{u}\in\bm{\Omega}^{{(2r+\alpha)}}_{[0,T]}$ in $\Delta[0,T]\times\mathbb{R}^d$. Furthermore, we obtain the following Schauder estimate 
\begin{equation} \label{Estimates of solutions of nonlocal system} 
\| \bm{u}\|^{(2r+\alpha)}_{[0,T]}\leq C\left(\| \bm{f}\|^{(\alpha)}_{[0,T]}+\| \bm{g}\|^{(2r+\alpha)}_{[0,T]}\right). 
\end{equation} 
Consequently, let $\bm{u}$ and $\widehat{\bm{u}}$ be solutions to \eqref{Nonlocal linear system} corresponding to $(\bm{f},\bm{g})$ and $(\widehat{\bm{f}},\widehat{\bm{g}})$, respectively, then  
\begin{equation} \label{Stability analysis of linear system} 
\| \bm{u}-\widehat{\bm{u}}\|^{(2r+\alpha)}_{[0,T]}\leq C\left(\| \bm{f}-\widehat{\bm{f}}\|^{(\alpha)}_{[0,T]}+\| \bm{g}-\widehat{\bm{g}}\|^{(2r+\alpha)}_{[0,T]}\right). 
\end{equation} 
\end{theorem}

\subsection{Nonlocal Fully Nonlinear Higher Order Parabolic Systems} \label{sec:nonlinear}
After studying the solvability of nonlocal linear system, we will adopt the method of linearization to prove the well-posedness of nonlocal fully nonlinear system of the form \eqref{Nonlocal fully nonlinear system}.

To take advantage of the results of nonlocal linear systems in Section \ref{sec:linear}, we require certain regularity assumptions on $\bm{F}$ and $\bm{g}$. Generally speaking, we require the initial condition $\bm{g}\in\bm{\Omega}^{(2r+\alpha)}_{[0,T]}$. The nonlinearity $\bm{F}$ is a vector-valued function $(t,s,y,z)\mapsto \bm{F}(t,s,y,z)$ defined in $\Pi=\Delta[0,T]\times\mathbb{R}^d\times B(\overline{z},R_0)$, where $\overline{z}\in\mathbb{R}^m\times(\mathbb{R}^d)^m\times\cdots\times(\mathbb{R}^{d^{2r}})^m\times\mathbb{R}^m\times(\mathbb{R}^d)^m\times\cdots\times(\mathbb{R}^{d^{2r}})^m$ and $B(\overline{z},R_0)$ is a open ball centered at $\overline{z}$ with a positive radius $R_0$. Denoting by $\bm{\mathcal{O}}$ the open set (ball) in $\bm{\Omega}^{(2r+\alpha)}_{[0,T]}$ consisting of all the functions $\bm{u}$ such that the range of $\big(\left(\partial_I\bm{u}\right)_{|I|\leq 2r}(t,s,y),  \left(\partial_I\bm{u}\right)_{|I|\leq 2r}(s,s,y)\big)$ is contained in the open ball, then the nonlinearity $\bm{F}$ can be regarded as a mapping from $\bm{\mathcal{O}}\to\bm{\Omega}^{(\alpha)}_{[0,T]}$. We require $\bm{F}$ satisfies that 
\begin{enumerate}[label=(\roman*)]
\item \textbf{(Ellipticity condition)} for any $\xi=(\xi_1,\dots,\xi_d)^\top\in\mathbb{R}^d$ and $v=(v^1,\cdots.v^m)^\top\in\mathbb{R}^m$, there exists a $\lambda>0$ such that     
\begin{eqnarray}
(-1)^{r-1}\sum_{a,b,|I|=2r}\partial_I \bm{F}^a_b(t,s,y,z)\xi_{i_1}\cdots\xi_{i_{2r}}v^av^b & \geq & \lambda|\xi|^{2r}|v|^2, \label{Uniform ellipticity condition of F 1} \\
(-1)^{r-1}\sum_{a,b,|I|=2r}\left(\partial_I \bm{F}^a_b+\partial_I \overline{\bm{F}}^a_b\right)(t,s,y,z)\xi_{i_1}\cdots\xi_{i_{2r}}v^av^b & \geq & \lambda|\xi|^{2r}|v|^2, \label{Uniform ellipticity condition of F 2} 
\end{eqnarray}
hold uniformly with respect to $(t,s,y,z)\in\Pi$.; 
\item \textbf{(H\"{o}lder continuity)} there exists a positive constant $K$ such that 
\begin{equation} \label{Holder continuity of F}
K:=\sup\limits_{t\in[0,\delta],z\in B(\overline{z},R_0)}|\bm{\mathcal{F}}(t,\cdot,\cdot,z)|^{(\alpha)}_{[0,t]\times\mathbb
	{R}^d}<\infty; 
\end{equation} 
\item \textbf{(Lipschitz continuity)} there exists a $L>0$ such that for any $(t,s,y,z_1)$, $(t,s,y,z_2)\in\Pi$,  
\begin{equation} \label{Lipschitz continuity of F} 
|\bm{\mathcal{F}}(t,s,y,z_1)-\bm{\mathcal{F}}(t,s,y,z_2)|\leq L|z_1-z_2|, 
\end{equation}
\end{enumerate}
where $\partial_I \bm{F}^a_b$ denotes the derivative of $\bm{F}^a$ with respect to its argument $\partial_I\bm{u}^b(t,s,y)$ while $\partial_I \overline{\bm{F}}^a_b$ denotes the derivative of $\bm{F}^a$ with respect to its argument $\partial_I\bm{u}^b(s,s,y)$ and the generic notation $\bm{\mathcal{F}}$ represents $\bm{F}$ itself and some of its first- and second-order derivatives, whose variables to be differentiated are indicated by ``$\surd$" in Tables \ref{tab:table1} and \ref{tab:table2}. Hereafter, we also adopt the similar notations for second-order derivatives of $\bm{F}$: $\partial^2_{IJ}\bm{F}^a_{bc}$ denotes the derivative of $\partial_I\bm{F}^a_b$ with respect to its argument $\partial_J\bm{u}^c(t,s,y)$ and $\partial^2_{IJ}\overline{\bm{F}}^a_{bc}$ denotes the derivative of $\partial_I\overline{\bm{F}}^a_b$ with respect to its argument $\partial_J\bm{u}^c(t,s,y)$. In fact, for a simple check, the assumptions above over $(t,s)$
and $z$ are satisfied if $\bm{F}$ is thrice continuously differentiable with respect to its corresponding arguments. 

\begin{table}[!ht] 
\centering
\begin{tabular}{c| c c c c c c c c c}
\hline
$\mathcal{X}$ & $t$ & $s$ & $y$ & $\partial_I\bm{u}^b(t,s,y)$ & $\partial_I\bm{u}^b(s,s,y)$ \\ 
\hline 
$\bm{F}_\mathcal{X}$ & $\surd$ & & & $\surd$ & $\surd$ \\ 
\hline 
\end{tabular}
\caption{First-order derivatives of $\bm{F}$ required to be H\"{o}lder and Lipschitz continuous}
\label{tab:table1}
\end{table} 

\begin{table}[!ht] 
\centering
\begin{tabular}{c| c c c c c c c c c}
\hline
\diagbox{$\mathcal{X}$}{$\bm{F}_{\mathcal{X}\mathcal{Y}}$}{$\mathcal{Y}$} & $t$ & $s$ & $y$ & $\partial_I\bm{u}^b(t,s,y)$ & $\partial_I\bm{u}^b(s,s,y)$ \\ 
\hline 
$t$ &  &  &  & $\surd$ & $\surd$ \\ 
$s$ &  &  &  &  & \\
$y$ &  &  &  &  & \\
$\partial_I\bm{u}^b(t,s,y)$ & $\surd$ &  &  & $\surd$ & $\surd$ \\
$\partial_I\bm{u}^b(s,s,y)$ & $\surd$ &  &  & $\surd$ & \\
\hline 
\end{tabular}
\caption{Second-order derivatives of $\bm{F}$ required to be H\"{o}lder and Lipschitz continuous}
\label{tab:table2}
\end{table} 

\subsubsection{Small-time Well-posedness of Nonlocal Fully Nonlinear Nonlinear Systems}
Before we present our main result, we stress that the standard linearization methods are not applicable for the nonlocal case. In the setting of local parabolic systems, \cite{Eidelman1969} introduced a so-called ``quasi-linearization method" and studied local existence for fully nonlinear parabolic problems by transforming fully nonlinear systems into quasi-linear systems. Noteworthily, \cite{Khudyaev1963,Sopolov1970} utilized a variant of this method to investigate fully nonlinear PDEs or systems. The linearization method, which we propose to prove for Theorem \ref{Well-posedness of fully nonlinear system} below, is substantially inspired by \cite{Kruzhkov1975,Lunardi1995}. Although there are some previous works on how to linearize nonlinear equations or systems, it is still difficult to extend the existing methods from a local setting to a nonlocal setting. In fact, even for a nonlocal linear system \eqref{Nonlocal linear system}, the literature is lack of its mathematical analysis, not to mention the nonlinear case and the conversion from nonlinearity to linearity.

Theorem \ref{Well-posedness of fully nonlinear system} below is our main innovative result, which shows the (small-time) well-posedness of nonlocal fully nonlinear systems \eqref{Nonlocal fully nonlinear system}.

\begin{theorem}\label{Well-posedness of fully nonlinear system}
Let $\bm{F}$ satisfies the conditions \eqref{Uniform ellipticity condition of F 1}-\eqref{Lipschitz continuity of F}. Suppose that $\bm{g}\in\bm{\Omega}^{{(2r+\alpha)}}_{[0,T]}$ and that the range of $\big(\left(\partial_I\bm{g}\right)_{|I|\leq 2r}(t,y),\left(\partial_I\bm{g}\right)_{|I|\leq 2r}(s,y)\big)$ is contained in the ball centered at $\overline{z}$ with radius $R_0/2$. Then, there exist $\delta>0$ and a unique $\bm{u}\in\bm{\Omega}^{{(2r+\alpha)}}_{[0,\delta]}$ satisfying \eqref{Nonlocal fully nonlinear system} in $\Delta[0,\delta]\times\mathbb{R}^d$. 
\end{theorem} 

It should be noted that in the small-time setting, we only require the conditions of \eqref{Uniform ellipticity condition of F 1}-\eqref{Lipschitz continuity of F} of $\bm{F}$ in an open ball $B(\overline{z},R_0)$ while the range of $\big(\left(\partial_I\bm{g}\right)_{|I|\leq 2r}(t,y),\left(\partial_I\bm{g}\right)_{|I|\leq 2r}(s,y)\big)$ is contained in $\bm{\mathcal{O}}$ with a smaller ball $B(\overline{z},R_0/2)$. For a pair $(\bm{F},\bm{g})$ that satisfies their coupled assumptions, Theorem \ref{Well-posedness of fully nonlinear system} provides the local/maximally-defined (see Remark \ref{Maximal interval}) well-posedness of \eqref{Nonlocal fully nonlinear system}. The (relaxed) local assumptions on $\bm{F}$ facilitates a larger class of \eqref{Nonlocal fully nonlinear system}.
From the proof of Theorem \ref{Well-posedness of fully nonlinear system} (see \eqref{Range of u}) and the example provided below, we find that the local solution always exists if {\color{black}$g\in\bm{\mathcal{O}}$ since $\bm{\mathcal{O}}$ is an open set in $\bm{\Omega}^{(2r+\alpha)}_{[0,T]}$}. Hence, to check if \eqref{Nonlocal fully nonlinear system} exists a local solution, it is more convenient to check if the conditions \eqref{Uniform ellipticity condition of F 1}-\eqref{Lipschitz continuity of F} of $\bm{F}$ can be satisfied in a small open ball centered at the range of $\bm{g}$, instead of in $B(\overline{z},R_0)$.
For the global solvability (well-posedness), we will discuss it in Subsection \ref{Large Well-posedness of Nonlocal Nonlinear Systems}.

\begin{remark}[\textbf{Maximally defined solutions}] \label{Maximal interval}
We have proven the local well-posedness of \eqref{Nonlocal fully nonlinear system} in $\Delta[0,\delta]\times\mathbb{R}^d$ and thus the diagonal condition can be determined for $s\in[0,\delta]$. After which, the nonlocal fully nonlinear system \eqref{Nonlocal fully nonlinear system} is reduced to a classical local fully nonlinear systems parameterized by $t$. Then we take $\delta$ as initial time and $\bm{u}(t,\delta,y)$ as initial datum, we can extend the solution to a larger time interval up to the maximal interval. It is analogous to the process of identifying the global solution of nonlocal linear systems in the proof of Theorem \ref{Well-posedness of u and v}. The procedure could be repeated up to a \textit{maximally defined solution} $\bm{u}:\Delta[0,\sigma]\times\mathbb{R}^d\to\mathbb{R}^m$, belonging to $\bm{\Omega}^{(2+\alpha)}_{[0,\sigma]}$ for any $\sigma<\tau$. The time region $\Delta[0,\tau]$ is maximal in the sense that if $\tau<\infty$, then there does not exist any solution of \eqref{Nonlocal fully nonlinear system} belonging to $\bm{\Omega}^{(2+\alpha)}_{[0,\tau]}$; see Figure \ref{fig:extension}. An example of $\tau<T$ can be proposed similarly as in the local case; see \cite[pp. 203]{Lieberman1996}. It is noteworthy that the problem of existence at large for arbitrary initial data is a difficult task even in the classical fully nonlinear case. The difficulty is caused by the fact that a priori estimate in a very high norm $|\cdot|^{(2+\alpha)}_{[a,b]\times\mathbb{R}^d}$ is needed to establish the existence at large. To this end, there will be severe restrictions on the nonlinearities. More details are discussed in \cite{Krylov1987,Lieberman1996}.
\end{remark}

\begin{remark}[\textbf{Stability analysis}] \label{NonlinearSA}
Consider a family of nonlinearities $\bm{F}(t,s,y,z;\lambda)$ parameterized by a parameter $\lambda\in\Lambda$, where $\Lambda$ is a Banach space under $\|\cdot\|_\Lambda$. For any $(t,s,y,\bm{u})\in\Delta[0,\delta]\times\mathbb{R}^d\times\bm{\Omega}^{(2r+\alpha)}_{[0,\delta]}$, it is assumed that 
\begin{equation*}
\|\bm{F}\big(\cdot,\cdot,\cdot,\left(\partial_I\bm{u}\right)_{|I|\leq 2r}(\cdot,\cdot,\cdot),  \left(\partial_I\bm{u}\right)_{|I|\leq 2r}(\cdot,\cdot,\cdot);\lambda\big)-\bm{F}\big(\cdot,\cdot,\cdot,\left(\partial_I\bm{u}\right)_{|I|\leq 2r}(\cdot,\cdot,\cdot),  \left(\partial_I\bm{u}\right)_{|I|\leq 2r}(\cdot,\cdot,\cdot);\widehat{\lambda}\big)\|^{(\alpha)}_{[0,\delta]}\leq \gamma\|\lambda-\widehat{\lambda}\|_{\Lambda}.
\end{equation*}
Suppose $\bm{u}$ and $\widehat{\bm{u}}$ correspond to $(\lambda,\bm{g})$ and $(\widehat{\lambda},\widehat{\bm{g}})$, respectively. We have the following estimate: 
\begin{equation*}
\| \bm{u}-\widehat{\bm{u}}\|^{(2r+\alpha)}_{[0,\delta]}\leq C\left(\|\lambda-\widehat{\lambda}\|_\Lambda+\|\bm{ g}-\widehat{\bm{g}}\|^{(2+\alpha)}_{[0,\delta]}\right), 
\end{equation*} 
which follows directly the proofs of our Theorem \ref{Well-posedness of fully nonlinear system} and Theorem 8.3.2 in \cite{Lunardi1995}. 	
\end{remark}

Before we study the global well-posedness, we provide an example to understand the assumptions on $\bm{F}$ and $\bm{g}$ in Theorem \ref{Well-posedness of fully nonlinear system}. 
Without loss of generality, we assume that $r=m=d=1$ and the coefficients of \eqref{State equation} and \eqref{Cost functional} are 
\begin{equation*} 
\left\{
\begin{array}{rcl}
h(t,s,y,\alpha) & = & C^{(1)}(t,s,y)+\frac{1}{2}C^{(2)}(t,s,y)\alpha^2, \\
b(s,y,\alpha) & = & B^{(1)}(s,y)+B^{(2)}(s,y)\alpha, \\
A(s,y,\alpha): & = & \frac{1}{2}\sigma(s,y,\alpha)\sigma(s,y,\alpha)=A^{(1)}(s,y)+\frac{1}{2}A^{(2)}(s,y)\alpha^2. 
\end{array}
\right. 
\end{equation*}
Then, it is clear that the optimun of the Hamiltonian is attained by $\phi(t,s,y,u,p,q)=-B^{(2)}(s,y)p/(A^{(2)}(s,y)q+C^{(2)}(t,s,y))$. Moreover, according to \eqref{Equilibrium strategy},
the equilibrium control is given by
$$
\overline{\alpha}(s,y)=\frac{-B^{(2)}(s,y)u_y(s,s,y)}{A^{(2)}(s,y)u_{yy}(s,s,y)+C^{(2)}(s,s,y)}.
$$
Consequently, with a variable substitution, the (backward) equilibrium HJB equation can be reformulated forwardly as:
\begin{equation} \label{Equilibrium HJB equation of an example}
\left\{
\begin{array}{lr}
u_s(t,s,y) = \Big(\dot{A}^{(1)}(s,y)+\frac{1}{2}\dot{A}^{(2)}(s,y)\left(\frac{\dot{B}^{(2)}(s,y)u_y(s,s,y)}{\dot{A}^{(2)}(s,y)u_{yy}(s,s,y)+\dot{C}^{(2)}(s,s,y)}\right)^2\Big)u_{yy}(t,s,y) \\
\qquad\qquad\qquad 
+\Big(\dot{B}^{(1)}(s,y) -\frac{\dot{B}^{(2)}(s,y)\dot{B}^{(2)}(s,y)u_y(s,s,y)}{\dot{A}^{(2)}(s,y)u_{yy}(s,s,y)+\dot{C}^{(2)}(s,s,y)}\Big)u_y(t,s,y) \\
\qquad\qquad\qquad\qquad 
+\dot{C}^{(1)}(t,s,y)+\frac{1}{2}\dot{C}^{(2)}(t,s,y)\left(\frac{\dot{B}^{(2)}(s,y)u_y(s,s,y)}{\dot{A}^{(2)}(s,y)u_{yy}(s,s,y)+\dot{C}^{(2)}(s,s,y)}\right)^2, \\
u(t,0,y) = \dot{g}(t,y),\quad 0\leq s \leq t\leq T,\quad y\in\mathbb{R}, 
\end{array}
\right. 
\end{equation}
where $\dot{G}(t,s,y)=G(T-t,T-s,y)$. Finally, by our local well-posedness results (Theorem \ref{Well-posedness of fully nonlinear system}), there exists $\delta>0$ such that \eqref{Equilibrium HJB equation of an example} is solvable in $\Delta[0,\delta]$, if there exists a constant $\epsilon>0$ such that   
\begin{enumerate}
\item $\dot{A}^{(2)}(s,y)\dot{g}_{yy}(s,s,y)+\dot{C}^{(2)}(s,s,y)\geq\epsilon$
\item $\frac{\partial\dot{F}}{\partial q}=\dot{ A}^{(1)}(s,y)+\frac{1}{2}\dot{A}^{(2)}(s,y)\left(\frac{\dot{B}^{(2)}(s,y)\dot{g}_y(s,s,y)}{\dot{A}^{(2)}(s,y)\dot{g}_{yy}(s,s,y)+\dot{C}^{(2)}(s,s,y)}\right)^2\geq\epsilon$ 
\item $\frac{\partial\dot{F}}{\partial q}+\frac{\partial\dot{F}}{\partial n}=\dot{A}^{(1)}(s,y)+\dot{f}(t,s,y,\dot{g}_y(t,s,y),\dot{g}_{yy}(t,s,y),\dot{g}_y(s,s,y),\dot{g}_{yy}(s,s,y))\geq\epsilon$
\end{enumerate}
where $\dot{F}$ is the nonlinearity of \eqref{Equilibrium HJB equation of an example} and $\dot{f}$ represents the remaining terms of $\frac{\partial\dot{F}}{\partial q}+\frac{\partial\dot{F}}{\partial n}$ excluding $\dot{A}^{(1)}$. In general, it is not necessary to identify the properties of nonlinearity $\dot{F}$ in a large ball $B(\overline{z},R_0)$. According to the proof of Theorem \ref{Well-posedness of fully nonlinear system}, the local solution of \eqref{Nonlocal fully nonlinear system} can arbitrarily approach to the initial data $g$ by choosing a small enough $\delta$; see \eqref{Range of u}. Consequently, if the domain, where nonlinearity $\dot{F}$ satisfies these requirements in Theorem \ref{Well-posedness of fully nonlinear system}, contains an open ball centered at $g$ (i.e. $g\in{\color{black}\bm{\mathcal{O}}}$), then there exists a local solution for the nonlocal systems. It is clear that $\dot{F}$ is locally Lipschitz and H\"{o}lder continuous. Moreover, for a large enough $\dot{A}^{(1)}$ and $\dot{C}^{(2)}$, the three inequalities above hold such that \eqref{Equilibrium HJB equation of an example} is solvable at least in a small time interval.


\subsubsection{On the Global Well-posedness of Nonlocal Nonlinear Systems} \label{Large Well-posedness of Nonlocal Nonlinear Systems}
In this subsection, we show that \eqref{Nonlocal fully nonlinear system} is well-posed globally, i.e. $\tau=T$, if a very sharp a priori estimate is available. Moreover, we introduce a class of nonlocal nonlinear system called nonlocal quasilinear system of the form \eqref{Nonlocal quasilinear system} and we establish its global solvability under a growth condition.


In contrast with the nonlocal linear systems \eqref{Nonlocal linear system}, where the (small-time) solution can be extended arbitrarily many times to a global solution over $\Delta[0,T]$ for any $T<\infty$,
it is possible for the nonlinear case \eqref{Nonlocal fully nonlinear system} that the extension procedure is terminated at some $\tau<T$. The dissatisfying result is caused mainly by the fact in the proof of Theorem \ref{Well-posedness of fully nonlinear system} that in order to obtain a $\frac{1}{2}$-contraction from $\bm{u}$ to $\bm{U}$ defined by $\bm{U}_s=\bm{L}_0\bm{U}+\bm{F}(\bm{u})-\bm{L}_0\bm{u}$, we need to strike a balance between $R$ and $\delta$ such that $C(R)\delta^\frac{\alpha}{2r}<\frac{1}{2}$. In the extension procedure in view of Remark \ref{Maximal interval}, it is possible that the solution $\bm{u}$ blows up near $\tau<T$. In this case, both $R$ and $C(R)$ tend to infinity under the norm $\|\cdot\|^{(2r+\alpha)}_{[0,\tau)}$. From this perspective, it becomes clear that the inequality $C(R)\delta^\frac{\alpha}{2r}<\frac{1}{2}$ has restricted $\delta$ to be infinitely small. Consequently, the extension procedure is forced to stop to generate a maximally defined solution over $[0,\tau)$ instead of a global solution over the whole interval $[0,T]$.

In fact, it has been an unavoidable problem in the study of differential equations. 
To extend the maximally defined solution from $[0,\tau)$ to $[0,\tau]$, the key step is to show that the mapping $s\mapsto\bm{u}(\cdot,s,\cdot)$ is uniformly continuous in some sense such that an analytic continuation argument works. Next, inspired by \cite{Lunardi1989,Prato1996}, we show that it is possible to have $\tau=T$ if a very sharp a priori estimate is available. 
\begin{theorem} \label{Global existence of fully nonlineaity} 
Let $\bm{F}$ and $\bm{g}$ satisfy the assumptions of Theorem \ref{Well-posedness of fully nonlinear system} with $\alpha$ replaced by $\alpha^\prime>\alpha$. For a fixed $\bm{g}\in\Omega^{(2r+\alpha^\prime)}_{[0,T]}$, let $\bm{u}$ be the maximally defined solution of problem \eqref{Nonlocal fully nonlinear system} over $[0,\tau)$. Assume further that there exists a finite constant $M>0$ such that
\begin{equation} \label{Upper bounded for global existence}
\| \bm{u}\|^{(2r+\alpha^\prime)}_{[0,\sigma]}\leq M ~ \text{for all} ~ \sigma\in[0,\tau), 
\end{equation}
then we have either $\lim_{s\to\tau}\bm{u}(\cdot,s,\cdot)\in\partial\bm{\mathcal{O}}$ or $\tau=T$. 
\end{theorem}

Generally speaking, in the proof of Theorem \ref{Well-posedness of fully nonlinear system}, the $R$ in the $C(R)\delta^\frac{\alpha}{2r}$ depends on $\|\bm{u}\|^{(2r+\alpha)}_{[0,\delta]}$ since the $\frac{1}{2}$-contraction operator $\bm{U}=\bm{\Lambda}(\bm{u})$ via  $\bm{U}_s=\bm{L}_0\bm{U}+\bm{F}(\bm{u})-\bm{L}_0\bm{u}$ defined in a closed set $\bm{\mathcal{U}}$ of $\bm{\Omega}^{(2r+\alpha)}_{[0,\delta]}$. Hence, the prior estimate of $\| \bm{u}\|^{(2r+\alpha)}_{[0,\sigma]}\leq M$ for all $\sigma\in[0,\tau)$ is not enough for the existence in the large. Instead, we need an estimate on the modulus of continuity of $s\mapsto\bm{u}(\cdot,s,\cdot)$. Similar sufficient conditions to obtain a priori estimates like \eqref{Upper bounded for global existence} for the classical PDEs/systems can be found in \cite{Krylov1987,Lieberman1996}. However, it is not straightforward to
express such conditions in terms of coefficients and data of the local and nonlocal fully nonlinear system.

Next, we show that the desired sharp a priori estimate is available for a class of nonlocal nonlinear systems, namely \textit{nonlocal quasilinear systems}, of the form: 
\begin{equation} \label{Nonlocal quasilinear system}  
\left\{
\begin{array}{rcl}
\bm{u}^a_s(t,s,y) & = & \sum\limits_{|I|= 2r,b\leq m}A^{aI}_b(s,y)\partial_I\bm{u}^b(t,s,y)+\bm{Q}^a\big(t,s,y,\left(\partial_I\bm{u}\right)_{|I|\leq 2r-1}(t,s,y),  \left(\partial_I\bm{u}\right)_{|I|\leq 2r-1}(s,s,y)\big), \\
\bm{u}(t,0,y) & = & \bm{g}(t,y),\qquad \qquad \hfill 0\leq s\leq t\leq T,\quad y\in\mathbb{R}^d, \qquad a=1,\ldots,m. 
\end{array}
\right. 
\end{equation} 
Compared with \eqref{Nonlocal fully nonlinear system}, \eqref{Nonlocal quasilinear system} is free of the highest order nonlocal term $(\partial_I\bm{u})_{|I|=2r}(s,s,y)$ and is linear in the highest order local term $(\partial_I\bm{u})_{|I|=2r}(t,s,y)$. It is clear that \eqref{Nonlocal quasilinear system} is a special case of \eqref{Nonlocal fully nonlinear system}. The nonlocal quasilinear systems are relevant from both theoretical and practical viewpoints, since they cover the equilibrium HJB systems of TIC SDG problems, where the diffusion of \eqref{State equation} is uncontrolled, i.e., $\sigma(s,y,\alpha)=\sigma(s,y)$. By leveraging Theorem \ref{Well-posedness of fully nonlinear system} and Theorem \ref{Global existence of fully nonlineaity} of nonlocal fully nonlinear systems \eqref{Nonlocal fully nonlinear system}, we aim to show that \eqref{Nonlocal quasilinear system} is solvable globally under some technical conditions in the theorem below.
In fact, our results have been the best in the existing literature on nonlocal PDEs/systems in terms of the global well-posedness issues.

\begin{theorem} \label{Wellposedness of Quasilinear systems} 
Suppose that all coefficient functions $A^{aI}_b$ of \eqref{Nonlocal quasilinear system} belong to $\bm{\Omega}^{{(\alpha)}}_{[0,T]}$ and satisfy \eqref{Ellipticity 1}, $\bm{g}\in\bm{\Omega}^{{(2r+\alpha)}}_{[0,T]}$, and the nonlinearity $\bm{Q}^a(t,s,y,w,\overline{w})$ has enough regularities required in \eqref{Holder continuity of F} and \eqref{Lipschitz continuity of F}, satisfies a linear growth condition: $|\bm{Q}^a|\leq K\left(1+|w|\right)$, and has its bounded first order partial derivatives with respect to $t$ and $w$. Then, the nonlocal quasilinear system \eqref{Nonlocal quasilinear system} admits a unique solution $\bm{u}\in\bm{\Omega}^{{(2r+\alpha)}}_{[0,T]}$ in $\Delta[0,T]\times\mathbb{R}^d$.
\end{theorem}

As closing remarks of this section, we review our studies on nonlocal linear \eqref{Nonlocal linear system}, quasilinear \eqref{Nonlocal quasilinear system}, and fully nonlinear systems \eqref{Nonlocal fully nonlinear system} in parallel. Our analyses are based on the Banach fixed point arguments to first establish their small-time solvability and then extend the results to a longer time horizon, while the later extension faces different situations for different systems.
In the case of nonlocal linear systems \eqref{Nonlocal linear system}, a contractive mapping can be constructed by choosing a suitably small $\delta$ such that $C\delta^{1-\frac{\alpha}{2r}}\leq\frac{1}{2}$ with a constant $C$ depending only on the data of \eqref{Nonlocal linear system}.
Hence, there is no issue as discussed at the beginning of Section \ref{Large Well-posedness of Nonlocal Nonlinear Systems} and thus the global well-posedness of \eqref{Nonlocal linear system} can be obtained without extra conditions. In contrast, the nonlocal nonlinear systems, including the quasilinear \eqref{Nonlocal quasilinear system} and the fully nonlinear \eqref{Nonlocal fully nonlinear system} systems,
need to balance $R$ and $\delta$ such that $C(R)\delta^\frac{\alpha}{2r}\leq \frac{1}{2}$. Through mathematical analyses, the $R$ for the quasilinear system \eqref{Nonlocal quasilinear system} is quantified by $\|\cdot\|^{(2r-1+\alpha)}_{[a,b]}$ while the fully nonlinear one is quantified by $\|\cdot\|^{(2r+\alpha)}_{[a,b]}$. This observation explains the different levels of difficulty when we establish their global existence; see Theorems \ref{Global existence of fully nonlineaity} and \ref{Wellposedness of Quasilinear systems}. It is an important and promising research direction to rewrite the condition \eqref{Upper bounded for global existence} in terms of the model coefficients of the original problem \eqref{Nonlocal fully nonlinear system}. 

{\color{black}
\subsection{Well-posedness in a Weighted Space} \label{Sec:WeightedNormsSpaces}
In this subsection, we extend the main results in the previous subsections to a weighted space, which allows its elements (functions) as well as their partial derivatives grow exponentially in the spatial variable $y$. Throughout this subsection, we consider the exponential weights, defined by $\varrho(y)=\exp\{1+\langle Sy,y\rangle^{1/2}\}$ for any $y\in\mathbb{R}^d$, $S$ being any symmetric positive-definite matrix with eigenvalues in $[\underline{\lambda},\overline{\lambda}]$ and $\underline{\lambda}>0$. First of all, we introduce the following weighted norms  
\begin{equation} \label{WeightedNorm1}
|\varphi|^{1,(l)}_{\varrho,[a,b]\times\mathbb{R}^d}=|\varphi/\varrho|^{(l)}_{[a,b]\times\mathbb{R}^d},
\end{equation}
\begin{equation} \label{WeightedNorm2}
|\varphi|^{2,(l)}_{\varrho,[a,b]\times\mathbb{R}^d}= \sum_{k\leq\lfloor l\rfloor}\sum_{2rh+j=k}\left|\frac{\partial^h_s\partial^j_y\varphi}{\varrho}\right|^\infty+\sum_{2rh+j=\lfloor l\rfloor}\Big\langle \frac{\partial^h_s\partial^j_y\varphi}{\varrho}\Big\rangle^{(l-\lfloor l\rfloor)}_y+\sum_{0<l-2rh-j<2r}\Big\langle \frac{\partial^h_s\partial^j_y\varphi}{\varrho}\Big\rangle^{\big(\frac{l-2rh-j}{2r}\big)}_s,
\end{equation}
\begin{equation} \label{WeightedNorm3} 
\begin{split}
	|\varphi|^{3,(l)}_{\varrho,[a,b]\times\mathbb{R}^d} = & \sum_{k\leq\lfloor l\rfloor}\sum_{2rh+j=k}\left|\frac{\partial^h_s\partial^j_y\varphi}{\varrho}\right|^\infty+\sum_{0<l-2rh-j<2r}\Big\langle \frac{\partial^h_s\partial^j_y\varphi}{\varrho}\Big\rangle^{\big(\frac{l-2rh-j}{2r}\big)}_s \\
	& + \sup\limits_{\begin{subarray}{c} {\color{black}s\in[a,b],y,y^\prime\in\mathbb{R}^d} \\ {\color{black}0<|y-y^\prime|\leq \rho_0} \end{subarray}}\frac{|\partial^h_s\partial^j_y\varphi(s,y)-\partial^h_s\partial^j_y\varphi(s,y^\prime)|}{|y-y^\prime|^{(l-\lfloor l\rfloor)}}\min\left\{\varrho^{-1}(y),\varrho^{-1}(y^\prime)\right\}.   
\end{split} 
\end{equation} 

Next, before defining weighted spaces, we illustrate the following equivalence property.  
\begin{lemma} \label{EquivalentNorms}
The three norms defined in \eqref{WeightedNorm1}-\eqref{WeightedNorm3} are equivalent. 
\end{lemma}

By Lemma \ref{EquivalentNorms}, the norms defined in \eqref{WeightedNorm1}-\eqref{WeightedNorm3} can be all denoted by an unified notation $|\varphi|^{(l)}_{\varrho,[a,b]\times\mathbb{R}^d}$. Similar to Subsection \ref{sec:normandBspace}, we can then define weighted norms $[\bm{u}]^{(l)}_{\varrho,[0,\delta]}$ and $\|\bm{u}\|^{(l)}_{\varrho,[0,\delta]}$ and weighted spaces $C^{l/2r,l}_\varrho$, $\bm{\Theta}^{(l)}_{\varrho,[0,\delta]}$ and $\bm{\Omega}^{(l)}_{\varrho,[0,\delta]}$. Although the norms defined in \eqref{WeightedNorm1}-\eqref{WeightedNorm3} are equivalent, it is useful to distinguish them since they have their own advantages. The first one \eqref{WeightedNorm1} presents an intuitive understanding of functions with an exponential growth in the spatial argument, while the other two weighted norms \eqref{WeightedNorm2} and \eqref{WeightedNorm3} are convenient in showing some related conclusions in Theorem \ref{WellposednessWeighted} below. 

Next, we introduce a class of nonlinearities that extend the nonlinearity $\bm{F}$ in Subsection \ref{sec:nonlinear} for the study of well-posedness in the weighted spaces. 
\begin{definition} \label{Def: AppropriatePair}
A pair of $(\bm{F},\bm{g})$ is appropriate if there exist $\delta$, $R>0$ such that for any $\bm{u}\in\big\{\bm{u}\in\bm{\Omega}^{(2r+\alpha)}_{\varrho,[0,\delta]}:\bm{u}(t,0,y)=\bm{g}(t,y),\| \bm{u}-\bm{g}\|^{(2r+\alpha)}_{\varrho,[0,\delta]}\leq R\big\}$,
\begin{enumerate}[label=(\alph*)]
	\item $\bm{F}\big(t,s,y,\left(\partial_I\bm{u}\right)_{|I|\leq 2r}(t,s,y),  \left(\partial_I\bm{u}\right)_{|I|\leq 2r}(s,s,y)\big)\in\bm{\Omega}^{(\alpha)}_{\varrho,[0,\delta]}$ while $\bm{F}_t$ at $\bm{u}$ belongs to $\bm{\Theta}^{(\alpha)}_{\varrho,[0,\delta]}$;
	\item both $\partial_I\bm{F}$ and $\partial_I\overline{\bm{F}}$ at $\bm{u}$ belong to $\bm{\Omega}^{(\alpha)}_{[0,\delta]}$;
	\item $
	\left\{
	\begin{array}{lr}
		(-1)^{r-1}\sum_{a,b,|I|=2r}\partial_I \bm{F}^a_b(t,s,y,\left(\partial_I\bm{g}\right)_{|I|\leq 2r}(t,y),  \left(\partial_I\bm{g}\right)_{|I|\leq 2r}(0,y))\xi_{i_1}\cdots\xi_{i_{2r}}v^av^b \geq  \lambda|\xi|^{2r}|v|^2, \\
		(-1)^{r-1}\sum_{a,b,|I|=2r}\left(\partial_I \bm{F}^a_b+\partial_I \overline{\bm{F}}^a_b\right)(t,s,y,\left(\partial_I\bm{g}\right)_{|I|\leq 2r}(t,y),  \left(\partial_I\bm{g}\right)_{|I|\leq 2r}(0,y))\xi_{i_1}\cdots\xi_{i_{2r}}v^av^b \geq \lambda|\xi|^{2r}|v|^2; 
	\end{array}
	\right. $
	\item $\left|\Delta_{s,y}\bm{\mathcal{F}}\big(t,s,y,\left(\partial_I\bm{u}\right)_{|I|\leq 2r}(t,s,y),  \left(\partial_I\bm{u}\right)_{|I|\leq 2r}(s,s,y)\big)\right|\leq C(R)\left(|s-s^\prime|^\frac{\alpha}{2}+|y-y^\prime|^\alpha\right)$,~ $\bm{\mathcal{F}}\in\big\{\partial_I\bm{F},\partial_I\overline{\bm{F}},\partial^2_{It}\bm{F},\partial^2_{It}\overline{\bm{F}}\big\}$;
	\item and $
	\left\{
	\begin{array}{lr}
		\left|\Delta_{s,y}\big[\partial^2_{IJ}\overline{\bm{F}}^a_{bc}\big(t,s,y,\left(\partial_I\bm{u}\right)_{|I|\leq 2r}(t,s,y),  \left(\partial_I\bm{u}\right)_{|I|\leq 2r}(s,s,y)\big)\cdot\partial_J\bm{u}^c_t(t,s,y)\big]\right|\leq C(R)\left(|s-s^\prime|^\frac{\alpha}{2}+|y-y^\prime|^\alpha\right), \\
		\left|\Delta_{s,y}\big[\partial^2_{IJ}\overline{\bm{F}}^a_{bc}\big(t,s,y,\left(\partial_I\bm{u}\right)_{|I|\leq 2r}(t,s,y),  \left(\partial_I\bm{u}\right)_{|I|\leq 2r}(s,s,y)\big)\cdot\partial_J\bm{u}^c_t(t,s,y)\big]\right|\leq C(R)\left(|s-s^\prime|^\frac{\alpha}{2}+|y-y^\prime|^\alpha\right),
	\end{array}
	\right.$
\end{enumerate}
where $\Delta_{s,y}\varphi(s,y):=|\varphi(s^\prime,y^\prime)-\varphi(s,y)|$.
\end{definition}

The conditions (a)-(e) in Definition \ref{Def: AppropriatePair} allow us to utilize the methodologies in Section \ref{sec:linear}-\ref{sec:nonlinear}, including the linearization method and the fixed-point argument, to study nonlocal fully nonlinear systems in a weighted space. The first three conditions (a)-(c) guarantee that the mapping $\bm{u}\longmapsto\bm{U}$, defined by $\bm{U}_s=\bm{L}_0\bm{U}+\bm{F}(\bm{u})-\bm{L}_0\bm{u}$, is well-defined. Moreover, with (d) and (e), we can prove that it is contractive.

Now, we are ready to show the extension of the well-posedness results for nonlocal systems in a weighted space. 
\begin{theorem} \label{WellposednessWeighted}
All well-posedness results for nonlocal systems in Subsections \ref{sec:linear}-\ref{sec:nonlinear} can be extended to the setting with weighted spaces defined in this subsection. Specifically, we have 
\begin{enumerate}
	\item If all coefficients of $\bm{L}$ defined in \eqref{eq:Loperator} belong to $\bm{\Omega}^{(\alpha)}_{[0,T]}$, $\bm{f}\in\bm{\Omega}^{{(\alpha)}}_{\varrho,[0,T]}$, and $\bm{g}\in\bm{\Omega}^{{(2r+\alpha)}}_{\varrho,[0,T]}$, then the nonlocal linear system \eqref{Nonlocal linear system} admits a unique solution $\bm{u}\in\bm{\Omega}^{{(2r+\alpha)}}_{\varrho,[0,T]}$ in $\Delta[0,T]\times\mathbb{R}^d$. Moreover,   
	\begin{equation} \label{Weighted: Estimates of solutions of nonlocal system} 
		\| \bm{u}\|^{(2r+\alpha)}_{\varrho,[0,T]}\leq C\left(\| \bm{f}\|^{(\alpha)}_{\varrho,[0,T]}+\| \bm{g}\|^{(2r+\alpha)}_{\varrho,[0,T]}\right). 
	\end{equation}  
	\item Suppose that the pair of $(\bm{F},\bm{g})$ is appropriate in the sense of Definition \ref{Def: AppropriatePair}. Then, there exist $\tau>0$ and a unique maximally-defined solution $\bm{u}\in\bm{\Omega}^{{(2r+\alpha)}}_{\varrho,[0,\tau]}$ satisfying \eqref{Nonlocal fully nonlinear system} in $\Delta[0,\tau]\times\mathbb{R}^d$. 
	\item Assume further that $\| \bm{u}\|^{(2r+\alpha^\prime)}_{\varrho,[0,\sigma]}\leq M$ for some finite constant $M>0$ across all $\sigma\in[0,\tau)$, then either the pair of $(\bm{F},\lim_{s\to\tau}\bm{u}(\cdot,s,\cdot))$ is not appropriate or $\tau=T$. Consequently, the nonlocal quasilinear system \eqref{Nonlocal quasilinear system} is globally solvable.  
\end{enumerate}
\end{theorem}

In this refined framework, although we allow the nonhomogeneous term $\bm{f}$ and the initial data $\bm{g}$ to increase exponentially in the spatial variable (more specifically, $\bm{f}\in\bm{\Omega}^{{(\alpha)}}_{\varrho,[0,T]}$ and $\bm{g}\in\bm{\Omega}^{{(2r+\alpha)}}_{\varrho,[0,T]}$), it is still required that all coefficients of $\bm{L}$ defined in \eqref{eq:Loperator} belong to $\bm{\Omega}^{(\alpha)}_{[0,T]}$. 
This also induces the condition (b) of Definition \ref{Def: AppropriatePair} that requires that $\partial_I\bm{F}$ and $\partial_I\overline{\bm{F}}$ at $\bm{u}$ belong to ordinary normed spaces instead of the weighted ones. Nevertheless, these conditions are necessary for our analyses, including the linearization method adopted in the later analysis of solvability of nonlocal nonlinear systems in the weighted spaces, and satisfied by our first financial example in Section \ref{Analysis of equilbirum HJB systems}.}

\begin{remark}\label{remark:relax}
[\textbf{Potential relaxations}] Echoing our discussion at the end of Section \ref{sec:relation}, we establish the first analytical framework for such emerging type of nonlocal PDEs/systems with a two-time-variable structure. Hence, we have to require some additional conditions and regularities for the nonlocal setting to support our proofs. From this perspective, one may work towards lifting the restrictions so as to embrace a larger class of nonlocal systems. 
{\color{black}It should be noticed that the main restrictions originate from the limitations on coefficients of the nonlocal linear operator $\bm{L}$ defined in \eqref{eq:Loperator}. 
}
Here, we list some promising future extensions. We first note that this paper adopts a concept of strongly ellipticity conditions \eqref{Ellipticity 1}-\eqref{Ellipticity 2}, \eqref{Uniform ellipticity condition of F 1}-\eqref{Uniform ellipticity condition of F 2}, {\color{black}and (c) in Definition \ref{Def: AppropriatePair}}. Following \cite{Eidelman1969,Ladyzhanskaya1968,Friedman1964}, the conditions can be substituted so as to contain a larger class of parabolic systems (in the sense of Petrowski-type). Moreover, we may substantially improve the results by modifying the underlying norms and spaces {\color{black} to obtain a more general PDE theory which allows degenerate coefficients.}  
\end{remark}  


\section{Well-posedness of Equilibrium HJB Systems and Examples} \label{Analysis of equilbirum HJB systems}
The previous two sections have established the linkage between the TIC SDGs and equilibrium HJB systems and the well-posedness of nonlocal parabolic systems that nest the equilibrium HJB systems. This section intends to summarize our results in the context of TIC SDGs and provide {\color{black}two examples} (in finance), the induced nonlocal fully nonlinear {\color{black}systems} of which {\color{black}are} globally solvable under some technical conditions in  Proposition{\color{black}s \ref{Exp1Propostion} and} \ref{Solvability of the financial example}. 


Let us denote $\dot{\bm{H}}(t,s,y,z)=\bm{H}(T-t,T-s,y,z)$ defined in \eqref{Equilibrium HJB system}, where $z=(u,p,q,l,m,n)$, and $\dot{\bm{g}}(t,y)=\bm{g}(T-t,y)$. Then we have the following theorem, which follows directly Proposition \ref{Equivalance between forward probelms and backward problems} and Theorems \ref{Well-posedness of fully nonlinear system} and \ref{Wellposedness of Quasilinear systems}. 
\begin{theorem} \label{Solvability of HJB system}
For any fixed $T>0$, suppose that $\dot{\bm{g}}\in\bm{\Omega}^{(2+\alpha)}_{[0,T]}$ and assume that $\dot{\bm{H}}$ and $\dot{\bm{g}}$ are regular enough in the sense that $\dot{\bm{H}}$ satisfies the conditions of \eqref{Uniform ellipticity condition of F 1}-\eqref{Lipschitz continuity of F} with the open ball $B\big(\dot{\bm{g}},R_0\big)$ containing the range of $\big(\left(\dot{\bm{g}}\right)_{|I|\leq 2}(t,y),\left(\dot{\bm{g}}\right)_{|I|\leq 2}(s,y)\big)$ for any $(t,s)\in\Delta[0,T]$ and some radius $R_0>0$. Then, for such a TIC SDG problem \eqref{State equation}-\eqref{Backward SDE} with \eqref{Cost functional}, we have that
\begin{enumerate}
\item in the case of that both the drift and the diffusion of \eqref{State equation} are controlled, there exist $\tau\in(0,T]$ and a unique maximally defined solution $\bm{u}\in\bm{\Omega}^{{(2+\alpha)}}_{[T-\tau,T]}$ satisfying the equilibrium HJB system \eqref{Equilibrium HJB system} in $\nabla[T-\tau,T]\times\mathbb{R}^d$. Moreover, whenever the domain of $\dot{\bm{H}}$ is large enough and \eqref{Upper bounded for global existence} holds, we have $\tau=T$; 
\item in the case of that only the drift is controlled while the diffusion of \eqref{State equation} is uncontrolled, i.e. $\sigma(s,y,a)=\sigma(s,y)$, the equilibrium HJB system \eqref{Equilibrium HJB system} admits a unique global solution $\bm{u}\in\bm{\Omega}^{{(2+\alpha)}}_{[0,T]}$ in $\nabla[0,T]\times\mathbb{R}^d$.
\end{enumerate}
\end{theorem}

The regularity requirements of $\dot{\bm{H}}$ and $\dot{\bm{g}}$ in Theorem \ref{Solvability of HJB system} characterize the ``needed regularity" of $\phi$ in Assumption \ref{assumption}. By the implied existence of the NE point at each time point from Assumption \ref{assumption},
it is clear that the Hamiltonian $\dot{\bm{H}}$ satisfies \eqref{Uniform ellipticity condition of F 1}, \eqref{Holder continuity of F}, and \eqref{Lipschitz continuity of F}. We only need to check if $\dot{\bm{H}}$ and $\dot{\bm{g}}$ jointly satisfy the condition \eqref{Uniform ellipticity condition of F 2}. Especially for a nonlocal quasilinear system (i.e. a TIC SDG with controls on drift only), it is easy to see that conditions \eqref{Uniform ellipticity condition of F 1}-\eqref{Uniform ellipticity condition of F 2} hold. Hence, with smooth enough coefficients in \eqref{State equation}-\eqref{Cost functional}, the corresponding equilibrium HJB system satisfies the requirements in Theorem \ref{Solvability of HJB system}. Finally, based on the conjectures in Section \ref{sec:relation}, we may conclude that the associated TIC SDG admits a TC-NE point $\overline{\bm{\alpha}}=\bm{\phi}\big(s,s,y,\left(\bm{u}\right)_{|I|\leq 2}(s,s,y)\big)$ and the TC-NE value function $\bm{V}(s,y)=\bm{u}(s,s,y)$ for $(s,y)\in[T-\tau,T]\times\mathbb{R}^d$. {\color{black}Analogously, Theorem \ref{WellposednessWeighted} supports that all results in Theorem \ref{Solvability of HJB system} still hold in a weighted space.} 


{\color{black}
\subsection{Financial Examples}
In this subsection, we provide two examples of TIC SDGs among $m$ players on $[0,T]$ for an arbitrary positive integer $m\geq 1$ and an arbitrary large time $T>0$, the first one of which studies how the investors (players) choose their own optimal investment strategies to increase their own exponential utility and the second one of which studies the optimal investment and consumption strategy pairs of investors that optimize their own power utility. These two examples showcase two different uses of our well-posedness results in Subsection \ref{Sec:WeightedNormsSpaces}. The first example satisfies all the conditions in our refined (weighted-norm) framework and thus the corresponding TIC problem is globally solvable over the whole time horizon $\nabla[0,T]$. The second example is, however, not covered by our framework due to its degeneracy property, but this interesting and relevant example is still considered here while it shows the necessity of extending our analytic framework from the non-degenerate setting to a degenerate one; see also the discussion in Remark \ref{remark:relax}. Though the second example is beyond our general framework, we make another problem-specific attempt in the similar spirit of the proof of Theorem \ref{Global existence of fully nonlineaity}. Specifically, with some suitable ansatzs of solutions, both examples admit explicit expressions of \eqref{Equilibrium strategy} and \eqref{Equilibrium HJB system} while the latter can be further reduced to ordinary differential equation (ODE) systems. We can then show the global solvability of these ODE systems and thus we obtain the global solvability of the two listed examples.


We first introduce the general setup for the two examples.} Suppose that the $m$ players have similar interests, e.g., a risk investment fund managed by $m$ investors (managers). To maintain and increase their own utility, each investor needs to study their own optimal investment and consumption strategy. Consider a market model in which there are one bond with the riskless interest rate $r>0$ and some risky assets while the $a$-th investor has estimated the appreciation rate of his/her return on investment (ROI) by $\mu_a>r$ and its volatility by $\sigma_a>0$. Further assume that the $m$ investors' ROIs are uncorrelated, the yield rate vector $\bm{P}\in\mathbb{R}^m$ of the $m$ investors is characterized by 
\begin{equation*}
d\bm{P}(s)=\mathrm{diag}\{\bm{P}(s)\}(\bm{\mu} ds+\bm{\sigma} d\bm{W}(s))
\end{equation*} 
where $\mathrm{diag}\{\bm{P}(s)\}$ is a diagonal matrix with main diagonal elements of $\bm{P}(s)$, $\bm{\mu}=(\mu_1,\mu_2,\cdots,\mu_m)^\top$, and $\bm{\sigma}=\mathrm{diag}\{\sigma_1,\sigma_2,\cdots,\sigma_m\}$. Denoted by $\bm{\alpha}^a(\cdot)$ the dollar amounts managed by the $a$-th investor and $\{X(t)\}_{t\in[0,T]}$ the aggregated wealth process (from this perspective, we are similarly considering a problem by a fund of funds), we can obtain the following FSDE for $\{X(t)\}_{t\in[0,T]}$: 
\begin{equation} \label{Example: Forward process} 
\left\{
\begin{array}{lr}
dX(s)=\big[rX(s)+(\bm{\mu}-r\mathbb{1})^\top\bm{\alpha}(s)-\mathbb{1}^\top \bm{c}(s)\big]ds+\bm{\alpha}^\top(s)\bm{\sigma} d\bm{W}(s), \quad t\leq s\leq T, \\
X(t)=y,\quad 0\leq t \leq T,\quad {\color{black}y\in\mathbb{X}}, 
\end{array}
\right. 
\end{equation} 
where $\mathbb{1}=(1,\ldots,1)^\top\in\mathbb{R}^m$ and $\bm{c}=(\bm{c}^1,\ldots,\bm{c}^m)^\top$ with $\bm{c}^{a}(\cdot)$ the consumption rate of the $a$-th investor {\color{black}valued in $\mathbb{C}$, $\bm{\alpha}^a$ is valued in $\mathbb{A}$, and $(\mathbb{A},\mathbb{C},\mathbb{X})$ will be specified for our examples. For any fixed $t\in[0,T]$, the admissible set of investment-consumption strategy pairs is then defined as the set of progressively measurable processes $\{(\bm{\alpha}(s),\bm{c}(s))\}_{s\ge t}$ such that $\bm{\alpha}^a(s)\in \mathbb{A}$ and $\bm{c}^a(s)\in \mathbb{C}$ for $a=1,\ldots,m$ and that the FSDE \eqref{Example: Forward process} has a strong solution $\{X(s)\}_{s\ge t}$} with $X(s)\in\mathbb{X}$, $\mathbb{P}$-a.s., for $s\ge t$.

Next, to characterize the investors' preferences, let $(\bm{Y}(\cdot),\bm{Z}(\cdot))$ be the adapted solution to the following BSDE:  
{\color{black}\begin{equation} \label{ExampleBackprocess}
\left\{
\begin{array}{lr}
	d\bm{Y}(s)=-\mathbb{h}(t,s,X(s),\bm{\alpha}(s),\bm{c}(s),\bm{Y}(s))ds+\bm{Z}(s)d\bm{W}(s), \quad t\leq s\leq T, \\
	\bm{Y}(T)=\mathbb{g}(t,X(T)),\quad 0\leq t \leq T, 
\end{array}
\right. 
\end{equation} 
where the generator $\mathbb{h}$ and terminal condition $\mathbb{g}$ are both deterministic $\mathbb{R}^m$-valued functions, and they will be specified in the study of different utility problems.}  
Then, we define the recursive utility functional of the $a$-th investor for $a=1,\ldots,m$ as follows: 
\begin{equation*}
\bm{J}^{a}(t,y;\bm{\alpha},\bm{c}):=\bm{Y}^{a}(t;t,y,\bm{\alpha},\bm{c}). 
\end{equation*} 
Consequently, the problem of maximizing $\bm{J}^{a}(t,y;\bm{\alpha},\bm{c})$ for $a=1,\ldots,m$ is a TIC SDG since {\color{black}$\mathbb{h}$ and $\mathbb{g}$ both} depend on the initial time point $t$. Note that in Section \ref{sec:sdg}, we illustrate with an SDG with minimization while it is equivalent to considering maximization. It is noteworthy that the $a$-th functional $\bm{J}^a$ is a Uzawa-type differential utility being not only recursive (in the sense that it depends on $\bm{Y}^a$ itself) but also dependent on other investors' utility functionals $\bm{Y}^{-a}$. It is sensible because enormous experiments in behavioral economics/finance show that people's assessment on their wellbeing is relative rather than absolute. Moreover, the controlled FBSDEs of the $m$ investors are coupled together through {\color{black}$(\bm{\alpha},\bm{c})$} and $\bm{Y}$ in the BSDE and $(\bm{\alpha},\bm{c})$ in the FSDE. Furthermore, compared to the existing literature on the well-posedness results, we allow the diffusion of the wealth process $X$ to be controlled.


{\color{black}
\subsubsection{TIC Merton Problem with Exponential Utility and Zero Consumption} \label{Sec: ExpU}
In the first example, we assume that 
\begin{equation} \label{ExpBSDE}
\left\{
\begin{array}{lr}
	\mathbb{A}=\mathbb{R}, \quad \mathbb{C}=\{0\}, \quad \mathbb{X}=\mathbb{R}, \\
	\mathbb{h}(t,s,X(s),\bm{\alpha}(s),\bm{c}(s),\bm{Y}(s))=-\bm{R}(t,s)\bm{Y}(s), \\
	\mathbb{g}(t,X(T))=-\bm{T}(t)\exp\{-\eta X(T)\}, \quad \eta>0, 
\end{array}
\right. 
\end{equation} 
where $\bm{R}$ and $\bm{T}$ are $\mathbb{R}^{m\times m}$- and $\mathbb{R}^m$-valued continuous and positive functions, respectively. Next, we will show that this example can be analyzed within our framework in Subsection \ref{Sec:WeightedNormsSpaces}. Specifically, in order to show the well-posedness of solutions to the TIC SDG \eqref{Example: Forward process}-\eqref{ExpBSDE}, the main steps are listed as follows: 
\begin{enumerate}
\item embed the original TIC SDG problem \eqref{Example: Forward process}-\eqref{ExpBSDE} ($P$) into a family of problems $P_\gamma$ parameterized by $\gamma\geq 0$ such that $P_0=P$; 
\item prove the global well-posedness of solutions of $P_\gamma$ in the case of $\gamma>0$ such that the mapping from $\gamma\in(0,\infty)$ to the solution $\bm{U}_\gamma(\cdot,\cdot,\cdot)\in\bm{\Omega}^{{(2+\alpha)}}_{\varrho,[0,T]}$ is well-defined; 
\item show that $\gamma\longmapsto\bm{U}_\gamma(\cdot,\cdot,\cdot)$ admits a unique analytic continuation at $\gamma=0$ such that the problem $P$ (i.e. $P_0$) has global existence and uniqueness of solutions as well.  
\end{enumerate}  
These three steps not only show the global well-posedness of solutions of nonlocal HJB system and the TIC SDG but also give explicit representations for equilibrium strategies \eqref{Equilibrium strategy} and equilibrium value functions of \eqref{Equilibrium HJB system}. We cannot directly analyze $P_0$ as its nonlinearity is not regular enough and thus we parametrize the problem such that the nonlinearity of $P_\gamma$ with $\gamma>0$ satisfies the regularity conditions in our framework. 

First of all, it is more convenient to consider a transformed state $X(s)\exp\{r(T-s)\}$ for the dynamics $X(s)$ of \eqref{Example: Forward process} before our analyses. By Corollary 5.6 of \cite{Yong1999}, it is clear that 
\begin{equation*} 
\left\{
\begin{array}{lr}
	dX(s)\exp\{r(T-s)\}=(\bm{\mu}-r\mathbb{1})^\top\exp\{r(T-s)\}\bm{\alpha}(s)ds+\bm{\alpha}^\top(s)\bm{\sigma}\exp\{r(T-s)\} d\bm{W}(s), \quad t\leq s\leq T, \\
	X(t)\exp\{r(T-t)\}=y\exp\{r(T-t)\},\quad 0\leq t \leq T,\quad y\in\mathbb{R}. 
\end{array}
\right. 
\end{equation*}
Consequently, without loss of generality, let us consider a modified version of \eqref{Example: Forward process} with the following form:
\begin{equation} \label{Example: Modified forward process} 
\left\{
\begin{array}{lr}
	dX(s)=\widehat{\bm{\mu}}^\top(s)\bm{\alpha}(s)ds+\bm{\alpha}^\top(s)\widehat{\bm{\sigma}}(s) d\bm{W}(s), \quad t\leq s\leq T, \\
	X(t)=y,\quad 0\leq t \leq T,\quad y\in\mathbb{R}, 
\end{array}
\right. 
\end{equation} 
where $\widehat{\bm{\mu}}(s)=(\bm{\mu}-r\mathbb{1})^\top\exp\{r(T-s)\}$ and $\widehat{\bm{\sigma}}(s)=\bm{\sigma}\exp\{r(T-s)\}$. 

\ \ 

\underline{\textbf{Step 1: A family of parameterized problems $\bm{P_\gamma}$.}} Next, let us consider a family of problems (BSDEs) parameterized by an external parameter $\gamma\geq 0$, 
\begin{equation} \label{ParaExpBSDE}
\left\{
\begin{array}{lr}
	\bm{c}(s)=0, \\
	\mathbb{h}_\gamma(t,s,X(s),\bm{\alpha}(s),\bm{c}(s),\bm{Y}(s))=\gamma(\bm{w}^\top_1(t,s,X(s))\otimes\mathbb{1})\bm{\alpha}(s)-(\bm{w}^\top_2(t,s,X(s))\otimes\mathbb{1})\left(\bm{\alpha}(s)\odot\bm{\alpha}(s)\right)-\bm{w}_3(t,s)\bm{Y}(s), \\
	\mathbb{g}_\gamma(t,X(T))=\gamma\bm{g}_1(t)\exp\{\eta X(T)\}-\bm{g}_2(t)\exp\{-\eta X(T)\}, \quad \eta>0,  
\end{array}
\right. 
\end{equation} 
where $\otimes$ denotes the Kronecker product, $\odot$ denotes the Hadamard product, $\bm{w}_3(t,s)=\bm{R}(t,s)$, $\bm{g}_2(t)=\bm{T}(t)$, and $\bm{w}_1(t,s,y)$, $\bm{w}_2(t,s,y)$, $\bm{g}_1(t)$, and $\bm{g}_2(t)$ are all $\mathbb{R}^m$-valued continuous functions that will be specified later. It is clear that \eqref{ParaExpBSDE} reduces to \eqref{ExpBSDE} when $\gamma=0$ and $\bm{w}_2=0$. Our later specification will also parametrize $\bm{w}_1$ and $\bm{w}_2$ with $\gamma$ and $\bm{w}_2\equiv 0$ when $\gamma=0$. Thus in this case, \eqref{ParaExpBSDE} is actually parameterized by a single parameter $\gamma$.

It is noteworthy that there are multiple embedding schemes
while any of them can work out the well-posedness of solutions of the problem \eqref{Example: Forward process}-\eqref{ExpBSDE} as long as the mapping from $\gamma$ to the solution of $P_\gamma$ is well-defined and is at least Cauchy-continuous at the point that reduces the parametrized problem to $P$. Moreover, the fact about whether the problem is well-posed is free of the choice of the embedding scheme. We will show that the embedding \eqref{ParaExpBSDE} with \eqref{Expassumptions} facilitate \textbf{Step 2} and \textbf{Step 3}. The relationship between parameterized data and solutions was discussed in the earlier stability analysis of nonlocal systems; see Remark \ref{NonlinearSA}.

\ \ 

\underline{\textbf{Step 2. The well-definedness of $\bm{\gamma\longmapsto\bm{U}_\gamma(\cdot,\cdot,\cdot)}$ with $\bm{\gamma>0}$.}} According to the definitions of $(X(\cdot),\bm{Y}(\cdot),\bm{Z}(\cdot))$ formulated by controlled FBSDEs \eqref{Example: Modified forward process}-\eqref{ParaExpBSDE}, the Hamiltonian system of the $m$ players has the form: for $a=1,\ldots,m$,
\begin{equation*}
\begin{split}
	\bm{\mathcal{H}}^a_\gamma(t,s,y,\alpha,u,p,q) & =\frac{1}{2}\left(\sum\limits_{1\leq b\leq m}(\widehat{\sigma}_b(s)\alpha^b)^2\right) q^a+\left(\sum\limits_{1\leq b\leq m}\widehat{\mu}_b(s)\alpha^b\right)p^a \\
	&  
	\quad
	+\sum\limits_{1\leq b\leq m}\gamma\bm{w}^{b}_1(t,s,y)\alpha^b 
	-\sum\limits_{1\leq b\leq m}\bm{w}^{b}_2(t,s,y)(\alpha^b)^2-\sum\limits_{1\leq b\leq m}\bm{w}^{ab}_3(t,s) u^b,
\end{split}
\end{equation*}
Maximizing the above with respect to $\alpha^a$ with fixed $\alpha^{-a}$, $p>0$, and $q<0$ yields 
\begin{equation*}
\overline{\bm{\alpha}}^a=\frac{\gamma\bm{w}^{a}_1(t,s,y)+\widehat{\mu}_a(s) p^a}{2\bm{w}^{a}_2(t,s,y)-\widehat{\sigma}_a^2(s) q^a}:=\frac{\widehat{\bm{w}}^{a}_1(t,s,y)+(\mu_a-r) p^a}{\widehat{\bm{w}}^{a}_2(t,s,y)-\sigma_a^2 q^a}\exp\{-r(T-s)\}, \qquad a=1,\ldots,m.   
\end{equation*} 
where $\widehat{\bm{w}}^{a}_1(t,s,y):=\gamma\bm{w}^{a}_1(t,s,y)\exp\{-r(T-s)\}$ and $\widehat{\bm{w}}^{a}_2(t,s,y):=2\bm{w}^{a}_2(t,s,y)\exp\{-2r(T-s)\}$. Thus, eventually, the equilibrium strategy will be given by
\begin{equation} \label{Example 1: Closed loop strategy} 
\overline{\bm{\alpha}}^a(s,y)=\frac{\widehat{\bm{w}}^{a}_1(s,s,y)+(\mu_a-r) \bm{U}^a_y(s,s,y)}{\widehat{\bm{w}}^{a}_2(s,s,y)-\sigma_a^2 \bm{U}^a_{yy}(s,s,y)}\exp\{-r(T-s)\}
\end{equation}
with $\bm{U}(t,s,y)=(U^1(t,s,y),\cdots,U^m(t,s,y))$ ($\gamma$ is suppressed) being the solution to an equilibrium HJB system of the form: 
\begin{equation} \label{Backward HJB Exponential utility}
\left\{
\begin{array}{l}
	\bm{U}^a_s(t,s,y)+\frac{1}{2}\sum\limits_{1\leq b\leq m}\left(\frac{\sigma_b\widehat{\bm{w}}^{b}_1(s,s,y)+\sigma_b(\mu_b-r) \bm{U}^b_y(s,s,y)}{\widehat{\bm{w}}^{b}_2(s,s,y)-\sigma_b^2 \bm{U}^b_{yy}(s,s,y)}\right)^2 \bm{U}^a_{yy}(t,s,y) +\sum\limits_{1\leq b\leq m}\left(\frac{(\mu_b-r)\widehat{\bm{w}}^{b}_1(s,s,y)+(\mu_b-r)^2 \bm{U}^b_y(s,s,y)}{\widehat{\bm{w}}^{b}_2(s,s,y)-\sigma_b^2 \bm{U}^b_{yy}(s,s,y)}\right)\bm{U}^a_y(t,s,y) \\ \qquad     
	+\exp\{-r(T-s)\}\sum\limits_{1\leq b\leq m}\gamma\bm{w}^{b}_1(t,s,y)\left(\frac{\widehat{\bm{w}}^{b}_1(s,s,y)+(\mu_b-r) \bm{U}^b_y(s,s,y)}{\widehat{\bm{w}}^{b}_2(s,s,y)-\sigma_b^2 \bm{U}^b_{yy}(s,s,y)}\right)-\exp\{-2r(T-s)\}\sum\limits_{1\leq b\leq m}\bm{w}^{b}_2(t,s,y)\left(\frac{\widehat{\bm{w}}^{b}_1(s,s,y)+(\mu_b-r) \bm{U}^b_y(s,s,y)}{\widehat{\bm{w}}^{b}_2(s,s,y)-\sigma_b^2 \bm{U}^b_{yy}(s,s,y)}\right)^2 \\
	\qquad 
	-\sum\limits_{1\leq b\leq m}\bm{w}^{ab}_3(t,s) \bm{U}^b(t,s,y) = 0, \\
	\bm{U}(t,T,y)=\gamma\bm{g}_1(t)\exp\{\eta y\}-\bm{g}_2(t)\exp\{-\eta y\},\quad 0\leq t\leq s \leq T, \quad y\in\mathbb{R}, \quad a=1,\ldots,m. 
\end{array}
\right. 
\end{equation}  
It is equivalent to solving the following forward problem:  
\begin{equation} \label{HJB Exponential utility}
\left\{
\begin{array}{l}
	\bm{U}^a_s(t,s,y)=\frac{1}{2}\sum\limits_{1\leq b\leq m}\left(\frac{\sigma_b\widehat{\bm{w}}^{b}_1(T-s,T-s,y)+\sigma_b(\mu_b-r) \bm{U}^b_y(s,s,y)}{\widehat{\bm{w}}^{b}_2(T-s,T-s,y)-\sigma_b^2 \bm{U}^b_{yy}(s,s,y)}\right)^2 \bm{U}^a_{yy}(t,s,y) +\sum\limits_{1\leq b\leq m}\left(\frac{(\mu_b-r)\widehat{\bm{w}}^{b}_1(T-s,T-s,y)+(\mu_b-r)^2 \bm{U}^b_y(s,s,y)}{\widehat{\bm{w}}^{b}_2(T-s,T-s,y)-\sigma_b^2 \bm{U}^b_{yy}(s,s,y)}\right)\bm{U}^a_y(t,s,y) \\ \qquad\qquad\qquad    
	+\exp\{-rs\}\sum\limits_{1\leq b\leq m}\gamma\bm{w}^{b}_1(T-t,T-s,y)\left(\frac{\widehat{\bm{w}}^{b}_1(T-s,T-s,y)+(\mu_b-r) \bm{U}^b_y(s,s,y)}{\widehat{\bm{w}}^{b}_2(T-s,T-s,y)-\sigma_b^2 \bm{U}^b_{yy}(s,s,y)}\right) \\
	\qquad\qquad\qquad  
	-\exp\{-2rs\}\sum\limits_{1\leq b\leq m}\bm{w}^{b}_2(T-t,T-s,y)\left(\frac{\widehat{\bm{w}}^{b}_1(T-s,T-s,y)+(\mu_b-r) \bm{U}^b_y(s,s,y)}{\widehat{\bm{w}}^{b}_2(T-s,T-s,y)-\sigma_b^2 \bm{U}^b_{yy}(s,s,y)}\right)^2 \\
	\qquad\qquad\qquad  
	-\sum\limits_{1\leq b\leq m}\bm{w}^{ab}_3(T-t,T-s) \bm{U}^b(t,s,y) = 0, \\
	\bm{U}(t,0,y)=\gamma\bm{g}_1(T-t)\exp\{\eta y\}-\bm{g}_2(T-t)\exp\{-\eta y\},\quad 0\leq s\leq t \leq T, \quad y\in\mathbb{R}, \quad a=1,\ldots,m. 
\end{array}
\right. 
\end{equation}  

Next, let us consider the partial derivatives of the nonlinearity of \eqref{HJB Exponential utility} with respect to its arguments in order to verify its regularities. The nonlinearity of \eqref{HJB Exponential utility} is denoted by $\mathbb{H}:=\mathbb{H}_\gamma(t,s,y,z)$. After a rather lenghty but staightforward calculation, we can obtain all partial derivatives in Table \ref{tab:table1} and \ref{tab:table2}, which are all listed in Appendix \ref{PDNonLinearity}. Consequently, for the initial condition $\bm{U}(t,0,y)=\gamma\bm{g}_1(T-t)\exp\{\eta y\}-\bm{g}_2(T-t)\exp\{-\eta y\}$, one can verify that the pair of $(\mathbb{H},\bm{U}(t,0,y))$ is appropriate in the sense of Definition \ref{Def: AppropriatePair} for some suitable $\widehat{\bm{w}}_1$ and $\widehat{\bm{w}}_2$; see \eqref{Expassumptions}. Hence, our well-posedness results in Subsection \ref{Sec:WeightedNormsSpaces} promise that there exist $\delta\in(0,T]$ and a unique solution satisfying \eqref{HJB Exponential utility} in $\Delta[0,\delta]\times\mathbb{R}$. Equivalently, the backward problem \eqref{Backward HJB Exponential utility} is solvable as well in $\nabla[T-\delta,T]\times\mathbb{R}$.

In order to find an explicit solution to \eqref{Backward HJB Exponential utility} and show its global solvability in the whole time horizon $\nabla[0,T]$, we consider the following ansatz: 
\begin{equation} \label{Expansatz}
\bm{U}(t,s,y)=\bm{\varphi}_1(t,s)\exp\{\eta y\}-\bm{\varphi}_2(t,s)\exp\{-\eta y\}, \quad (t,s)\in\nabla[T-\delta,T], 
\end{equation}
for some suitable $\bm{\varphi}_1(\cdot,\cdot)$, and $\bm{\varphi}_2(\cdot,\cdot)$. Then we have $\bm{\varphi}_1(t,T)=\gamma\bm{g}_1(t)$ and $\bm{\varphi}_2(t,T)=\bm{g}_2(t)$. Furthermore, let us assume that 
\begin{equation} \label{Expassumptions}
\left\{
\begin{array}{rcl}
	\widehat{\bm{w}}^a_1(t,s,y)=\gamma\bm{w}^{a}_1(t,s,y)\exp\{-r(T-s)\}=\gamma \bm{W}^a_1(t,s)\exp\{\eta y\}\exp\{-r(T-s)\}, \\
	\widehat{\bm{w}}^a_2(t,s,y)=2\bm{w}^a_2(t,s,y)\exp\{-2r(T-s)\}=\frac{\sigma^2_a\eta}{\mu_a-r}\widehat{\bm{w}}^a_1(t,s,y)+2\sigma^2_a\eta^2\bm{\varphi}^a_1(t,s)\exp\{\eta y\},   
\end{array}
\right. 
\end{equation}
where $\bm{W}^a_1(t,s)$ is a given continuously differentiable and positive function.
Under the assumptions of \eqref{Expassumptions}, we have $\overline{\bm{\alpha}}^a(s,y)=\frac{1}{\eta}\frac{\mu_a-r}{\sigma^2_a}\exp\{-r(T-s)\}$ for $a=1,2,\cdots,m$. Subsequently, by simple calculation, $\bm{\varphi}_1(\cdot,\cdot)$ and $\bm{\varphi}_2(\cdot,\cdot)$ solve the following ODE systems, respectively, 
\begin{equation} \label{Exp2ODESys}
\left\{
\begin{array}{l}
	(\bm{\varphi}_1)_s(t,s)+\bm{N}_1(t,s)\bm{\varphi}_1(t,s)+\bm{M}_1(t,s) = 0, \\
	\bm{\varphi}_1(t,T)=\gamma\bm{g}_1(t),\quad (t,s)\in\nabla[T-\delta,T], 
\end{array} 
\right. \qquad \qquad 
\left\{
\begin{array}{l}
	(\bm{\varphi}_2)_s(t,s)+\bm{N}_2(t,s)\bm{\varphi}_2(t,s) = 0, \\
	\bm{\varphi}_2(t,T)=\bm{g}_2(t),\quad (t,s)\in\nabla[T-\delta,T], 
\end{array} 
\right. 
\end{equation}
where $\bm{N}_1(t,s)=\mathrm{diag}\left\{\sum_b\frac{3(\mu_b-r)^2}{2\sigma^2_b},\cdots,\sum_b\frac{3(\mu_b-r)^2}{2\sigma^2_b}\right\}-\left(\frac{(\mu_1-r)^2}{\sigma^2_1},\cdots,\frac{(\mu_m-r)^2}{\sigma^2_m}\right)\otimes\mathbb{1}-\bm{w}_3(t,s)$ and $\bm{M}_1(t,s):=\gamma\overline{\bm{M}}_1(t,s)=\gamma\left[\left(\frac{(\mu_1-r)\exp\{-r(T-s)\}}{2\sigma^2_1\eta},\cdots,\frac{(\mu_m-r)\exp\{-r(T-s)\}}{2\sigma^2_m\eta}\right)\otimes\mathbb{1}\right]\bm{W}_1$. Moreover, $\bm{N}_2(t,s)=-~\mathrm{diag}\left\{\sum_b\frac{(\mu_b-r)^2}{2\sigma^2_b},\cdots,\sum_b\frac{(\mu_b-r)^2}{2\sigma^2_b}\right\}-\bm{w}_3(t,s)$. By the classical theory of ODE systems, systems \eqref{Exp2ODESys} admit a unique solution for $(t,s)\in\nabla[T-\delta,T]$ since $\sup_{t,s}|\bm{N}_i(t,s)|$ ($i=1,2$) are both bounded. Furthermore, the ansatz solution \eqref{Expansatz} of \eqref{Backward HJB Exponential utility} can be represented by  
\begin{equation} \label{ExpSolODE}
\bm{U}(t,s,y)=  \left[\gamma\bm{\digamma}_1(t,s)\bm{\digamma}^{-1}_1(t,T)\bm{g}_1(t)+\gamma\int^T_s\bm{\digamma}_1(t,s)\bm{\digamma}^{-1}_1(t,\tau)\overline{\bm{M}}_1(t,\tau)d\tau\right]\exp\{\eta y\} -\bm{\digamma}_2(t,s)\bm{\digamma}^{-1}_2(t,T)\bm{g}_2(t)\exp\{-\eta y\},
\end{equation}
where $\bm{\digamma}_i(t,s)$ is the fundamental matrix of the $i$-th ODE system of \eqref{Exp2ODESys}, $\bm{\digamma}^{-1}_i(t,s)$ the associated inverse matrix, and 
\begin{equation} \label{FMF}
\bm{\digamma}_i(t,s)=\bm{I}+\int^T_s\bm{N}_i(t,\tau)d\tau+\int^T_s\bm{N}_i(t,\tau)\int^T_\tau \bm{N}_i(t,\sigma)d\sigma d\tau + \cdots, \quad i=1,2,  
\end{equation}
in which $\bm{I}$ is $m\times m$ identity matrix. Note that \eqref{FMF} converges absolutely for every $s\in[T-t,T]$ and uniformly on every compact interval in $[T-t,T]$. In particular for $i=1,2$, if the matrix $\bm{N}_i(t,s)$ satisfies the Lappo--Danilevskii condition (see Remark \ref{LD condition}), then $\bm{\digamma}_i(t,s)=\exp\left\{\int^T_s\bm{N}_i(t,\tau)d\tau\right\}$.

Note that $\bm{U}(t,s,y)$ of \eqref{ExpSolODE} does not blow-up at $s=T-\delta$ for any $t\in[0,T-\delta]$ such that we can update a new terminal condition at $s=T-\delta$. Furthermore, thanks to \eqref{Expassumptions}, the uniformly elliptic conditions, the locally Lipschitz and H\"{o}lder continuity still hold within a small open ball centered at the range of the updated data. Consequently, one can repeat indefinitely the solving procedure up to a global solution for \eqref{Backward HJB Exponential utility} over $\nabla[0,T]$. With our well-posedness results of nonlocal systems in Subsection \ref{Sec:WeightedNormsSpaces}, we show that the mapping from the parameter $\gamma>0$ into the solution of \eqref{Backward HJB Exponential utility}, i.e. $\bm{U}_\gamma(t,s,y):=\bm{U}(t,s,y)$, is well-defined. 

\ \ 

\underline{\textbf{Step 3. Analytic continuation of $\bm{\gamma\longmapsto\bm{U}_\gamma(\cdot,\cdot,\cdot)}$ at $\bm{\gamma=0}$.}} For the original problem $P_0$, i.e. $\gamma=0$, our well-posedness results of nonlocal systems are not feasible even in a small time interval since the locally Lipschitz and H\"{o}lder continuity conditions are violated for some derivatives of the nonlinearity $\mathbb{H}$. However, for any fixed $\gamma>0$, we have shown that the mapping $\gamma\longmapsto\bm{U}_\gamma(\cdot,\cdot,\cdot)$ is well-defined and has an explicit formula \eqref{ExpSolODE}. From which, we can easily see that the mapping $\gamma\longmapsto\bm{U}_\gamma(\cdot,\cdot,\cdot)$ is at least uniformly continuous in $\gamma$ and thus we can extend it at $\gamma=0$ uniquely. Consequently, we can obtain the unique solution for $P_0$ in $\nabla[0,T]$, $\bm{U}(t,s,y)=-\bm{\digamma}_2(t,s)\bm{g}_2(t)\exp\{-\eta y\}$. Furthermore, the closed-loop TC-NE point and the corresponding TC-NE value function of the TIC SDG \eqref{Example: Forward process}-\eqref{ExpBSDE} have the following explicit representations: 
\begin{equation} \label{ExpSol}
\overline{\bm{\alpha}}^a(s,y)=\frac{1}{\eta}\frac{(\mu_a-r)}{\sigma_a^2}\exp\{-r(T-s)\}, \quad \bm{V}(s,y)=-\bm{\digamma}_2(s,s)\bm{T}(s)\exp\{-\eta y\exp\{r(T-s)\}\}, \quad (s,y)\in[0,T]\times\mathbb{R}, 
\end{equation}
by noting also the relationship between \eqref{Example: Forward process} and \eqref{Example: Modified forward process}.

Indeed, by directly making an ansatz $\bm{U}(t,s,y)=-\bm{\varphi}_2(t,s)\exp\{-\eta y\}$ for the solution of $P_0$ (i.e. \eqref{Example: Modified forward process}-\eqref{ExpBSDE}), we can still obtain the same explicit solution \eqref{ExpSol}. However, as we stressed before, our well-posedness results do not cover the problem $P_0$ since \eqref{Expassumptions} and \eqref{Exp2ODESys} induce $\widehat{\bm{w}}_1=\widehat{\bm{w}}_2=\bm{\varphi}_1=0$ when $\gamma=0$. Hence, it is necessary to embed $P_0$ into a family of problem $P_\gamma$ ($\gamma\geq 0$). 

Finally, let us summarize our results in the following proposition. 
\begin{proposition} \label{Exp1Propostion}
Suppose that $\bm{R}$ and $\bm{T}$ are continuously differentiable, then the TIC SDG \eqref{Example: Forward process}-\eqref{ExpBSDE} admits a unique solution in $\nabla[0,T]$, and the closed-loop TC-NE point and the TC-NE value function are given in \eqref{ExpSol}.   
\end{proposition}

\begin{remark} \label{LD condition}
The matrix $\bm{N}_i(t,s)$ satisfies the Lappo-Danilevskii condition, which means that it commutes with its integral, i.e.  $\bm{N}_i(t,s)\cdot\int^T_s\bm{N}_i(t,\tau)d\tau=\int^T_s\bm{N}_i(t,\tau)d\tau\cdot \bm{N}_i(t,s)$. Let us list four cases in which the condition holds, (1) $m=1$; (2) $\bm{w}(t,s)=\bm{w}(t)$; (3) $\bm{w}(t,s)$ is a diagonal matrix; (4) $\bm{N}(t,s)$ and $\bm{N}(t,\tau)$ commute for all $s$, $\tau$, and $t$. 
\end{remark}
}


{\color{black}
\subsubsection{TIC Merton Investment-Consumption Problem with Power Utility} \label{Sec: PowerU}
In our second example, we assume that 
\begin{equation} \label{PowerBSDE}
\left\{
\begin{array}{lr}
	\mathbb{A}=\mathbb{R}, \quad \mathbb{C}=[0,\infty), \quad \mathbb{X}=(0,\infty), \\
	\mathbb{h}(t,s,X(s),\bm{\alpha}(s),\bm{c}(s),\bm{Y}(s))=\bm{v}(t,s)\bm{c}^\beta(s)-\bm{w}(t,s)\bm{Y}(s), \\
	\mathbb{g}(t,X(T))=\bm{g}(t)X^\beta(T), \quad \beta>0, 
\end{array}
\right. 
\end{equation} 
where $\bm{v}$ and $\bm{w}$ are both $\mathbb{R}^{m\times m}$-valued functions and $\bm{g}$ is $\mathbb{R}^m$-valued continuous and positive function. In this case, each player $a$ ($a=1,2,\cdots,m$) needs to choose an investment and consumption strategy pair $(\bm{\alpha}^a(s),\bm{c}^a(s))$ valued in $\mathbb{R}\times[0,\infty)$ to optimize their own power utility. With a specific model, we can obtain explicit expressions of \eqref{Equilibrium strategy} and \eqref{Equilibrium HJB system} while the latter can be further reduced to an ordinary differential equation (ODE) system with an ansatz. In the similar spirit of the proof of Theorem \ref{Global existence of fully nonlineaity}, we can show the global solvability of the ODE system and thus we obtain the global well-posedness of \eqref{Equilibrium HJB system}. However, although our results are applicable to the state process with controlled drift and volatility, which is the case of this example, they do not cover this example due to its degeneracy property. Moreover, since the power utility function is defined over $(0,\infty)$, we also need the constraint that the solution $\{X(s)\}_{t,T}$ of \eqref{Example: Forward process} is almost surely nonnegative, i.e. $\mathbb{X}_s=(0,\infty)$. Such a constraint is not necessary for our first example since the domain of exponential utility function is $\mathbb{R}$.}

According to the definitions of $(X(\cdot),\bm{Y}(\cdot),\bm{Z}(\cdot))$ formulated by controlled FBSDEs \eqref{Example: Forward process}-\eqref{ExampleBackprocess} with \eqref{PowerBSDE}, the Hamiltonian system of the $m$ players has the form: for $a=1,\ldots,m$,
\begin{equation*}
\begin{split}
\bm{\mathcal{H}}^a(t,s,y,\alpha,c,u,p,q) & =\frac{1}{2}\left(\sum\limits_{1\leq b\leq m}(\sigma_b\alpha^b)^2\right) q^a+\left[ry+\sum\limits_{1\leq b\leq m}\left((\mu_b-r)\alpha^b-c^b\right)\right]p^a+\sum\limits_{1\leq b\leq m}\left(\bm{v}^{ab}(t,s) (c^b)^\beta-\bm{w}^{ab}(t,s) u^b\right),
\end{split}
\end{equation*}
where $\bm{v}^{ab}$ and $\bm{w}^{ab}$ represent the $(a,b)$-entry of matrices $\bm{v}$ and $\bm{w}$, respectively. Maximizing the above with respect to $\alpha^a$ and $c^a$ with fixed $\alpha^{-a}$, $c^{-a}$, $p>0$, and $q<0$ yields 
\begin{equation*}
\overline{\bm{\alpha}}^a=-\frac{(\mu_a-r)p^a}{\sigma_a^2 q^a}, \qquad \overline{\bm{c}}^a=\left(\frac{p^a}{\beta \bm{v}^{aa}(t,s)}\right)^\frac{1}{\beta-1}, \qquad a=1,\ldots,m.   
\end{equation*} 
Thus, eventually, the equilibrium strategy will be given by
\begin{equation} \label{Example: Closed loop strategy} 
\overline{\bm{\alpha}}^a(s,y)=-\frac{(\mu_a-r) \bm{U}^a_y(s,s,y)}{\sigma_a^2 \bm{U}^a_{yy}(s,s,y)}, \quad \overline{\bm{c}}^a(s,y)=\left(\frac{\bm{U}^a_y(s,s,y)}{\beta \bm{v}^{aa}(s,s)}\right)^\frac{1}{\beta-1}   
\end{equation}
for $(s,y)\in[0,T]\times(0,\infty)$ with $\bm{U}(t,s,y)=(U^1(t,s,y),\cdots,U^m(t,s,y))$ being the solution to an equilibrium HJB system:
\begin{equation} \label{Example: Hamiltonian system} 
\left\{
\begin{array}{lr}
\bm{U}^a_s(t,s,y)+\bm{\mathcal{H}}^a(t,s,y,\overline{\bm{\alpha}}(s,y),\overline{\bm{c}} (s,y),\bm{U}(t,s,y),\bm{U}^a_y(t,s,y),\bm{U}^a_{yy}(t,s,y))=0, \\
\bm{U}(t,T,y)=\bm{g}(t)y^\beta,\quad 0\leq t\leq s \leq T,\quad y\in(0,\infty), \quad a=1,\ldots,m. 
\end{array}
\right. 
\end{equation} 
The $m$ HJB equations system of \eqref{Example: Hamiltonian system} are coupled with each other via the equilibrium strategy $(\overline{\bm{\alpha}},\overline{\bm{c}})$ and the recursive dependence on $\bm{Y}$ in the generators of \eqref{PowerBSDE}, i.e. the terms of $\bm{U}(t,s,y)$ of \eqref{Example: Hamiltonian system}. By substituting \eqref{Example: Closed loop strategy} into the Hamiltonian, \eqref{Example: Hamiltonian system} becomes
\begin{equation} \label{PowerHJB}
\left\{
\begin{array}{l}
\bm{U}^a_s(t,s,y)+\frac{1}{2}\left(\sum\limits_{1\leq b\leq m}\left(\frac{(\mu_b-r) \bm{U}^b_y(s,s,y)}{\sigma_b \bm{U}^b_{yy}(s,s,y)}\right)^2\right) \bm{U}^a_{yy}(t,s,y)+\left[ry-\sum\limits_{1\leq b\leq m}\left(\frac{(\mu_b-r)^2 \bm{U}^b_y(s,s,y)}{\sigma_b^2 \bm{U}^b_{yy}(s,s,y)}-\left(\frac{\bm{U}^b_y(s,s,y)}{\beta \bm{v}^{bb}(s,s)}\right)^\frac{1}{\beta-1}\right)\right]\bm{U}^a_y(t,s,y) \qquad ~~{} \\
\hfill +\sum\limits_{1\leq b\leq m}\left(\bm{v}^{ab}(t,s) \left(\frac{\bm{U}^b_y(s,s,y)}{\beta \bm{v}^{bb}(s,s)}\right)^\frac{\beta}{\beta-1}-\bm{w}^{ab}(t,s) \bm{U}^b(t,s,y)\right) = 0, \\
\bm{U}(t,T,y)=\bm{g}(t)y^\beta,\quad 0\leq t\leq s \leq T,\quad y\in(0,\infty), \quad a=1,\ldots,m. 
\end{array}
\right. 
\end{equation}  

{\color{black}It is clear that the first order derivative of the nonlinearity of \eqref{PowerHJB} with respect to $\bm{U}^a_{yy}(t,s,y)$ at $\bm{U}(t,T,y)$ would be degenerate. Consequently, our well-posedness results are not applicable to analyze its solvability. Thus, this second example is more of an inspiration and serves as an indication of the validity of the general (degenerate) case.} To facilitate our analysis of \eqref{Example: Hamiltonian system}, we need a more explicit form of $\bm{U}^a(t,s,y)$ and thus inspired by its terminal condition, we consider the following ansatz: 
\begin{equation} \label{ansatz}
\bm{U}^a(t,s,y)=\bm{\varphi}^a(t,s)y^\beta, \quad 0\leq t\leq s\leq T, 
\end{equation}
for some suitable $\bm{\varphi}^a(\cdot,\cdot)$, $a=1,\ldots,m$. Then we have $\bm{\varphi}^a(t,T)=\bm{g}^a(t)$ and 
by simple calculation, $\bm{\varphi}(t,s)$ satisfies the following system of ODEs: 
\begin{equation*} 
\left\{
\begin{array}{lr}
\bm{\varphi}^a_s(t,s)        
+\left[r\beta+\sum\limits_{1\leq b\leq m}\left(\frac{(\mu_b-r)^2\beta}{2\sigma_b^2 (1-\beta)}\right)-\sum\limits_{1\leq b\leq m}\beta\left(\frac{\bm{\varphi}^b(s,s)}{\bm{v}^{bb}(s,s)}\right)^\frac{1}{\beta-1}\right]\bm{\varphi}^a(t,s)+\sum\limits_{1\leq b\leq m}\bm{v}^{ab}(t,s) \left(\frac{\bm{\varphi}^b(s,s)}{\bm{v}^{bb}(s,s)}\right)^\frac{\beta}{\beta-1}-\sum\limits_{1\leq b\leq m}\bm{w}^{ab}(t,s) \bm{\varphi}^b(t,s) = 0, \\
\bm{\varphi}(t,T)=\bm{g}(t),\quad 0\leq t\leq s \leq T, \quad a=1,\ldots,m. 
\end{array}
\right. 
\end{equation*}  
Denoting by $k = r\beta+\sum\limits_{1\leq b\leq m}\left(\frac{((\mu_b-r)^2\beta}{2\sigma_b^2 (1-\beta)}\right)$, the ODE system above becomes 
\begin{equation} \label{ODE system in a matrix form} 
\left\{
\begin{array}{lr}
\bm{\varphi}_s(t,s)+\bm{A}(t,s,\bm{\varphi}(s,s))\cdot\bm{\varphi}(t,s)+\bm{f}(t,s,\bm{\varphi}(s,s)) = 0, \\
\bm{\varphi}(t,T)=\bm{g}(t),\quad 0\leq t\leq s \leq T. 
\end{array}
\right. 
\end{equation} 
where $$\bm{A}(t,s,\bm{\varphi}(s,s))=-\bm{w}(t,s)+\mathrm{diag}\left\{k-\sum_b\beta\left(\frac{\bm{\varphi}^b(s,s)}{\bm{v}^{bb}(s,s)}\right)^\frac{1}{\beta-1},\cdots,k-\sum_b\beta\left(\frac{\bm{\varphi}^b(s,s)}{\bm{v}^{bb}(s,s)}\right)^\frac{1}{\beta-1}\right\}$$  
$$\bm{f}(t,s,\bm{\varphi}(s,s))=\left(\sum_b \bm{v}^{1b}(t,s) \left(\frac{\bm{\varphi}^b(s,s)}{\bm{v}^{bb}(s,s)}\right)^\frac{\beta}{\beta-1},\cdots,\sum_b \bm{v}^{mb}(t,s) \left(\frac{\bm{\varphi}^b(s,s)}{\bm{v}^{bb}(s,s)}\right)^\frac{\beta}{\beta-1}\right)^\top$$. 


According to the classical theory of system of ODEs, the fundamental matrix $\bm{\chi}$ makes it possible to write every solution of the inhomogeneous system \eqref{ODE system in a matrix form} in the form of Cauchy's formula 
\begin{equation} \label{Integral equation of psi} 
\bm{\psi}(t,s) = \bm{\chi}(t,s,\bm{\psi}(s,s))\bm{\chi}^{-1}(t,T,\bm{\psi}(T,T))\bm{g}(t)+\int^T_s \bm{\chi}(t,s,\bm{\psi}(s,s))\bm{\chi}^{-1}(t,\tau,\bm{\psi}(\tau,\tau))\bm{f}(t,\tau,\bm{\psi}(\tau,\tau))d\tau,
\end{equation}
where $\bm{\chi}^{-1}$ is the inverse matrix of $\bm{\chi}$, and 
\begin{equation} \label{ODE fundamental solution}
\bm{\chi}(t,s,\bm{\psi}(s,s)) = \bm{I}+\int^T_s \bm{A}(t,\tau,\bm{\psi}(\tau,\tau))d\tau+\int^T_s \bm{A}(t,\tau,\bm{\psi}(\tau,\tau))\int^T_\tau \bm{A}(t,\sigma,\bm{\psi}(\sigma,\sigma))d\sigma d\tau + \cdots,
\end{equation}
in which $\bm{I}$ is $m\times m$ identity matrix. Note that \eqref{ODE fundamental solution} converges absolutely for every $s\in[0,t]$ and uniformly on every compact interval in $[0,t]$. Taking $t=s$ gives us that 
\begin{equation*}  
\bm{\psi}(s,s) = \bm{\chi}(s,s,\bm{\psi}(s,s))\bm{\chi}^{-1}(s,T,\bm{\psi}(T,T))\bm{g}(t)+\int^T_s \bm{\chi}(s,s,\bm{\psi}(s,s))\bm{\chi}^{-1}(s,\tau,\bm{\psi}(\tau,\tau))\bm{f}(s,\tau,\bm{\psi}(\tau,\tau))d\tau,
\end{equation*}
By introducing $\overline{\bm{\psi}}(s)=\bm{\psi}(s,s)$, we have 
\begin{equation} \label{General integral equation} 
\overline{\bm{\psi}}(s) = \bm{\chi}(s,s,\overline{\bm{\psi}}(s))\bm{\chi}^{-1}(s,T,\overline{\bm{\psi}}(T))\bm{g}(s)+\int^T_s \bm{\chi}(s,s,\overline{\bm{\psi}}(s))\bm{\chi}^{-1}(s,\tau,\overline{\bm{\psi}}(\tau))\bm{f}(s,\tau,\overline{\bm{\psi}}(\tau))d\tau,
\end{equation}
which is a nonlinear integral system for the unknown function $s\longmapsto\overline{\bm{\psi}}(s)$. Once the diagonal value $\overline{\bm{\psi}}(s)=\bm{\psi}(s,s)$ can be determined uniquely, there exists a unique solution $\bm{\psi}(t,s)$ from the integral equation \eqref{Integral equation of psi} of $\bm{\psi}(t,s)$. {\color{black}By \eqref{Example: Closed loop strategy} and \eqref{ansatz}, the equilibrium investment-consumption strategy and equilibrium value functions can be represented with $\overline{\bm{\psi}}(s)$ in \eqref{General integral equation} as follows: 
\begin{equation} \label{Example 2: Closed loop strategy} 
	\overline{\bm{\alpha}}^a(s,y)=\frac{(\mu_a-r) }{\sigma_a^2 (1-\beta)}y, \quad \overline{\bm{c}}^a(s,y)=\left(\frac{\bm{v}^{aa}(s,s)}{\overline{\bm{\psi}}(s)}\right)^\frac{1}{1-\beta}y, \quad \bm{V}(s,y)=\overline{\bm{\psi}}(s)y^\beta, \quad (s,y)\in[0,T]\times(0,\infty),    
\end{equation}

}
The preliminary analyses above provide us the analytical form of the TC-NE value function, while the key is to make use of the ansatz \eqref{ansatz} to transform the nonlocal PDE system \eqref{Example: Hamiltonian system} into a classical (local) ODE system \eqref{ODE system in a matrix form} (and equivalently, an conventional integral system \eqref{General integral equation}). We can then again use the contraction mapping arguments to establish the local well-posedness of \eqref{General integral equation}.
Moreover, to prove its solvability in an arbitrary large time interval, we shall show the boundedness of the solution of \eqref{General integral equation} such that the extension procedure can be completed.
The following proposition supplements the mathematical details of the above.
\begin{proposition} \label{Solvability of the financial example} 
Suppose that $\bm{v}$, $\bm{w}$, and $\bm{g}$ are continuously differentiable, then there exists $\delta\in(0,T]$ such that the TIC SDG problem \eqref{Example: Forward process}-\eqref{ExampleBackprocess} with \eqref{PowerBSDE} admits a closed-loop TC-NE point $(\overline{\bm{\alpha}},\overline{\bm{c}})(s,y)$ given by \eqref{Example: Closed loop strategy} and the corresponding TC-NE value function $\bm{V}(s,y)={\color{black}\overline{\bm{\psi}}(s)}y^\beta$ over $s\in [T-\delta,T]$. Moreover, if $\bm{w}$ is a diagonal matrix and $\bm{v}$ and $\bm{g}$ satisfy \eqref{Condition: g_0}-\eqref{Condition: gamma}, then $\delta=T$, which implies that the TIC SDG problem \eqref{Example: Forward process}-\eqref{ExampleBackprocess} with \eqref{PowerBSDE} is globally solvable.
\end{proposition}
\begin{remark}
For the condition \eqref{Condition: gamma}, \cite{Yong2012,Wei2017} have investigated the case where $\bm{v}^{ab}(t,s)=\bm{v}^{aa}(t,s)$ for $b\neq a$ and $a=1,\ldots,m$. As they showed, the continuous differentiability of $\bm{v}$ in $s$ can guarantee \eqref{Condition: gamma} for this special case. {\color{black}Moreover, in contrast to Proposition \ref{Exp1Propostion} that provides the existence and uniqueness of solutions of equilibrium HJB system \eqref{Backward HJB Exponential utility}, Proposition \ref{Solvability of the financial example} only promises the existence of solutions for the TIC SDG problem \eqref{Example: Forward process}-\eqref{ExampleBackprocess} with \eqref{PowerBSDE}. Since the equilibrium HJB system \eqref{PowerHJB} cannot be covered by the current framework, we merely constructed one solution for the power utility model via the ansatz \eqref{ansatz}.} 
\end{remark}

Our TIC SDG examples and results generalize the ones in the existing literature. Specifically, in the case of $m=1$, the TIC SDG is reduced to the TIC stochastic control problem in \cite{Wei2017} with recursive utility functional and in \cite{Yong2012} with non-recursive one. Noteworthy is that the well-posedness results in \cite{Yong2012,Wei2017} do not allow the diffusion to be controlled. Moreover, when $m=1$, $\bm{v}$, $\bm{w}$, and $\bm{g}$ are all independent of $t$, the problem is reduced to the TC case with recursive utility functional studied in \cite{Karoui2001}. Based on these restrictions, if $\bm{v}$ is constant and $\bm{w}=0$, the examples are further reduced to the classical Merton problem in \cite{Merton1971}.


\section{Feynman--Kac Formula for Nonlocal Parabolic Systems} \label{sec:stochrep}
In this section, we provide a nonlocal version of the Feynman--Kac formula, which establishes a closed link between the solutions to a flow of FBSDEs in the multidimensional case and nonlocal second-order parabolic systems. All the proofs are deferred to \ref{app:pf}.

Before we present the formula, we reveal more properties of the solution to \eqref{Nonlocal fully nonlinear system}. Like the classical theory of parabolic systems, the stronger conditions imposed to the nonlinearity $\bm{F}$ and the given data $\bm{g}$ suggest the higher regularity of the corresponding solutions of nonlocal higher-order systems. 
\begin{lemma} \label{More regularities of solutions}
Let $k$ and $K$ be both non-negative integers satisfying $k\leq K$. Suppose that $\bm{F}$ is smooth and regular enough and $\bm{g}\in\bm{\Omega}^{2r+K+\alpha}_{[0,T]}$, then there exist $\delta>0$ and a unique $\bm{u}$ in $\Delta[0,\delta]\times\mathbb{R}^d$ satisfying \eqref{Nonlocal fully nonlinear system} with $D^k_y\bm{u}\in\bm{\Omega}^{2r+\alpha}_{[0,\delta]}$ for all $k\leq K$.  
\end{lemma}

Next, to connect parabolic systems with the theory of FBSDEs, we consider a second-order backward nonlocal fully nonlinear system with $r=1$ of the form:
\begin{equation} \label{Backward nonlocal fully nonlinear equation}  
\left\{
\begin{array}{lr}
\bm{u}_s(t,s,y) + \bm{F}\big(t,s,y,\bm{u}(t,s,y),\bm{u}_y(t,s,y),\bm{u}_{yy}(t,s,y),\bm{u}(s,s,y),\bm{u}_y(s,s,y),\bm{u}_{yy}(s,s,y)\big)=0, \\
\bm{u}(t,T,y)=\bm{g}(t,y),\hfill t_0\leq t\leq s\leq T,\quad y\in\mathbb{R}^d,
\end{array}
\right.
\end{equation} 
where $\bm{F}$ has enough regularities and $t_0$ is suitable in the sense that $[t_0,T]$ is a subset of the time interval for the maximally defined solution of \eqref{Backward nonlocal fully nonlinear equation}. The following theorem reveals the relationship between the solutions to a nonlocal fully nonlinear second-order system and to a flow of 2FBSDEs \eqref{Flow of 2FBSDEs}. 

\begin{theorem} \label{F-K formula}
Suppose that $\bm{F}$ has enough regularities, $\sigma(s,y)\in C^{1,2}([t_0,T]\times\mathbb{R}^d)$, and $\bm{g}\in\bm{\Omega}^{{(3+\alpha)}}_{[t_0,T]}$. Then, \eqref{Backward nonlocal fully nonlinear equation} admits a unique solution $\bm{u}(t,s,y)$ that is first-order continuously differentiable in $s$ and third-order continuously differentiable with respect to $y$ in $\nabla[t_0,T]\times\mathbb{R}^d$. Moreover, for any $a=1,\ldots,m$, let
\begin{eqnarray} \label{F-K formula for 2FBSVIE}
& \bm{Y}^a(t,s) := \bm{u}^a(t,s,\bm{X}(s)), \qquad & \bm{Z}^a(t,s) :=  \left(\sigma^\top \bm{u}^a_y\right)(t,s,\bm{X}(s)), \\
&\bm{\Gamma}^a(t,s) :=\left(\sigma^\top\left(\sigma^\top \bm{u}^a_y\right)_y\right)(t,s,\bm{X}(s)), \qquad & \bm{A}^a(t,s) := \mathcal{D}\left(\sigma^\top \bm{u}^a_y\right)(t,s,\bm{X}(s)), \nonumber 
\end{eqnarray} 
where $\left(\sigma^\top \bm{u}^a_y\right)(t,s,y)=\sigma^\top(s,y)\bm{u}^a_y(t,s,y)$ and the operator $\mathcal{D}$ is defined by 
\begin{equation*}
\mathcal{D}\bm{\varphi}^a=\bm{\varphi}_s+\frac{1}{2}\sum^d_{i,j=1}\left(\sigma\sigma^\top\right)_{ij}\frac{\partial^2\bm{\varphi}^a}{\partial y_i\partial y_j}+\sum^d_{i=1}b_i\frac{\partial \bm{\varphi}^a}{\partial y_i},
\end{equation*}
then the family of random fields $\left(\bm{X}(\cdot),\bm{Y}(\cdot,\cdot),\bm{Z}(\cdot,\cdot),\bm{\Gamma}(\cdot,\cdot),\bm{A}(\cdot,\cdot)\right)$ is an adapted solution of the following flow of 2FBSDEs: 
\begin{align} \label{Flow of 2FBSDEs}
\bm{X}(s) & = y+\int^s_{t_0}b(\tau,\bm{X}(\tau))d\tau+\int^s_{t_0}\sigma(\tau,\bm{X}(\tau))d\bm{W}(\tau), \\
\bm{Y}^a(t,s) & = \bm{g}^a(t,\bm{X}(T))+\int^T_s\mathbb{F}^a(t,\tau,\bm{X}(\tau),\bm{Y}(t,\tau),\bm{Y}(\tau,\tau),\bm{Z}(t,\tau),\bm{Z}(\tau,\tau),\bm{\Gamma}(t,\tau),\bm{\Gamma}(\tau,\tau))d \tau-\int^T_s\left(\bm{Z}^a\right)^\top(t,\tau)d\bm{W}(\tau), \nonumber \\
\bm{Z}^a(t,s) & = \bm{Z}^a(t,t_0)+\int^s_{t_0}\bm{A}^a(t,\tau)d\tau+\int^s_{t_0}\bm{\Gamma}^a(t,\tau)d\bm{W}(\tau), \quad t_0\leq t\leq s\leq T,\quad y\in\mathbb{R}^d, \nonumber   
\end{align} 
where $\mathbb{F}^a$ is defined by 
\begin{equation}
\begin{split}
	& \mathbb{F}^a\big(t,\tau,\bm{X}(\tau),\bm{Y}(t,\tau),\bm{Y}(\tau,\tau),\bm{Z}(t,\tau),\bm{Z}(\tau,\tau),\Gamma(t,\tau),\Gamma(\tau,\tau)\big) \\
	=& \overline{\bm{F}}^a\big(t,\tau,\bm{X}(\tau),\bm{u}(t,\tau,\bm{X}(\tau)),\bm{u}_y(t,\tau,\bm{X}(\tau)),\bm{u}_{yy}(t,\tau,\bm{X}(\tau)),\bm{u}(\tau,\tau,\bm{X}(\tau)),\bm{u}_y(\tau,\tau,\bm{X}(\tau)),\bm{u}_{yy}(\tau,\tau,\bm{X}(\tau))\big)
\end{split}
\end{equation}
with the definition of $\overline{\bm{F}}^a$
\begin{equation*}
\begin{split}
	& \overline{\bm{F}}^a\big(t,\tau,y,\bm{u}(t,\tau,y),\bm{u}_y(t,\tau,y),\bm{u}_{yy}(t,\tau,y),\bm{u}(\tau,\tau,y),\bm{u}_y(\tau,\tau,y),\bm{u}_{yy}(\tau,\tau,y)\big) \\
	:=& \bm{F}^a\big(t,\tau,y,\bm{u}(t,\tau,y),\bm{u}_y(t,\tau,y),\bm{u}_{yy}(t,\tau,y),\bm{u}(\tau,\tau,y),\bm{u}_y(\tau,\tau,y),\bm{u}_{yy}(\tau,\tau,y)\big) \\
	& -\frac{1}{2}\sum^d_{i,j=1}\left(\sigma\sigma^\top\right)_{ij}(\tau,y)\frac{\partial^2 \bm{u}^a}{\partial y_i\partial y_j}(t,\tau,y)-\sum^d_{i=1}b_i(\tau,y)\frac{\partial \bm{u}^a}{\partial y_i}(t,\tau,y). 
\end{split}
\end{equation*}
\end{theorem}

We make three important observations about the stochastic system \eqref{Flow of 2FBSDEs}: (I) When the generator $\mathbb{F}$ is independent of diagonal terms, i.e. $\bm{Y}(\tau,\tau)$, $\bm{Z}(\tau,\tau)$, and $\bm{\Gamma}(\tau,\tau)$, the flow of FBSDEs \eqref{Flow of 2FBSDEs} is reduced to a family of 2FBSDEs parameterized by $t$, which is exactly the 2FBSDEs in \cite{Kong2015} and equivalent to the ones in \cite{Cheridito2007} for any fixed $t$; (II) \eqref{Flow of 2FBSDEs} is more general than the systems in \cite{Wang2019,Wang2020,Hamaguchi2020,Lei2023} since it allows for a nonlinearity of $(\bm{Y}(t,\tau),\bm{Z}(t,\tau),\bm{\Gamma}(t,\tau))$ by increasing the dimensions and/or introducing an additional SDE of $(\bm{\Gamma},\bm{A})$ as well as diagonal terms $(\bm{Y}(\tau,\tau),\bm{Z}(\tau,\tau),\bm{\Gamma}(\tau,\tau))$ in almost arbitrary way; (III) Theorem \ref{F-K formula} shows how to solve the flow of multidimensional 2FBSDEs \eqref{Flow of 2FBSDEs} from the perspective of nonlocal systems. Inspired by \cite{Cheridito2007,Soner2011}, the opposite implication of solutions (from 2FBSDEs to PDE) is likely valid by establishing the well-posedness of \eqref{Flow of 2FBSDEs} in the theoretical framework of SDEs. However, it is beyond the scope of this paper, while we will prove the existence and uniqueness of \eqref{Flow of 2FBSDEs} in our future works.


\section{Conclusions} \label{sec:conclusion}
We provided the conditions on the nonlocal higher-order systems, under which the global well-posedness of the linear, quasilinear, fully nonlinear systems can be proved. The results are significant for a general class of nonzero-sum TIC SDGs that we formulated and discussed in Section \ref{sec:sdg}. Moreover, we present a nonlocal multidimensional version of a Feynman--Kac formula.
It provides new insights into the studies of a flow of 2FBSDEs or 2BSVIEs.

We presented two immediate applications (SDG and SDE) drawing upon our main results from the PDE perspective. In fact,
the study of systems of differential equations is crucial for developing other mathematical tools (in PDE), such as quasilinearization among many others, in a new environment (here, with nonlocality).
The quasilinearization is a common technique in the classical (fully nonlinear) PDE problems. Specifically, under suitable regularity assumptions, we can differentiate the fully nonlinear equation of an unknown function $\varphi$ with respect to each state variable $y_i$ and yield an induced quasilinear system for $\left(\varphi,\frac{\partial\varphi}{\partial y_1},\cdots,\frac{\partial\varphi}{\partial y_d}\right)$. Before taking advantage of the mathematical results of quasilinear systems, it is crucial to verify the equivalence between the original fully nonlinear equation for $\varphi$ and the induced quasilinear system for $\left(\varphi,\frac{\partial\varphi}{\partial y_1},\cdots,\frac{\partial\varphi}{\partial y_d}\right)$. The verification, however, requires the existence and uniqueness of (nonlocal) linear systems. Hence, our study for nonlocal linear systems can serve as a prerequisite for one to develop quasilinearization methods for nonlocal differential equations. 



\section*{Acknowledgements}
The authors are thankful to two anonymous reviewers for constructive comments that helped improve the paper. Chi Seng Pun gratefully acknowledges Ministry of Education (MOE), AcRF Tier 2 grant (Reference No: MOE-T2EP20220-0013) for the funding of this research. Data sharing is not applicable to this paper as no dataset was generated or analyzed during this study.


\bibliography{NonlocalityRef}%



\appendix
\section{Proofs of Statements} \label{app:pf}
\begin{proof}[Proof of Proposition \ref{Equivalance between forward probelms and backward problems}]
Let $\bm{u}(t,s,y)=\bm{v}(T-t,T-s,y)$, then \eqref{Backward system} can be written as   
\begin{equation} \label{Backward system 1}
\left\{
\begin{array}{lr}
	\bm{v}_{T-s}(T-t,T-s,y)+\bm{F}\big(t,s,y,\left(\partial_I\bm{v}\right)_{|I|\leq 2r}(T-t,T-s,y),\left(\partial_I\bm{v}\right)_{|I|\leq 2r}(T-s,T-s,y)\big)=0, \\
	\bm{v}(T-t,0,y)=\bm{g}(t,y),\hfill 0\leq t\leq s\leq T,\quad y\in\mathbb{R}^d.  
\end{array}
\right. 
\end{equation}
Next, we introduce $t^\prime=T-t$, $s^\prime=T-s$ and $y^\prime=y$. Then \eqref{Backward system 1} is equivalent to 
\begin{equation*} 
\left\{
\begin{array}{lr}
	\bm{v}_{s^\prime}(t^\prime,s^\prime,y^\prime)=\bm{F}\big(T-t^\prime,T-s^\prime,y^\prime,\left(\partial_I\bm{v}\right)_{|I|\leq 2r}(t^\prime,s^\prime,y^\prime),  \left(\partial_I\bm{v}\right)_{|I|\leq 2r}(s^\prime,s^\prime,y^\prime)\big), \\
	\bm{v}(t^\prime,0,y^\prime)=\bm{g}(T-t^\prime,y^\prime),\hfill 0\leq s^\prime\leq t^\prime\leq T,\quad y^\prime\in\mathbb{R}^d.
\end{array}
\right. 
\end{equation*}
By modifying the nonlinearity $\bm{F}$ and the data $\bm{g}$ but not really affecting their properties with respect to the time variables, the problem above can be further reformulated as a forward problem:
\begin{equation*} 
\left\{
\begin{array}{lr}
	\bm{v}_{s^\prime}(t^\prime,s^\prime,y^\prime)=\bm{F}^\prime\big(t^\prime,s^\prime,y^\prime,\left(\partial_I\bm{v}\right)_{|I|\leq 2r}(t^\prime,s^\prime,y^\prime),  \left(\partial_I\bm{v}\right)_{|I|\leq 2r}(s^\prime,s^\prime,y^\prime)\big), \\
	\bm{v}(t^\prime,0,y^\prime)=\bm{g}^\prime(t^\prime,y^\prime),\hfill 0\leq s^\prime\leq t^\prime\leq T,\quad y^\prime\in\mathbb{R}^d, 
\end{array}
\right. 
\end{equation*} 
which completes the proof. 
\end{proof}

\begin{proof}[Proof of Lemma \ref{Equivalence between systems}]
The first claim is obvious. We prove the second one below. According to \eqref{Induced nonlocal linear PDE system}, it is obvious that $\frac{\partial\bm{u}}{\partial t}$ satisfies 
\begin{equation}
\left\{
\begin{array}{rcl}
	\left(\frac{\partial\bm{u}}{\partial t}\right)^a_s(t,s,y) & = & \sum\limits_{|I|\leq 2r,b\leq m}\left(A+B\right)^{aI}_b(\cdot)\partial_I \left(\frac{\partial\bm{u}}{\partial t}\right)^b(t,s,y)+\sum\limits_{|I|\leq 2r,b\leq m}\left(\frac{\partial A}{\partial t}+\frac{\partial B}{\partial t
	}\right)^{aI}_b(\cdot)\partial_I \bm{u}^b(t,s,y), \\
	&&-\sum\limits_{|I|\leq 2r,b\leq m}\left(\frac{\partial B}{\partial t}\right)^{aI}_b(\cdot)\int^t_s \partial_I\bm{v}^b(\theta,s,y)d\theta-\sum\limits_{|I|\leq 2r,b\leq m}B^{aI}_b(\cdot)\partial_I \bm{v}^b(t,s,y)+\bm{f}^a_t(\cdot), \qquad \hfill a=1,\ldots,m, \\
	\left(\frac{\partial\bm{u}}{\partial t}\right)(t,0,y) & = & \bm{g}_t(t,y), \hfill 0\leq s\leq t\leq T, \quad y\in\mathbb{R}^d.  
\end{array}
\right.
\end{equation}
Hence, the difference between $\bm{v}$ and $\frac{\partial\bm{u}}{\partial t}$ satisfies
\begin{equation*}
\left\{
\begin{array}{rcl}
	\left(\frac{\partial\bm{u}}{\partial t}-\bm{v}\right)^a_s(t,s,y) & = & \sum\limits_{|I|\leq 2r,b\leq m}\left(A+B\right)^{aI}_b(\cdot)\partial_I \left(\frac{\partial\bm{u}}{\partial t}-\bm{v}\right)^b(t,s,y), \qquad \hfill a=1,\ldots,m, \\ 
	\left(\frac{\partial\bm{u}}{\partial t}-\bm{v}\right)(t,0,y) & = & \bm{0}, \hfill 0\leq s\leq t\leq T,~y\in\mathbb{R}^d,
\end{array}
\right.
\end{equation*} 
which implies the unique trivial solution $\frac{\partial\bm{u}}{\partial t}-\bm{v}=\bm{0}$ for any $0\leq s\leq t\leq T$ and $y\in\mathbb{R}^d$, thanks to the classical theory of higher-order parabolic system; see \cite{Ladyzhanskaya1968,Eidelman1969,Solonnikov1965}. 

By replacing $\bm{v}$ with $\frac{\partial\bm{u}}{\partial t}$ in the first equation of \eqref{Induced nonlocal linear PDE system}, we have 
\begin{eqnarray*}
\bm{u}^a_s(t,s,y) &=&\sum\limits_{|I|\leq 2r,b\leq m}\left(A+B\right)^{aI}_b(\cdot)\partial_I\bm{u}^b(t,s,y)-\sum\limits_{|I|\leq 2r,b\leq m}B^{aI}_b(\cdot)\int^t_s\partial_I \left(\frac{\partial\bm{u}}{\partial t}\right)^b(\theta,s,y)d\theta+\bm{f}^a(\cdot) \\
&=&\sum\limits_{|I|\leq 2r,b\leq m}A^{aI}_b(\cdot)\partial_I\bm{u}^b(t,s,y)+\sum\limits_{|I|\leq 2r,b\leq m}B^{aI}_b(\cdot)\partial_I\bm{u}^b(s,s,y)+\bm{f}^a(\cdot),   
\end{eqnarray*}
which completes the proof.
\end{proof}

\begin{proof}[Proof of Theorem \ref{Well-posedness of u and v}]
We first adopt the Banach fixed point arguments to prove the local well-posedness of \eqref{Induced nonlocal linear PDE system} and then extend the local solution to the whole triangular time region $\Delta[0,T]$.   

\vspace{0.3cm} 
\noindent(\textbf{Method of
contraction mapping}) According to \eqref{Induced nonlocal linear PDE system}, we first construct a mapping $\bm{\Gamma}$ from $\bm{v}$ to $\bm{V}$, where $\bm{V}$ is part of the solution $(\bm{u},\bm{V})$ to 
\begin{equation} \label{Contraction of linear system} 
\left\{
\begin{array}{rcl}
	\bm{u}^a_s(t,s,y) & = & \sum\limits_{|I|\leq 2r,b\leq m}\left[\left(A+B\right)^{aI}_b(\cdot)\partial_I\bm{u}^b(t,s,y)-B^{aI}_b(\cdot)\int^t_s\partial_I \bm{v}^b(\theta,s,y)d\theta\right]+\bm{f}^a(\cdot), \\
	\bm{V}^a_s(t,s,y) & = & \sum\limits_{|I|\leq 2r,b\leq m}\left[A^{aI}_b(\cdot)\partial_I \bm{V}^b(t,s,y)+\left(\frac{\partial A}{\partial t}+\frac{\partial B}{\partial t
	}\right)^{aI}_b(\cdot)\partial_I \bm{u}^b(t,s,y)-\left(\frac{\partial B}{\partial t}\right)^{aI}_b(\cdot)\int^t_s \partial_I\bm{v}^b(\theta,s,y)d\theta\right]+\bm{f}^a_t(\cdot), \\
	\left(\bm{u},\bm{V}\right)(t,0,y) & = & \left(\bm{g},\bm{g}_t\right)(t,y), \hfill a=1,\ldots,m,~0\leq s\leq t\leq \delta,~y\in\mathbb{R}^d.   
\end{array}
\right.
\end{equation} 
The operator $\bm{\Gamma}(\bm{v})=\bm{V}$ is defined in the set 
\begin{equation*}
\bm{\mathcal{V}}=\left\{\bm{v}(\cdot,\cdot,\cdot)\in C(\Delta[0,\delta]\times\mathbb{R}^d;\mathbb{R}^m):~[\bm{v}]^{(2r+\alpha)}_{[0,\delta]}<\infty\right\}.
\end{equation*}  
Thanks to the theory of classical parabolic system parameterized by $t$, the operator $\bm{V}=\bm{\Gamma}(\bm{v})$ is well-defined. Next, we are to prove that this mapping is a contraction, i.e. for any $\bm{v}$, $\widehat{\bm{v}}\in\bm{\mathcal{V}}$, it holds that 
\begin{equation}
\left[\bm{\Gamma}\left(\bm{v}\right)-\bm{\Gamma}\left(\widehat{\bm{v}}\right)\right]^{(2r+\alpha)}_{[0,\delta]}\leq\frac{1}{2}\left[\bm{v}-\widehat{\bm{v}}\right]^{(2r+\alpha)}_{[0,\delta]}. 
\end{equation} 

It is clear that 
\begin{equation} 
\left\{
\begin{array}{rcl}
	\left(\bm{u}-\widehat{\bm{u}}\right)^a_s(t,s,y) & = & \sum\limits_{|I|\leq 2r,b\leq m}\left[\left(A+B\right)^{aI}_b(\cdot)\partial_I\left(\bm{u}-\widehat{\bm{u}}\right)^b(t,s,y)-B^{aI}_b(\cdot)\int^t_s\partial_I \left(\bm{v}-\widehat{\bm{v}}\right)^b(\theta,s,y)d\theta\right], \\
	\left(\bm{V}-\widehat{\bm{V}}\right)^a_s(t,s,y) & = &\sum\limits_{|I|\leq 2r,b\leq m}\left[A^{aI}_b(\cdot)\partial_I \left(\bm{V}-\widehat{\bm{V}}\right)^b(t,s,y)-\left(\frac{\partial B}{\partial t}\right)^{aI}_b(\cdot)\int^t_s \partial_I\left(\bm{v}-\widehat{\bm{v}}\right)^b(\theta,s,y)d\theta \right. \\
	& & \left.\qquad\qquad\qquad+\left(\frac{\partial A}{\partial t}+\frac{\partial B}{\partial t
	}\right)^{aI}_b(\cdot)\partial_I \left(\bm{u}-\widehat{\bm{u}}\right)^b(t,s,y)\right], \\
	\left(\bm{u}-\widehat{\bm{u}},\bm{V}-\widehat{\bm{V}}\right)(t,0,y) & = & \left(\bm{0},\bm{0}\right), \hfill a=1,\ldots,m,~0\leq s\leq t\leq \delta,~y\in\mathbb{R}^d.  
\end{array}
\right.
\end{equation} 

By the classical theory of parabolic system (see \cite[Theorem 4.10]{Solonnikov1965} or \cite[Theorem 10.2]{Ladyzhanskaya1968}), we have that for any fixed $t\in[0,\delta]$,
\begin{equation} \label{Contraction 1}
\left|\left(\bm{u}-\widehat{\bm{u}}\right)(t,\cdot,\cdot)\right|^{(2r+\alpha)}_{[0,t]\times\mathbb{R}^d}\leq C\sum\limits_{|I|\leq 2r,b\leq m}\left|\int^t_\cdot\partial_I\left(\bm{v}-\widehat{\bm{v}}\right)^b(\theta,\cdot,\cdot)d\theta\right|^{(\alpha)}_{[0,t]\times\mathbb{R}^d}
\end{equation}
for some constant $C$, which is generic in this paper and its value could vary in different inequalities below. Then, similar to \cite[Section 2]{Lei2023}, it holds that
\begin{equation} \label{Contraction 2} 
\left|\left(\bm{u}-\widehat{\bm{u}}\right)(t,\cdot,\cdot)\right|^{(2r+\alpha)}_{[0,t]\times\mathbb{R}^d}\leq C\delta^{1-\frac{\alpha}{2r}}\cdot\sup\limits_{t\in[0,\delta]}\left\{\left|\left(\bm{v}-\widehat{\bm{v}}\right)(t,\cdot,\cdot)\right|^{(2r+\alpha)}_{[0,t]\times\mathbb{R}^d}\right\}. 
\end{equation}
Similarly, we also have a priori estimate
\begin{equation}
\left|\left(\bm{V}-\widehat{\bm{V}}\right)(t,\cdot,\cdot)\right|^{(2r+\alpha)}_{[0,t]\times\mathbb{R}^d}\leq C\delta^{1-\frac{\alpha}{2r}}\cdot\sup\limits_{t\in[0,\delta]}\left\{\left|\left(\bm{v}-\widehat{\bm{v}}\right)(t,\cdot,\cdot)\right|^{(2r+\alpha)}_{[0,t]\times\mathbb{R}^d}\right\}. 
\end{equation}
Here, the constant $C$ is independent of $\delta$, $\bm{g}$, $\bm{v}$, and $\widehat{\bm{v}}$. Consequently, under the norm $[\cdot]^{(2r+\alpha)}_{[0,\delta]}:=\sup_{t\in[0,\delta]}|\cdot|^{(2r+\alpha)}_{[0,t]\times\mathbb{R}^d}$, we achieve a contraction 
\begin{equation}
\left[\bm{\Gamma}(\bm{v})-\bm{\Gamma}(\widehat{\bm{v}})\right]^{(2r+\alpha)}_{[0,\delta]}\leq \frac{1}{2}\left[ \bm{v}-\widehat{\bm{v}}\right]^{(2r+\alpha)}_{[0,\delta]}. 
\end{equation}
with a suitably small $\delta$. 

\ \ 

\noindent (\textbf{A contraction $\bm{\Gamma}$ mapping $\bm{\mathcal{V}}$ into itself.}) On the other hand, we also have 
\begin{equation}  
\left[\bm{V}\right]^{(2r+\alpha)}_{[0,\delta]}=\left[\bm{\Gamma}(\bm{v})\right]^{(2r+\alpha)}_{[0,\delta]} \leq C\left([\bm{v}]^{(2r+\alpha)}_{[0,\delta]}+[\bm{f}]^{(\alpha)}_{[0,\delta]}+[\bm{g}]^{(2r+\alpha)}_{[0,\delta]}\right)<\infty. 
\end{equation} 
Here, $C$ is a constant depending on $A$ and $B$. Therefore, it follows that $\bm{V}\in\bm{\Theta}^{(2r+\alpha)}_{[0,\delta]}$. Consequently, $\bm{\Gamma}$ is a contraction, mapping $\bm{\mathcal{V}}$ into itself, and thus it has a unique fixed point $\bm{v}$ in $\bm{\Theta}^{(2r+\alpha)}_{[0,\delta]}$ such that $\bm{\Gamma}(\bm{v})=\bm{v}$. Finally, the unique fixed point $\bm{v}(t,s,y)$ uniquely determines a function $\bm{u}(t,s,y)$ via \eqref{Contraction of linear system}. Moreover, given $\bm{v}\in\bm{\Theta}^{{(2r+\alpha)}}_{[0,\delta]}$, it is clear that $[\bm{u}]^{(2r+\alpha)}_{[0,\delta]}\leq\infty$ as well. Therefore, there exists a unique solution pair $(\bm{u},\bm{v})\in\bm{\Theta}^{{(2r+\alpha)}}_{[0,\delta]}\times\bm{\Theta}^{{(2r+\alpha)}}_{[0,\delta]}$ to \eqref{Nonlocal linear system} in $\Delta[0,\delta]\times\mathbb{R}^d$. 

\ \

\noindent (\textbf{Extension of solutions to the whole time (triangular) region})
We have proven that there exists a time horizon $\delta_1\in(0,T]$ such that the local well-posedness of $(\bm{u},\bm{v})$ holds in $R_1=\{0\le s\le t\le \delta_1\}$ in Figure \ref{fig:extension}:
\begin{figure}[!ht]
\centering
\includegraphics[width=0.3\textwidth]{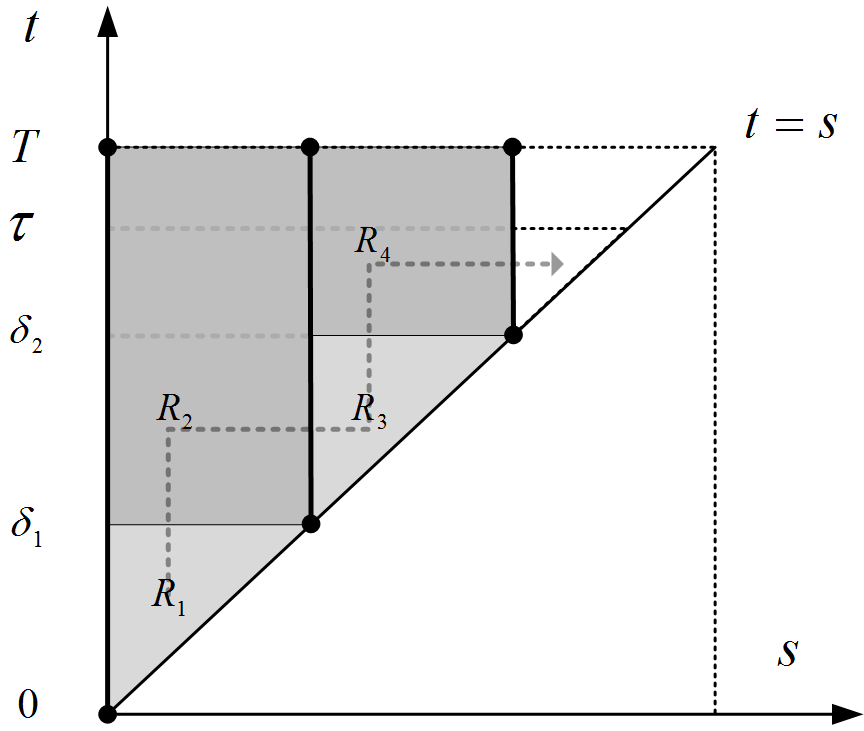}
\caption{Extension from $\Delta[0,\delta]$ to a larger time (triangular) region}
\label{fig:extension}
\end{figure}

\noindent If $\delta_1=T$, the proof is completed. Otherwise, we begin the extension procedure of solutions of \eqref{Induced nonlocal linear PDE system}. Due to the local well-posedness of \eqref{Induced nonlocal linear PDE system} over $R_1$, we can determine the diagonal condition for $s\in[0,\delta_1]$. Then the nonlocal equations \eqref{Nonlocal linear system} and \eqref{Differenate nonlocal linear system with respect to t} reduce to classical systems with a parameter $t$. Therefore, we can extend uniquely our solution $(\bm{u},\bm{v})$ from $R_1$ to $R_1\cup R_2=\{0\le s\le \delta_1,~s\le t\le T\}$ in Figure \ref{fig:extension}. Subsequently, we acquire a new initial condition at $s=\delta_1$ for $t\in[\delta_1,T]$. Taking $\delta_1$ as an initial time and $(\bm{u}(t,\delta_1,y),\bm{v}(t,\delta_1,y))$ as initial datum, one can extend the solution to a larger time intervals $R_1\cup R_2\cup R_3$ as illustrated in Figure \ref{fig:extension}. Hence, we can extend uniquely the solution from $\Delta[0,\delta_1]$ to $\Delta[0,\delta_2]$, and then $R_1\cup R_2\cup R_3\cup R_4$. Considering the facts that $\bm{\Gamma}$ is defined in the whole space $\bm{\Theta}^{(2r+\alpha)}_{[a,b]}$ and the constant $C$ in front of \eqref{Contraction 1} and \eqref{Contraction 2} only depend on $A$ and $B$ instead of the norms of $\bm{g}$, $\bm{v}$, and $\widehat{\bm{v}}$, we can always construct a $\frac{1}{2}$-contraction mapping to find the solution in a larger time region. Consequently, the extension procedure could be repeated up to a global solution pair $(\bm{u},\bm{v})\in\bm{\Theta}^{(2+\alpha)}_{[0,T]}\times\bm{\Theta}^{(2+\alpha)}_{[0,T]}$.  
\end{proof}

\begin{proof}[Proof of Theorem \ref{Schauder estimates}]
As Lemma \ref{Equivalence between systems} shows, the first component $\bm{u}$ of the unique solution $(\bm{u},\bm{v})$ of \eqref{Induced nonlocal linear PDE system} solves the nonlocal linear equation \eqref{Nonlocal linear system} in $\Delta[0,T]\times\mathbb{R}^d$. By noting of $\bm{v}=\bm{u}_t$ in \eqref{Induced nonlocal linear PDE system}, it is clear that $\bm{u}\in\bm{\Omega}^{(2r+\alpha)}_{[0,T]}$ thanks to $(\bm{u},\bm{u}_t)\in\bm{\Theta}^{{(2r+\alpha)}}_{[0,T]}\times\bm{\Theta}^{{(2r+\alpha)}}_{[0,T]}$. Moreover, by \eqref{Induced nonlocal linear PDE system}, we have 
\begin{equation} \label{Estimate 1}
\sum\limits_{a\leq m}\left|\bm{u}^a(t,\cdot,\cdot)\right|^{(2r+\alpha)}_{[0,t]\times\mathbb{R}^d}\leq C\left(\delta^{1-\frac{\alpha}{2r}}\sup\limits_{t\in[0,\delta]}\sum\limits_{a\leq m}\left|\bm{u}^a_t(t,\cdot,\cdot)\right|^{(2r+\alpha)}_{[0,t]\times\mathbb{R}^d}+\sup\limits_{t\in[0,\delta]}\sum\limits_{a\leq m}\left|\bm{f}^a(t,\cdot,\cdot)\right|^{(\alpha)}_{[0,t]\times\mathbb{R}^d}+\sup\limits_{t\in[0,\delta]}\sum\limits_{a\leq m}\left|\bm{g}^a(t,\cdot)\right|^{(2r+\alpha)}_{\mathbb{R}^d}\right)
\end{equation}
and 
\begin{equation} \label{Estimate 2} 
\begin{split}
	\sum\limits_{a\leq m}\left|\bm{u}^a_t(t,\cdot,\cdot)\right|^{(2r+\alpha)}_{[0,t]\times\mathbb{R}^d}&\leq C\left(\sum\limits_{a\leq m}\left|\bm{u}^a(t,\cdot,\cdot)\right|^{(2r+\alpha)}_{[0,t]\times\mathbb{R}^d}+\delta^{1-\frac{\alpha}{2r}}\sup\limits_{t\in[0,\delta]}\sum\limits_{a\leq m}\left|\bm{u}^a_t(t,\cdot,\cdot)\right|^{(2r+\alpha)}_{[0,t]\times\mathbb{R}^d}\right. \\
	& \qquad \qquad\left.+\sup\limits_{t\in[0,\delta]}\sum\limits_{a\leq m}\left|\bm{f}^a_t(t,\cdot,\cdot)\right|^{(\alpha)}_{[0,t]\times\mathbb{R}^d}+\sup\limits_{t\in[0,\delta]}\sum\limits_{a\leq m}\left|\bm{g}^a_t(t,\cdot)\right|^{(2r+\alpha)}_{\mathbb{R}^d}\right).
\end{split}
\end{equation}

Consequently, for a small enough $\delta$, it holds 
\begin{equation*}
\begin{split}
	& \sum\limits_{a\leq m}\left\{\left|\bm{u}^a(t,\cdot,\cdot)\right|^{(2r+\alpha)}_{[0,t]\times\mathbb{R}^d}+\left|\bm{u}^a_t(t,\cdot,\cdot)\right|^{(2r+\alpha)}_{[0,t]\times\mathbb{R}^d}\right\} \leq \sup\limits_{t\in[0,\delta]}\sum\limits_{a\leq m}\left|\bm{u}^a(t,\cdot,\cdot)\right|^{(2r+\alpha)}_{[0,t]\times\mathbb{R}^d}+\sup\limits_{t\in[0,\delta]}\sum\limits_{a\leq m}\left|\bm{u}^a_t(t,\cdot,\cdot)\right|^{(2r+\alpha)}_{[0,t]\times\mathbb{R}^d} \\
	\leq &~ C\left(\sup\limits_{t\in[0,\delta]}\sum\limits_{a\leq m}\left\{\left|\bm{f}^a(t,\cdot,\cdot)\right|^{(\alpha)}_{[0,t]\times\mathbb{R}^d}+\left|\bm{f}^a_t(t,\cdot,\cdot)\right|^{(\alpha)}_{[0,t]\times\mathbb{R}^d}\right\}+\sup\limits_{t\in[0,\delta]}\sum\limits_{a\leq m}\left\{\left|\bm{g}^a(t,\cdot)\right|^{(2r+\alpha)}_{\mathbb{R}^d}+\left|\bm{g}^a_t(t,\cdot)\right|^{(2r+\alpha)}_{\mathbb{R}^d}\right\}\right),
\end{split}
\end{equation*} 
which leads to a Schauder prior estimate of local solutions of \eqref{Nonlocal linear system} in $\Delta[0,\delta]\times\mathbb{R}^d$
\begin{equation} \label{Local Schauder estimate} 
\| \bm{u}\|^{(2r+\alpha)}_{[0,\delta]}\leq C\left(\| \bm{f}\|^{(\alpha)}_{[0,\delta]}+\| \bm{g}\|^{(2r+\alpha)}_{[0,\delta]}\right). 
\end{equation} 

Now, we have established a Schauder prior estimate \eqref{Local Schauder estimate} of solutions of \eqref{Nonlocal linear system} in a short-time region $\Delta[0,\delta]$. Next, we will show that it still holds for the case $\delta=T$.

First, given the well-posedness of \eqref{Nonlocal linear system} over $R_1$, we consider a classical PDE 
\begin{equation} \label{Reduced local system}  
\left\{
\begin{array}{rcl}
	\bm{u}^a_s(t,s,y) & = & \sum\limits_{|I|\leq 2r,b\leq m}A^{aI}_b(t,s,y)\partial_I\bm{u}^b(t,s,y)+\bm{\eta}^a(t,s,y), \\
	\bm{u}(t,0,y) & = & \bm{g}(t,y),\qquad \qquad \hfill 0\leq s\leq t\leq \delta_1,\quad y\in\mathbb{R}^d,
\end{array}
\right. 
\end{equation}
where $\bm{\eta}^a(t,s,y):=\sum_{|I|\leq 2r,b\leq m}B^{aI}_b(t,s,y)\partial_I\bm{u}^b(s,s,y)+\bm{f}(t,s,y)$. By the classical theory of parabolic system (see \cite[Theorem 4.10]{Solonnikov1965} or \cite[Theorem 10.2]{Ladyzhanskaya1968}), \eqref{Reduced local system} admits a unique solution in $(t,s,y)\in R_2\times\mathbb{R}^d$. Moreover, for any $t\in[\delta_1,T]$, we have 
\begin{equation*} 
\left|\bm{u}(t,\cdot,\cdot)\right|^{(2r+\alpha)}_{[0,\delta_1]\times\mathbb{R}^d}\leq C\left(\left|\bm{\eta}(t,\cdot,\cdot)\right|^{(\alpha)}_{[0,\delta_1]\times\mathbb{R}^d}+\left|\bm{g}(t,\cdot)\right|^{(2r+\alpha)}_{\mathbb{R}^d}\right)\leq  C\left(\| \bm{f}\|^{(\alpha)}_{[0,T]}+\| \bm{g}\|^{(2r+\alpha)}_{[0,T]}\right),
\end{equation*}
the second inequality of which comes from \eqref{Local Schauder estimate}. After taking supremum with respect to $t$, it implies directly that $\| \bm{u}\|^{(2r+\alpha)}_{R_1 \cup R_2}$ (i.e. $\| \bm{u}\|^{(2r+\alpha)}_{[0,T]\times[0,t\wedge\delta_1]}$) is bounded by $\| \bm{f}\|^{(\alpha)}_{[0,T]}$ and $\|\bm{g}\|^{(2r+\alpha)}_{[0,T]}$. The notation of $\| \cdot\|^{(2r+\alpha)}_{R_1 \cup R_2}$ means that the supremum with respect to time pair $(t,s)$ is taken over $R_1 \cup R_2$. 

Next, if we update the initial condition with $\bm{g}^\prime(t,y):=\bm{u}(t,\delta_1,y)$, then it is clear that \begin{equation*}
\|\bm{g}^\prime\|^{(2r+\alpha)}_{[\delta,T]}\leq C\left(\| \bm{f}\|^{(\alpha)}_{[0,T]}+\| \bm{g}\|^{(2r+\alpha)}_{[0,T]}\right)<\infty. 
\end{equation*}
Hence, Theorem \ref{Well-posedness of u and v} tells us that the problem 
\begin{equation} \label{Updated linear system}  
\left\{
\begin{array}{rcl}
	\bm{u}_s(t,s,y) & = & \left(\bm{L}\bm{u}\right)(t,s,y)+\bm{f}(t,s,y), \\
	\bm{u}(t,\delta,y) & = & \bm{g}^\prime(t,y),\qquad \qquad \hfill \delta_1\leq s\leq t\leq T,\quad y\in\mathbb{R}^d.
\end{array}
\right. 
\end{equation}
admits a unique solution $\bm{u}\in\bm{\Omega}^{(2r+\alpha)}_{[\delta_1,\delta_2]}$ and there exists a constant $\delta_2\in(\delta_1,T]$ such that 
\begin{equation}  
\| \bm{u}\|^{(2r+\alpha)}_{[\delta_1,\delta_2]}\leq C\left(\| \bm{f}\|^{(\alpha)}_{[\delta_1,\delta_2]}+\| \bm{g}^\prime\|^{(2r+\alpha)}_{[\delta_1,\delta_2]}\right) \leq C\left(\| \bm{f}\|^{(\alpha)}_{[0,T]}+\| \bm{g}\|^{(2r+\alpha)}_{[0,T]}\right). 
\end{equation} 
Similarly, we can argue that $\| \bm{u}\|^{(2r+\alpha)}_{[\delta_1,T]\times[\delta_1,t\wedge\delta_2]}=\| \bm{u}\|^{(2r+\alpha)}_{R_3 \cup R_4}\leq C\left(\| \bm{f}\|^{(\alpha)}_{[0,T]}+\| \bm{g}\|^{(2r+\alpha)}_{[0,T]}\right)$. Hence, for each extension from $(t,s)\in[\delta_n,T]\times[\delta_n,t\wedge\delta_{n+1}]$ into $(t,s)\in[\delta_{n+1},T]\times[\delta_{n+1},t\wedge\delta_{n+2}]$, it holds that  
\begin{equation}  \label{piecewise estimates}
\| \bm{u}\|^{(2r+\alpha)}_{[\delta_n,T]\times[\delta_n,t\wedge\delta_{n+1}]}\leq C\left(\| \bm{f}\|^{(\alpha)}_{[0,T]}+\| \bm{g}\|^{(2r+\alpha)}_{[0,T]}\right), \quad n=0,1,2,\cdot,N,
\end{equation} 
where $\delta_0=0$ and $\delta_{N+1}=T$. Furthermore, it is clear that $N$ is finite and determined only by $A$ and $B$ according to \eqref{Contraction 1}, \eqref{Contraction 2}, \eqref{Estimate 1}, and \eqref{Estimate 2}. 

Subsequently, for any $t\in[0,T]$, $0\leq s\leq s^\prime\leq t\leq T$, and $y\in\mathbb{R}^d$, we assume that $s\in[\delta_n,\delta_{n+1}]$ and $s^\prime\in[\delta_m,\delta_{m+1}]$ for $0\leq n< m \leq N$ without loss of generality. Then, it follows that 
\begin{equation*}
\begin{split}
	& \sup\limits_{\begin{subarray}{c} 0\leq s< s^\prime\leq t \\ y\in\mathbb{R}^d\end{subarray}}\frac{|D^i_sD^j_y\bm{u}(t,s,y)-D^i_sD^j_y\bm{u}(t,s^\prime,y)|}{|s-s^\prime|^{\frac{l-2ri-j}{2r}}} \\
	\leq  & \sup\limits_{\begin{subarray}{c} 0\leq s< s^\prime\leq t \\ y\in\mathbb{R}^d\end{subarray}}\Bigg\{\frac{|D^i_sD^j_y\bm{u}(t,s,y)-D^i_sD^j_y\bm{u}(t,\delta_{n+1},y)|}{|s-\delta_{n+1}|^{\frac{l-2ri-j}{2r}}}+\sum_{i< k< j}\frac{|D^i_sD^j_y\bm{u}(t,\delta_k,y)-D^i_sD^j_y\bm{u}(t,\delta_{k+1},y)|}{|\delta_k-\delta_{k+1}|^{\frac{l-2ri-j}{2r}}}+\frac{|D^i_sD^j_y\bm{u}(t,\delta_m,y)-D^i_sD^j_y\bm{u}(t,s^\prime,y)|}{|\delta_m-s^\prime|^{\frac{l-2ri-j}{2r}}} \Bigg\}   \\
	\leq & \sum_{n\leq k\leq m}\| \bm{u}\|^{(2r+\alpha)}_{[\delta_k,T]\times[\delta_k,t\wedge\delta_j]}\leq C\left(\| \bm{f}\|^{(\alpha)}_{[0,T]}+\| \bm{g}\|^{(2r+\alpha)}_{[0,T]}\right)  
\end{split}
\end{equation*}
holds for $0<l-2ri-j<2r$. Here, the constant $C$ only depends on $\alpha$, $T$, $A$ and $B$. In addition, whenever $(t,s)$, $(t,s^\prime)\in[\delta_n,T]\times[\delta_n,t\wedge\delta_{n+1}]$, the inequality is obvious from \eqref{piecewise estimates}. Consequently, according to the definition of $\|\cdot\|^{(l)}_{[a,b]}$, we have 
\begin{equation*} 
\| \bm{u}\|^{(2r+\alpha)}_{[0,T]}\leq C\left(\| \bm{f}\|^{(\alpha)}_{[0,T]}+\| \bm{g}\|^{(2r+\alpha)}_{[0,T]}\right). 
\end{equation*}  

Finally, by considering the nonlocal system satisfied by the difference between $\bm{u}$ and $\widehat{\bm{u}}$, we can similarly derive the stability analysis \eqref{Stability analysis of linear system}.
\end{proof} 

\begin{proof}[Proof of Theorem \ref{Well-posedness of fully nonlinear system}]
Overall speaking, we search for the solution of nonlocal fully nonlinear system as a fixed point of the operator $\bm{\Lambda}$, defined in
\begin{equation*}
\bm{\mathcal{U}}=\left\{\bm{u}\in\bm{\Omega}^{(2r+\alpha)}_{[0,\delta]}:\bm{u}(t,0,y)=\bm{g}(t,y),\| \bm{u}-\bm{g}\|^{(2r+\alpha)}_{[0,\delta]}\leq R\right\}  
\end{equation*} 
for a constant $R$, by $\bm{\Lambda}(\bm{u})=\bm{U}$, where $\bm{U}$ is the solution of 
\begin{equation} \label{Definition of Lamda}
\left\{
\begin{array}{lr}
	\bm{U}_s(t,s,y)=\bm{L}_0\bm{U}+\bm{F}\big(t,s,y,\left(\partial_I\bm{u}\right)_{|I|\leq 2r}(t,s,y),  \left(\partial_I\bm{u}\right)_{|I|\leq 2r}(s,s,y)\big)-\bm{L}_0\bm{u}, \\
	\bm{U}(t,0,y)=\bm{g}(t,y),\quad 0\leq s\leq t\leq \delta,\quad y\in\mathbb{R}^d
\end{array}
\right. 
\end{equation}
in which
\begin{equation} \label{Nonlocal linear operator L0}
\begin{split}
	\left(\bm{L}_0\bm{u}\right)^a(t,s,y)
	&=\sum\limits_{|I|\leq 2r,b\leq m}\partial_I \bm{F}^a_b\big(t,0,y,\bm{\theta}_0(t,y)\big)\cdot\partial_I\bm{u}^b(t,s,y)+\sum\limits_{|I|\leq 2r,b\leq m}\partial_I \overline{\bm{F}}^a_b\big(t,0,y,\bm{\theta}_0(t,y)\big)\cdot\partial_I\bm{u}^b(s,s,y)
\end{split}
\end{equation}
with $\bm{\theta}_0(t,y):=\big(\left(\partial_I\bm{g}\right)_{|I|\leq 2r}(t,y),  \left(\partial_I\bm{g}\right)_{|I|\leq 2r}(0,y)\big)$. Note that $\partial_I \bm{F}^a_b\big(t,0,y,\bm{\theta}_0(t,y)\big)$
is meant to be evaluated at $(t,0,y,\bm{\theta}_0(t,y))$, i.e. $\big(t,0,y,\left(\partial_I\bm{g}\right)_{|I|\leq 2r}(t,y),  \left(\partial_I\bm{g}\right)_{|I|\leq 2r}(0,y)\big)$.
Similarly, the same convention applies to $\partial_I \overline{\bm{F}}^a_b(t,0,y,\bm{\theta}_0(t,y))$.
Generally speaking, there are three conditions imposed to $\delta$ and $R$:
\begin{enumerate}
\item To validate the form $\bm{F}\big(t,s,y,\left(\partial_I\bm{u}\right)_{|I|\leq 2r}(t,s,y),  \left(\partial_I\bm{u}\right)_{|I|\leq 2r}(s,s,y)\big)$, it requires that the range of various derivatives of $\bm{u}$ in $\bm{\mathcal{U}}$ is contained in $B(\overline{z},R_0)$. Specifically, since
\begin{equation} \label{Range of u}
	\sup\limits_{\Delta[0,\delta]\times\mathbb{R}^d}\sum_{|I|\leq 2r,b\leq m}\left(\left|\partial_I\bm{u}^b(t,s,y)-\partial_I\bm{g}^b(t,y)\right|+\left|\partial_I\bm{u}^b(s,s,y)-\partial_I\bm{g}^b(s,y)\right|\right)\leq C\delta^\frac{\alpha}{2r}R,
\end{equation}
it should hold that $C\delta^\frac{\alpha}{2r}R\leq R_0/2$; 
\item To ensure that $\bm{\Lambda}$ is a $\frac{1}{2}$-contraction, we need to show that
\begin{equation} \label{eq:Lambdaineq}
	\|\bm{\Lambda}(\bm{u})-\bm{\Lambda}(\widehat{\bm{u}})\|^{(2r+\alpha)}_{[0,\delta]}\leq C(R)\delta^\frac{\alpha}{2r}\|\bm{u}-\widehat{\bm{u}}\|^{(2r+\alpha)}_{[0,\delta]},
\end{equation} 
which requires $C(R)\delta^\frac{\alpha}{2r}\leq\frac{1}{2}$; 
\item Before applying the Banach fixed point theorem, we need to prove that $\bm{\Lambda}$ maps $\bm{\mathcal{U}}$ into itself, i.e. $\|\bm{\Lambda}(\bm{u})-\bm{g}\|^{(2r+\alpha)}_{[0,\delta]}\leq R$. Hence, $R$ should be suitably large such that $\| \bm{\Lambda}(\bm{g})-\bm{g}\|^{(2r+\alpha)}_{[0,\delta]}\leq R/2$. 
\end{enumerate}

First, it is clear that the range of the derivatives of $\bm{u}$ in $\bm{\mathcal{U}}$ is contained in $B(\overline{z},R_0)$ because of \eqref{Range of u} as well as the fact that the range of $\bm{g}$ is contained in $B(\overline{z},R_0/2)$.

Next, we are to show that the operator $\bm{\Lambda}(\bm{u})$ is a contraction defined in $\bm{\mathcal{U}}$, i.e. for any $\bm{u}$, $\widehat{\bm{u}}\in\bm{\mathcal{U}}$, \eqref{eq:Lambdaineq} holds.
Let us consider the equation for $\bm{U}-\widehat{\bm{U}}:=\Lambda(\bm{u})-\Lambda(\widehat{\bm{u}})$, satisfying 
\begin{equation} 
\left\{
\begin{array}{rcl}
	\big(\bm{U}-\widehat{\bm{U}}\big)_s(t,s,y) & = & \bm{L}_0\big(\bm{U}-\widehat{\bm{U}}\big)+\bm{F}\big(t,s,y,\left(\partial_I\bm{u}\right)_{|I|\leq 2r}(t,s,y),  \left(\partial_I\bm{u}\right)_{|I|\leq 2r}(s,s,y)\big) \\
	&& \quad -\bm{F}\big(t,s,y,\left(\partial_I\widehat{\bm{u}}\right)_{|I|\leq 2r}(t,s,y),  \left(\partial_I\widehat{\bm{u}}\right)_{|I|\leq 2r}(s,s,y)\big)-\bm{L}_0\left(\bm{u}-\widehat{\bm{u}}\right), \\
	\big(\bm{U}-\widehat{\bm{U}}\big)(t,0,y) & = & \bm{0}, \hfill 0\leq s\leq t\leq \delta,\quad y\in\mathbb{R}^d.
\end{array}
\right. 
\end{equation}

According to the prior estimates \eqref{Estimates of solutions of nonlocal system} and \eqref{Stability analysis of linear system} of nonlocal linear system, we have 
\begin{equation}
\|\bm{U}-\widehat{\bm{U}}\|^{(2r+\alpha)}_{[0,\delta]}\leq C\|\bm{\varphi}\|^{(\alpha)}_{[0,\delta]},
\end{equation}
where the constant $C$ is independent of $\delta$ and the inhomogeneous term $\bm{\varphi}$ is given by
\begin{equation*}
\begin{split}
	\bm{\varphi}(t,s,y)&=\bm{F}\big(t,s,y,\left(\partial_I\bm{u}\right)_{|I|\leq 2r}(t,s,y),  \left(\partial_I\bm{u}\right)_{|I|\leq 2r}(s,s,y)\big)-\bm{F}\big(t,s,y,\left(\partial_I\widehat{\bm{u}}\right)_{|I|\leq 2r}(t,s,y),  \left(\partial_I\widehat{\bm{u}}\right)_{|I|\leq 2r}(s,s,y)\big)-\bm{L}_0\left(\bm{u}-\widehat{\bm{u}}\right), 
\end{split}
\end{equation*}
whose $a$-th entry, $\bm{\varphi}^a(t,s,y)$ for $a=1,\ldots,m$, admits an integral representation: 
\begin{equation} \label{Integral represnetation of varphi} 
\begin{split}
	&\int^1_0\frac{d}{d\sigma}\bm{F}^a\big(t,s,y,\bm{\theta}_\sigma(t,s,y)\big)d\sigma-\left(\bm{L}_0\left(\bm{u}-\widehat{\bm{u}}\right)\right)^a \\
	=&\int^1_0\sum\limits_{|I|\leq 2r,b\leq m}\partial_I \bm{F}^a_b\big(t,s,y,\bm{\theta}_\sigma(t,s,y)\big)\cdot\left(\partial_I\bm{u}^b(t,s,y)-\partial_I\widehat{\bm{u}}^b(t,s,y)\right)d\sigma \\
	&\quad +\int^1_0\sum\limits_{|I|\leq 2r,b\leq m}\partial_I \overline{\bm{F}}^a_b\big(t,s,y,\bm{\theta}_\sigma(t,s,y)\big)\cdot\left(\partial_I\bm{u}^b(s,s,y)-\partial_I\widehat{\bm{u}}^b(s,s,y)\right)d\sigma-\left(\bm{L}_0\left(\bm{u}-\widehat{\bm{u}}\right)\right)^a \\
	=&\int^1_0\sum\limits_{|I|\leq 2r,b\leq m}\Big(\partial_I \bm{F}^a_b\big(t,s,y,\bm{\theta}_\sigma(t,s,y)\big)-\partial_I \bm{F}^a_b\big(t,0,y,\bm{\theta}_0(t,y)\big)\Big)\cdot\partial_I\left(\bm{u}-\widehat{\bm{u}}\right)^b(t,s,y)d\sigma \\
	&\quad +\int^1_0\sum\limits_{|I|\leq 2r,b\leq m}\Big(\partial_I \overline{\bm{F}}^a_b\big(t,s,y,\bm{\theta}_\sigma(t,s,y)\big)-\partial_I \overline{\bm{F}}^a_b\big(t,0,y,\bm{\theta}_0(t,y)\big)\Big)\cdot\partial_I\left(\bm{u}-\widehat{\bm{u}}\right)^b(s,s,y)d\sigma,  
\end{split}
\end{equation} 
in which $\bm{\theta}_\sigma(t,s,y):=\sigma\big(\left(\partial_I\bm{u}\right)_{|I|\leq 2r}(t,s,y),  \left(\partial_I\bm{u}\right)_{|I|\leq 2r}(s,s,y)\big)+(1-\sigma)\big(\left(\partial_I\widehat{\bm{u}}\right)_{|I|\leq 2r}(t,s,y),  \left(\partial_I\widehat{\bm{u}}\right)_{|I|\leq 2r}(s,s,y)\big)$.

In order to obtain $\|\bm{\varphi}\|^{(\alpha)}_{[0,\delta]}$, we need to estimate $|\bm{\varphi}^a(t,\cdot,\cdot)|^{(\alpha)}_{[0,t]\times\mathbb{R}^d}$ and $|\bm{\varphi}^a_t(t,\cdot,\cdot)|^{(\alpha)}_{[0,t]\times\mathbb{R}^d}$ for any $t\in[0,\delta]$ and $a=1,\ldots,m$. 

\ \ 

\noindent \textbf{(Estimates of $|\bm{\varphi}^a(t,\cdot,\cdot)|^{(\alpha)}_{[0,t]\times\mathbb{R}^d}$).} In the investigation of the difference $|\bm{\varphi}^a(t,s,y)-\bm{\varphi}^a(t,s^\prime,y)|$ for any $0\leq s\leq s^\prime\leq t\leq \delta\leq T$ and $y\in\mathbb{R}^d$, it is convenient to add and subtract 
\begin{equation*}
\int^1_0\sum\limits_{|I|\leq 2r,b\leq m}\partial_I \bm{F}^a_b\big(t,s^\prime,y,\bm{\theta}_\sigma(t,s^\prime,y)\big)\cdot\partial_I\left(\bm{u}-\widehat{\bm{u}}\right)^b(t,s,y)d\sigma
+\int^1_0\sum\limits_{|I|\leq 2r,b\leq m}\partial_I \overline{\bm{F}}^a_b\big(t,s^\prime,y,\bm{\theta}_\sigma(t,s^\prime,y)\big)\cdot\partial_I\left(\bm{u}-\widehat{\bm{u}}\right)^b(s,s,y)d\sigma.
\end{equation*}
Subsequently, we ought to estimate 
\begin{align*}
\big|\partial_I\bm{F}^a_b\big(t,s,y,\bm{\theta}_\sigma(t,s,y)\big)-\partial_I\bm{F}^a_b\big(t,s^\prime,y,\bm{\theta}_\sigma(t,s^\prime,y)\big)\big|,~\big|\partial_I\overline{\bm{F}}^a_b\big(t,s,y,\bm{\theta}_\sigma(t,s,y)\big)-\partial_I\overline{\bm{F}}^a_b\big(t,s^\prime,y,\bm{\theta}_\sigma(t,s^\prime,y)\big)\big|, \\
\big|\partial_I\bm{F}^a_b\big(t,s^\prime,y,\bm{\theta}_\sigma(t,s^\prime,y)\big)-\partial_I\bm{F}^a_b\big(t,0,y,\bm{\theta}_0(t,y)\big)\big|, \text{ and } \big|\partial_I\overline{\bm{F}}^a_b\big(t,s^\prime,y,\bm{\theta}_\sigma(t,s^\prime,y)\big)-\partial_I\overline{\bm{F}}^a_b\big(t,0,y,\bm{\theta}_0(t,y)\big)\big|,  
\end{align*} 
where $|I|\leq 2r$ and $a,b=1,\ldots,m$. Note that
\begin{equation*}
\begin{split}
	&\big|\partial_I\bm{F}^a_b\big(t,s,y,\bm{\theta}_\sigma(t,s,y)\big)-\partial_I\bm{F}^a_b\big(t,s^\prime,y,\bm{\theta}_\sigma(t,s^\prime,y)\big)\big| \\
	\leq &~K(s^\prime-s)^\frac{\alpha}{2r}+L\sum\limits_{b\leq m}\left(|\bm{u}^b(t,\cdot,\cdot)|^{(2r+\alpha)}_{[0,t]\times\mathbb{R}^d}(s^\prime-s)^\frac{\alpha}{2r}+\sup\limits_{\overline{s}\in(s,s^\prime)}|\bm{u}^b_t(\overline{s},\cdot,\cdot)|^{(2r+\alpha)}_{[0,\overline{s}]\times\mathbb{R}^d}(s^\prime-s)^1\right. \\
	&\qquad\qquad\qquad\qquad\qquad+|\bm{u}^b(s^\prime,\cdot,\cdot)|^{(2r+\alpha)}_{[0,s^\prime]\times\mathbb{R}^d}(s^\prime-s)^\frac{\alpha}{2r} +|\widehat{\bm{u}}^b(t,\cdot,\cdot)|^{(2r+\alpha)}_{[0,t]\times\mathbb{R}^d}(s^\prime-s)^\frac{\alpha}{2r} \\
	&\left.\qquad\qquad\qquad\qquad\qquad+\sup\limits_{\overline{s}\in(s,s^\prime)}|\widehat{\bm{u}}^b_t(\overline{s},\cdot,\cdot)|^{(2r+\alpha)}_{[0,\overline{s}]\times\mathbb{R}^d}(s^\prime-s)^1+|\widehat{\bm{u}}^b(s^\prime,\cdot,\cdot)|^{(2r+\alpha)}_{[0,s^\prime]\times\mathbb{R}^d}(s^\prime-s)^\frac{\alpha}{2r}\right)\\
	\leq &\left(K+L\left(\| \bm{u}\|^{(2r+\alpha)}_{[0,\delta]}+\| \widehat{\bm{u}}\|^{(2r+\alpha)}_{[0,\delta]}\right)\right)(s^\prime-s)^\frac{\alpha}{2r} \leq C_1(R)(s^\prime-s)^\frac{\alpha}{2r}
\end{split} 
\end{equation*}
and
\begin{equation*}
\begin{split}
	&\big|\partial_I\bm{F}^a_b\big(t,s^\prime,y,\bm{\theta}_\sigma(t,s^\prime,y)\big)-\partial_I\bm{F}^a_b\big(t,0,y,\bm{\theta}_0(t,y)\big)\big| \\
	\leq &~K(s^\prime-0)^\frac{\alpha}{2r}+L\sum\limits_{b\leq m}\left(|\left(\bm{u}-\bm{g}\right)^b(t,\cdot,\cdot)|^{(2r+\alpha)}_{[0,t]\times\mathbb{R}^d}(s^\prime-0)^\frac{\alpha}{2r}+\sup\limits_{\overline{s}\in(0,s^\prime)}|\bm{g}^b_t(\overline{s},\cdot)|^{(2r+\alpha)}_{\mathbb{R}^d}(s^\prime-0)^1+|\bm{u}^b(s^\prime,\cdot,\cdot)|^{(2r+\alpha)}_{[0,s^\prime]\times\mathbb{R}^d}(s^\prime-0)^\frac{\alpha}{2r}\right. \\
	&\left.\qquad\qquad\qquad\qquad\qquad +|\left(\widehat{\bm{u}}-\bm{g}\right)^b(t,\cdot,\cdot)|^{(2r+\alpha)}_{[0,t]\times\mathbb{R}^d}(s^\prime-0)^\frac{\alpha}{2r}+\sup\limits_{\overline{s}\in(0,s^\prime)}|\bm{g}^b_t(\overline{s},\cdot)|^{(2r+\alpha)}_{\mathbb{R}^d}(s^\prime-0)^1+|\widehat{\bm{u}}^b(s^\prime,\cdot,\cdot)|^{(2r+\alpha)}_{[0,s^\prime]\times\mathbb{R}^d}(s^\prime-0)^\frac{\alpha}{2r}\right) \\
	\leq & \left(K+L\left(\| \bm{u}-\bm{g}\|^{(2r+\alpha)}_{[0,\delta]}+\| \widehat{\bm{u}}-\bm{g}\|^{(2r+\alpha)}_{[0,\delta]}+\| \bm{g}\|^{(2r+\alpha)}_{[0,\delta]}\right)\right)(s^\prime-0)^\frac{\alpha}{2r} \\
	\leq & ~C_2(R)\delta^\frac{\alpha}{2r} 
\end{split}
\end{equation*}
where $L>0$ is a constant which can be different from line to line and the subscripts of $C$ are to represent different constant values within the derivation. In a similar manner, we can obtain
\begin{equation*}
\begin{split}
	\big|\partial_I\overline{\bm{F}}^a_b\big(t,s,y,\bm{\theta}_\sigma(t,s,y)\big)-\partial_I\overline{\bm{F}}^a_b\big(t,s^\prime,y,\bm{\theta}_\sigma(t,s^\prime,y)\big)\big|&\leq C_3(R)(s^\prime-s)^\frac{\alpha}{2r}, \\
	\big|\partial_I\overline{\bm{F}}^a_b\big(t,s^\prime,y,\bm{\theta}_\sigma(t,s^\prime,y)\big)-\partial_I\overline{\bm{F}}^a_b\big(t,0,y,\bm{\theta}_0(t,y)\big)\big|&\leq C_4(R)\delta^\frac{\alpha}{2r}. 
\end{split} 
\end{equation*}
Consequently, it holds that 
\begin{equation} \label{Holder continuity of varphi with respect to s}
\begin{split}
	&\big|\bm{\varphi}^a(t,s,y)-\bm{\varphi}^a(t,s^\prime,y)\big| \\
	\leq & \int^1_0\sum\limits_{|I|\leq 2r,b\leq m}\Big|\partial_I\bm{F}^a_b\big(t,s,y,\bm{\theta}_\sigma(t,s,y)\big)-\partial_I\bm{F}^a_b\big(t,s^\prime,y,\bm{\theta}_\sigma(t,s^\prime,y)\big)\Big|\cdot\Big|\partial_I\left(\bm{u}-\widehat{\bm{u}}\right)^b(t,s,y)\Big|d\sigma \\
	& + \int^1_0\sum\limits_{|I|\leq 2r,b\leq m}\Big|\partial_I\overline{\bm{F}}^a_b\big(t,s,y,\bm{\theta}_\sigma(t,s,y)\big)-\partial_I\overline{\bm{F}}^a_b\big(t,s^\prime,y,\bm{\theta}_\sigma(t,s^\prime,y)\big)\Big|\cdot\Big|\partial_I\left(\bm{u}-\widehat{\bm{u}}\right)^b(s,s,y)\Big|d\sigma \\
	& + \int^1_0\sum\limits_{|I|\leq 2r,b\leq m}\Big|\partial_I\bm{F}^a_b\big(t,s^\prime,y,\bm{\theta}_\sigma(t,s^\prime,y)\big)-\partial_I\bm{F}^a_b\big(t,0,y,\bm{\theta}_0(t,y)\big)\Big|\Big|\partial_I\left(\bm{u}-\widehat{\bm{u}}\right)^b(t,s,y)-\partial_I\left(\bm{u}-\widehat{\bm{u}}\right)^b(t,s^\prime,y)\Big|d\sigma \\
	&+\int^1_0\sum\limits_{|I|\leq 2r,b\leq m}\Big|\partial_I\overline{\bm{F}}^a_b\big(t,s^\prime,y,\bm{\theta}_\sigma(t,s^\prime,y)\big)-\partial_I\overline{\bm{F}}^a_b\big(t,0,y,\bm{\theta}_0(t,y)\big)\Big|\Big|\partial_I\left(\bm{u}-\widehat{\bm{u}}\right)^b(s,s,y)-\partial_I\left(\bm{u}-\widehat{\bm{u}}\right)^b(s^\prime,s^\prime,y)\Big|d\sigma \\
	\leq &~C_1(R)(s^\prime-s)^\frac{\alpha}{2r}\delta^\frac{\alpha}{2r}\sum\limits_{b\leq m}\left|\partial_I\left(\bm{u}-\widehat{\bm{u}}\right)^b(t,\cdot,\cdot)\right|^{(\alpha)}_{[0,t]\times\mathbb{R}^d}+C_2(R)\delta^\frac{\alpha}{2r}(s^\prime-s)^\frac{\alpha}{2r}\sum\limits_{b\leq m}\left|\partial_I\left(\bm{u}-\widehat{\bm{u}}\right)^b(t,\cdot,\cdot)\right|^{(\alpha)}_{[0,t]\times\mathbb{R}^d} \\
	&+C_3(R)(s^\prime-s)^\frac{\alpha}{2r}\delta^\frac{\alpha}{2r}\sum\limits_{b\leq m}\left|\partial_I\left(\bm{u}-\widehat{\bm{u}}\right)^b(s,\cdot,\cdot)\right|^{(\alpha)}_{[0,s]\times\mathbb{R}^d} \\
	&+C_4(R)\delta^\frac{\alpha}{2r}(s^\prime-s)^\frac{\alpha}{2r}\sum\limits_{b\leq m}\left(\sup\limits_{\overline{s}\in(s,s^\prime)}\left|\partial_I\left(\bm{u}-\widehat{\bm{u}}\right)^b_t(\overline{s},\cdot,\cdot)\right|^{(\alpha)}_{[0,\overline{s}]\times\mathbb{R}^d}+\left|\partial_I\left(\bm{u}-\widehat{\bm{u}}\right)^b(s^\prime,\cdot,\cdot)\right|^{(\alpha)}_{[0,s^\prime]\times\mathbb{R}^d}\right)       \\
	\leq &~ C_5(R)\delta^{\frac{\alpha}{2r}}(s^\prime-s)^{\frac{\alpha}{2r}}\| \bm{u}-\widehat{\bm{u}}\|^{(2r+\alpha)}_{[0,\delta]}, 
\end{split}
\end{equation}
which implies the following by noting that $\bm{\varphi}(t,0,y)\equiv 0$,
\begin{equation} \label{Boundness of varphi} 
|\bm{\varphi}^a(t,\cdot,\cdot)|^{\infty}_{[0,t]\times\mathbb{R}^d}\leq C_5(R)\delta^\frac{\alpha}{r}\| \bm{u}-\widehat{\bm{u}}\|^{(2r+\alpha)}_{[0,\delta]}.
\end{equation}

To estimate $|\bm{\varphi}^a(t,s,y)-\bm{\varphi}^a(t,s,y^\prime)|$, it is convenient to add and subtract 
\begin{equation*}
\begin{split}
	&\int^1_0\sum\limits_{|I|\leq 2r,b\leq m}\Big(\partial_I \bm{F}^a_b\big(t,s,y^\prime,\bm{\theta}_\sigma(t,s,y^\prime)\big)-\partial_I\bm{F}^a_b\big(t,0,y^\prime,\bm{\theta}_0(t,y^\prime)\big)\Big)\cdot\partial_I\left(\bm{u}-\widehat{\bm{u}}\right)^b(t,s,y)d\sigma \\
	+&\int^1_0\sum\limits_{|I|\leq 2r,b\leq m}\Big(\partial_I \overline{\bm{F}}^a_b\big(t,s,y^\prime,\bm{\theta}_\sigma(t,s,y^\prime)\big)-\partial_I\overline{\bm{F}}^a_b\big(t,0,y^\prime,\bm{\theta}_0(t,y^\prime)\big)\Big)\cdot\partial_I\left(\bm{u}-\widehat{\bm{u}}\right)^b(s,s,y)d\sigma.
\end{split}
\end{equation*}
Note that
\begin{equation*}
\begin{split}
	&\left|\partial_I \bm{F}^a_b\big(t,s,y,\bm{\theta}_\sigma(t,s,y)\big)-\partial_I \bm{F}^a_b\big(t,s,y^\prime,\bm{\theta}_\sigma(t,s,y^\prime)\big)\right|+\left|\partial_I\bm{F}^a_b\big(t,0,y,\bm{\theta}_0(t,y)\big)-\partial_I\bm{F}^a_b\big(t,0,y^\prime,\bm{\theta}_0(t,y^\prime)\big)\right| \\ 
	\leq &~2K|y-y^\prime|^\alpha+L\left|\bm{\theta}_\sigma(t,s,y)-\bm{\theta}_\sigma(t,s,y^\prime)\right|+L\left|\bm{\theta}_0(t,y)-\bm{\theta}_0(t,y^\prime)\right| \\
	\leq & 2K|y-y^\prime|^\alpha+L|y-y^\prime|^\alpha\sum\limits_{b\leq m}\Bigg(|\bm{u}^b(t,\cdot,\cdot)|^{(2r+\alpha)}_{[0,t]\times\mathbb{R}^d}+\left|\bm{u}^b(s,\cdot,\cdot)\right|^{(2r+\alpha)}_{[0,s]\times\mathbb{R}^d}+|\widehat{\bm{u}}^b(t,\cdot,\cdot)|^{(2r+\alpha)}_{[0,t]\times\mathbb{R}^d}\\
	&\qquad\qquad\qquad\qquad\qquad\qquad\quad +\left|\widehat{\bm{u}}^b(s,\cdot,\cdot)\right|^{(2r+\alpha)}_{[0,s]\times\mathbb{R}^d}+|\bm{g}^b(t,\cdot)|^{(2r+\alpha)}_{\mathbb{R}^d}+\left|\bm{g}^b(0,\cdot)\right|^{(2r+\alpha)}_{\mathbb{R}^d}\Bigg) \\
	\leq &\left(2K+L\left(\| \bm{u}\|^{(2r+\alpha)}_{[0,\delta]}+\| \widehat{\bm{u}}\|^{(2r+\alpha)}_{[0,\delta]}+\| \bm{g}\|^{(2r+\alpha)}_{[0,\delta]}\right)\right)|y-y^\prime|^\alpha \leq C_6(R)|y-y^\prime|^\alpha, 
\end{split}
\end{equation*}	
and for every $y\in\mathbb{R}^d$, 
\begin{equation*}
\begin{split}
	&\big|\partial_I \bm{F}^a_b\big(t,s,y,\bm{\theta}_\sigma(t,s,y)\big)-\partial_I \bm{F}^a_b\big(t,0,y,\bm{\theta}_0(t,y)\big)\big| \leq K(s-0)^\frac{\alpha}{2}+L\left|\bm{\theta}_\sigma(t,s,y)-\bm{\theta}_0(t,y)\right| \\
	\leq &~K(s-0)^\frac{\alpha}{2}+L\sum\limits_{b\leq m}\left(|\left(\bm{u}-\bm{g}\right)^b(t,\cdot,\cdot)|^{(2r+\alpha)}_{[0,t]\times\mathbb{R}^d}(s-0)^\frac{\alpha}{2r}+\sup\limits_{\overline{s}\in(0,s)}\bm{g}^b_t(\overline{s},\cdot)|^{(2r+\alpha)}_{\mathbb{R}^d}(s-0)^1+|\bm{u}^b(s,\cdot,\cdot)|^{(2r+\alpha)}_{[0,s]\times\mathbb{R}^d}(s-0)^\frac{\alpha}{2r}\right. \\
	&\qquad\qquad\qquad\qquad\quad \left.+|\left(\widehat{\bm{u}}-\bm{g}\right)^b(t,\cdot,\cdot)|^{(2r+\alpha)}_{[0,t]\times\mathbb{R}^d}(s-0)^\frac{\alpha}{2r}+\sup\limits_{\overline{s}\in(0,s)}|\bm{g}^b_t(\overline{s},\cdot)|^{(2r+\alpha)}_{\mathbb{R}^d}(s-0)^1+|\widehat{\bm{u}}^b(s,\cdot,\cdot)|^{(2r+\alpha)}_{[0,s]\times\mathbb{R}^d}(s-0)^\frac{\alpha}{2r}\right) \\
	\leq & \left(K+L\left(\| \bm{u}-\bm{g}\|^{(2r+\alpha)}_{[0,\delta]}+\| \widehat{\bm{u}}-\bm{g}\|^{(2r+\alpha)}_{[0,\delta]}+\| \bm{g}\|^{(2r+\alpha)}_{[0,\delta]}\right)\right)(s-0)^\frac{\alpha}{2r} \leq C_7(R)\delta^\frac{\alpha}{2r}.
\end{split}
\end{equation*} 
Similarly, we also have 
\begin{eqnarray*}
\left|\partial_I \overline{\bm{F}}^a_b\big(t,s,y,\bm{\theta}_\sigma(t,s,y)\big)-\partial_I \overline{\bm{F}}^a_b\big(t,s,y^\prime,\bm{\theta}_\sigma(t,s,y^\prime)\big)\right|+\left|\partial_I\overline{\bm{F}}^a_b\big(t,0,y,\bm{\theta}_0(t,y)\big)-\partial_I\overline{\bm{F}}^a_b\big(t,0,y^\prime,\bm{\theta}_0(t,y^\prime)\big)\right| & \leq & C_8(R)|y-y^\prime|^\alpha, \\
\big|\partial_I \overline{\bm{F}}^a_b\big(t,s,y,\bm{\theta}_\sigma(t,s,y)\big)-\partial_I \overline{\bm{F}}^a_b\big(t,0,y,\bm{\theta}_0(t,y)\big)\big| & \leq & C_{9}(R)\delta^\frac{\alpha}{2r}.
\end{eqnarray*}
Hence, we have 
\begin{equation} \label{Holder continuity of varphi with respect to y}
\begin{split}
	&\big|\bm{\varphi}^a(t,s,y)-\bm{\varphi}^a(t,s,y^\prime)\big| \\
	\leq &~ C_6(R)|y-y^\prime|^\alpha\delta^\frac{\alpha}{2r}\sum\limits_{b\leq m}\left|\partial_I\left(\bm{u}-\widehat{\bm{u}}\right)^b(t,\cdot,\cdot)\right|^{(\alpha)}_{[0,t]\times\mathbb{R}^d}+C_7(R)\delta^\frac{\alpha}{2r}|y-y^\prime|^\alpha\sum\limits_{b\leq m}\left|\partial_I\left(\bm{u}-\widehat{\bm{u}}\right)^b(t,\cdot,\cdot)\right|^{(\alpha)}_{[0,t]\times\mathbb{R}^d} \\
	&+C_8(R)|y-y^\prime|^\alpha\delta^\frac{\alpha}{2r}\sum\limits_{b\le m}\left|\partial_I\left(\bm{u}-\widehat{\bm{u}}\right)^b(s,\cdot,\cdot)\right|^{(\alpha)}_{[0,t]\times\mathbb{R}^d}+C_{9}(R)\delta^\frac{\alpha}{2r}|y-y^\prime|^\alpha\sum\limits_{b\leq m}\left|\partial_I\left(\bm{u}-\widehat{\bm{u}}\right)^b(s,\cdot,\cdot)\right|^{(\alpha)}_{[0,t]\times\mathbb{R}^d} \\
	\leq &~C_{10}(R)\delta^\frac{\alpha}{2r}|y-y^\prime|^\alpha\| \bm{u}-\widehat{\bm{u}}\|^{(2r+\alpha)}_{[0,\delta]}.  
\end{split}
\end{equation}

From \eqref{Holder continuity of varphi with respect to s}, \eqref{Boundness of varphi} and \eqref{Holder continuity of varphi with respect to y}, for any $t\in[0,\delta]$ and $a=1,\ldots,m$, we obtain
\begin{equation} \label{Estimate of varphi}
|\bm{\varphi}^a(t,\cdot,\cdot)|^{(\alpha)}_{[0,t]\times\mathbb{R}^d}\leq C_{11}(R)\delta^{\frac{\alpha}{2r}}\| \bm{u}-\widehat{\bm{u}}\|^{(2r+\alpha)}_{[0,\delta]}.
\end{equation}

\ \ 

\noindent (\textbf{Estimates of $|\bm{\varphi}^a_t(t,\cdot,\cdot)|^{(\alpha)}_{[0,t]\times\mathbb{R}^d}$.}) We now analyze the H\"{o}lder continuity of $\bm{\varphi}^a_t(t,\cdot,\cdot)$ with respect to $s$ and $y$ in $[0,t]\times\mathbb{R}^d$. According to the integral representation  of $\bm{\varphi}^a(t,s,y)$ \eqref{Integral represnetation of varphi}, we have
\begin{equation*}
\begin{split}
	\bm{\varphi}^a_t(t,s,y) = & \int^1_0\sum\limits_{|I|\leq 2r,b\leq m}\Bigg[\frac{\partial\Big(\partial_I \bm{F}^a_b\big(t,s,y,\bm{\theta}_\sigma(t,s,y)\big)-\partial_I \bm{F}^a_b\big(t,0,y,\bm{\theta}_0(t,y)\big)\Big)}{\partial t}\cdot\partial_I\left(\bm{u}-\widehat{\bm{u}}\right)^b(t,s,y) \\
	&~
	+\Big(\partial_I \bm{F}^a_b\big(t,s,y,\bm{\theta}_\sigma(t,s,y)\big)-\partial_I \bm{F}^a_b\big(t,0,y,\bm{\theta}_0(t,y)\big)\Big)\cdot\partial_I\left(\bm{u}_t-\widehat{\bm{u}}_t\right)^b(t,s,y)\Bigg]d\sigma \\
	&~+\int^1_0\sum\limits_{|I|\leq 2r,b\leq m}\frac{\partial\Big(\partial_I \overline{\bm{F}}^a_b\big(t,s,y,\bm{\theta}_\sigma(t,s,y)\big)-\partial_I \overline{\bm{F}}^a_b\big(t,0,y,\bm{\theta}_0(t,y)\big)\Big)}{\partial t}\cdot\partial_I\left(\bm{u}-\widehat{\bm{u}}\right)^b(s,s,y)d\sigma \\
	=: & \Big\{M_1+M_2+M_3\Big\}+\Big\{M_4+M_5\Big\},
\end{split}
\end{equation*}
where
\begin{equation*}
\begin{split}
	M_1= & \int^1_0\sum\limits_{|I|\leq 2r,b\leq m}\Big(\partial^2_{tI} \bm{F}^a_b\big(t,s,y,\bm{\theta}_\sigma(t,s,y)\big)-\partial^2_{tI} \bm{F}^a_b\big(t,0,y,\bm{\theta}_0(t,y)\big)\Big)\cdot\partial_I\left(\bm{u}-\widehat{\bm{u}}\right)^b(t,s,y)d\sigma \\
	M_2= & \int^1_0\sum\limits_{|I|\leq 2r,b\leq m}\Bigg[\sum\limits_{|J|\leq 2r,c\leq m}\Bigg(\partial^2_{IJ}\bm{F}^a_{bc}\big(t,s,y,\bm{\theta}_\sigma(t,s,y)\big)\cdot\Big(\sigma\partial_J\bm{u}^c_t(t,s,y)+(1-\sigma)\partial_J\widehat{\bm{u}}^c_t(t,s,y)\Big) \\
	&\qquad\qquad\qquad\qquad\qquad\quad
	-\partial^2_{IJ}\bm{F}^a_{bc}\big(t,0,y,\bm{\theta}_0(t,y)\big)\cdot\partial_J\bm{g}^c_t(t,y)\Bigg)\Bigg]\cdot\partial_I\left(\bm{u}-\widehat{\bm{u}}\right)^b(t,s,y)d\sigma \\
	M_3= & \int^1_0\sum\limits_{|I|\leq 2r,b\leq m}\Big(\partial_I \bm{F}^a_b\big(t,s,y,\bm{\theta}_\sigma(t,s,y)\big)-\partial_I \bm{F}^a_b\big(t,0,y,\bm{\theta}_0(t,y)\big)\Big)\cdot\partial_I\left(\bm{u}_t-\widehat{\bm{u}}_t\right)^b(t,s,y)d\sigma \\
	M_4= & \int^1_0\sum\limits_{|I|\leq 2r,b\leq m}\Big(\partial^2_{tI} \overline{\bm{F}}^a_b\big(t,s,y,\bm{\theta}_\sigma(t,s,y)\big)-\partial^2_{tI} \overline{\bm{F}}^a_b\big(t,0,y,\bm{\theta}_0(t,y)\big)\Big)\cdot\partial_I\left(\bm{u}-\widehat{\bm{u}}\right)^b(s,s,y)d\sigma \\
	M_5= & \int^1_0\sum\limits_{|I|\leq 2r,b\leq m}\Bigg[\sum\limits_{|J|\leq 2r,c\leq m}\Bigg(\partial^2_{IJ}\overline{\bm{F}}^a_{bc}\big(t,s,y,\bm{\theta}_\sigma(t,s,y)\big)\cdot\Big(\sigma\partial_J\bm{u}^c_t(t,s,y)+(1-\sigma)\partial_J\widehat{\bm{u}}^c_t(t,s,y)\Big) \\
	&\qquad\qquad\qquad\qquad\qquad\quad
	-\partial^2_{IJ}\overline{\bm{F}}^a_{bc}\big(t,0,y,\bm{\theta}_0(t,y)\big)\cdot\partial_J\bm{g}^c_t(t,y)\Bigg)\Bigg]\cdot\partial_I\left(\bm{u}-\widehat{\bm{u}}\right)^b(s,s,y)d\sigma.
\end{split}
\end{equation*} 
It is easy to see that the estimates of $M_1$, $M_3$, and $M_4$ are similar to the terms of $|\bm{\varphi}^a(t,\cdot,\cdot)|^{(\alpha)}_{[0,t]\times\mathbb{R}^d}$. Hence, we focus on the remaining two terms $M_2$ and $M_5$. We denote $\bm{\eta}^a(t,s,y)=M_2+M_5$.

In order to estimate $|\bm{\eta}^a(t,s,y)-\bm{\eta}^a(t,s^\prime,y)|$ for $0\leq s\leq s^\prime\leq t\leq \delta\leq T$ and any $y\in\mathbb{R}^d$, it is convenient to add and subtract  
\begin{equation*}
\begin{split}
	&\int^1_0\sum\limits_{|I|\leq 2r,b\leq m}\Bigg[\sum\limits_{|J|\leq 2r,c\leq m}\Bigg(\partial^2_{IJ}\bm{F}^a_{bc}\big(t,s^\prime,y,\bm{\theta}_\sigma(t,s^\prime,y)\big)\cdot\Big(\sigma\partial_J\bm{u}^c_t(t,s^\prime,y)+(1-\sigma)\partial_J\widehat{\bm{u}}^c_t(t,s^\prime,y)\Big)\Bigg)\Bigg]\partial_I\left(\bm{u}-\widehat{\bm{u}}\right)^b(t,s,y)d\sigma \\
	+ & \int^1_0\sum\limits_{|I|\leq 2r,b\leq m}\Bigg[\sum\limits_{|J|\leq 2r,c\leq m}\Bigg(\partial^2_{IJ}\overline{\bm{F}}^a_{bc}\big(t,s^\prime,y,\bm{\theta}_\sigma(t,s^\prime,y)\big)\cdot\Big(\sigma\partial_J\bm{u}^c_t(t,s^\prime,y)+(1-\sigma)\partial_J\widehat{\bm{u}}^c_t(t,s^\prime,y)\Big)\Bigg)\Bigg]\partial_I\left(\bm{u}-\widehat{\bm{u}}\right)^b(s,s,y)d\sigma. 
\end{split}
\end{equation*}
Subsequently, we need to estimate
\begin{align}
& \bigg|\partial^2_{IJ}\bm{F}^a_{bc}\big(t,s,y,\bm{\theta}_\sigma(t,s,y)\big)\cdot\Big(\sigma\partial_J\bm{u}^c_t(t,s,y)+(1-\sigma)\partial_J\widehat{\bm{u}}^c_t(t,s,y)\Big) \nonumber \\
& \qquad\qquad\qquad -\partial^2_{IJ}\bm{F}^a_{bc}\big(t,s^\prime,y,\bm{\theta}_\sigma(t,s^\prime,y)\big)\cdot\Big(\sigma\partial_J\bm{u}^c_t(t,s^\prime,y)+(1-\sigma)\partial_J\widehat{\bm{u}}^c_t(t,s^\prime,y)\Big)\bigg|, \label{eq:1stF-1stterm} \\
&\bigg|\partial^2_{IJ}\overline{\bm{F}}^a_{bc}\big(t,s,y,\bm{\theta}_\sigma(t,s,y)\big)\cdot\Big(\sigma\partial_J\bm{u}^c_t(t,s,y)+(1-\sigma)\partial_J\widehat{\bm{u}}^c_t(t,s,y)\Big) \nonumber \\
&\qquad\qquad\qquad-\partial^2_{IJ}\overline{\bm{F}}^a_{bc}\big(t,s^\prime,y,\bm{\theta}_\sigma(t,s^\prime,y)\big)\cdot\Big(\sigma\partial_J\bm{u}^c_t(t,s^\prime,y)+(1-\sigma)\partial_J\widehat{\bm{u}}^c_t(t,s^\prime,y)\Big)\bigg|, \label{eq:1stbarF-1stterm} \\
&\bigg|\partial^2_{IJ}\bm{F}^a_{bc}\big(t,s^\prime,y,\bm{\theta}_\sigma(t,s^\prime,y)\big)\cdot\Big(\sigma\partial_J\bm{u}^c_t(t,s^\prime,y)+(1-\sigma)\partial_J\widehat{\bm{u}}^c_t(t,s^\prime,y)\Big)-\partial^2_{IJ}\bm{F}^a_{bc}\big(t,0,y,\bm{\theta}_0(t,y)\big)\cdot\partial_J\bm{g}^c_t(t,y)\bigg|, \label{eq:1stF-2ndterm} \\
&\bigg|\partial^2_{IJ}\overline{\bm{F}}^a_{bc}\big(t,s^\prime,y,\bm{\theta}_\sigma(t,s^\prime,y)\big)\cdot\Big(\sigma\partial_J\bm{u}^c_t(t,s^\prime,y)+(1-\sigma)\partial_J\widehat{\bm{u}}^c_t(t,s^\prime,y)\Big) -\partial^2_{IJ}\overline{\bm{F}}^a_{bc}\big(t,0,y,\bm{\theta}_0(t,y)\big)\cdot\partial_J\bm{g}^c_t(t,y)\bigg|. \label{eq:1stbarF-2ndterm}
\end{align}
Note that
\begin{equation*}
\begin{split}
	\eqref{eq:1stF-1stterm} \leq &~ \bigg|\partial^2_{IJ}\bm{F}^a_{bc}\big(t,s,y,\bm{\theta}_\sigma(t,s,y)\big)\cdot\Big(\sigma\partial_J\bm{u}^c_t(t,s,y)+(1-\sigma)\partial_J\widehat{\bm{u}}^c_t(t,s,y)\Big) \\
	&\qquad\qquad\qquad
	-\partial^2_{IJ}\bm{F}^a_{bc}\big(t,s^\prime,y,\bm{\theta}_\sigma(t,s^\prime,y)\big)\cdot\Big(\sigma\partial_J\bm{u}^c_t(t,s,y)+(1-\sigma)\partial_J\widehat{\bm{u}}^c_t(t,s,y)\Big)\bigg| \\
	& + \bigg|\partial^2_{IJ}\bm{F}^a_{bc}\big(t,s^\prime,y,\bm{\theta}_\sigma(t,s^\prime,y)\big)\cdot\Big(\sigma\partial_J\bm{u}^c_t(t,s,y)+(1-\sigma)\partial_J\widehat{\bm{u}}^c_t(t,s,y)\Big) \\
	&\qquad\qquad\quad
	-\partial^2_{IJ}\bm{F}^a_{bc}\big(t,s^\prime,y,\bm{\theta}_\sigma(t,s^\prime,y)\big)\cdot\Big(\sigma\partial_J\bm{u}^c_t(t,s^\prime,y)+(1-\sigma)\partial_J\widehat{\bm{u}}^c_t(t,s^\prime,y)\Big)\bigg| \leq C_{12}(R)(s^\prime-s)^\frac{\alpha}{2r}, \\
	\eqref{eq:1stF-2ndterm} \leq & ~ \bigg|\partial^2_{IJ}\bm{F}^a_{bc}\big(t,s^\prime,y,\bm{\theta}_\sigma(t,s^\prime,y)\big)\cdot\Big(\sigma\partial_J\bm{u}^c_t(t,s^\prime,y)+(1-\sigma)\partial_J\widehat{\bm{u}}^c_t(t,s^\prime,y)\Big) -\partial^2_{IJ}\bm{F}^a_{bc}\big(t,s^\prime,y,\bm{\theta}_\sigma(t,s^\prime,y)\big)\cdot\partial_J\bm{g}^c_t(t,y)\bigg| \\
	& \qquad\qquad\qquad\qquad\qquad + \bigg|\partial^2_{IJ}\bm{F}^a_{bc}\big(t,s^\prime,y,\bm{\theta}_\sigma(t,s^\prime,y)\big)\cdot\partial_J\bm{g}^c_t(t,y)-\partial^2_{IJ}\bm{F}^a_{bc}\big(t,0,y,\bm{\theta}_0(t,y)\big)\cdot\partial_J\bm{g}^c_t(t,y)\bigg| \leq C_{13}(R)\delta^\frac{\alpha}{2r}. 
\end{split}
\end{equation*}
Similarly, we also have $\eqref{eq:1stbarF-1stterm} \leq C_{14}(R)(s^\prime-s)^\frac{\alpha}{2r}$ and $\eqref{eq:1stbarF-2ndterm} \leq C_{15}(R)\delta^\frac{\alpha}{2r}$.
Hence, we obtain that 
\begin{equation}
\begin{split}
	\big|\bm{\eta}^a(t,s,y)-\bm{\eta}^a(t,s^\prime,y)\big| \leq &~ C_{12}(R)(s^\prime-s)^\frac{\alpha}{2r}\delta^\frac{\alpha}{2r}\sum\limits_{b\leq m}\left|\partial_I\left(\bm{u}-\widehat{\bm{u}}\right)^b(t,\cdot,\cdot)\right|^{(\alpha)}_{[0,t]\times\mathbb{R}^d}+
	C_{13}(R)\delta^\frac{\alpha}{2r}(s^\prime-s)^\frac{\alpha}{2r}\sum\limits_{b\leq m}\left|\partial_I\left(\bm{u}-\widehat{\bm{u}}\right)^b(t,\cdot,\cdot)\right|^{(\alpha)}_{[0,t]\times\mathbb{R}^d} \\
	& ~+C_{14}(R)(s^\prime-s)^\frac{\alpha}{2r}\delta^\frac{\alpha}{2r}\sum\limits_{b\leq m}\left|\partial_I\left(\bm{u}-\widehat{\bm{u}}\right)^b(s,\cdot,\cdot)\right|^{(\alpha)}_{[0,s]\times\mathbb{R}^d} \\
	& ~+C_{15}(R)\delta^\frac{\alpha}{2r}(s^\prime-s)^\frac{\alpha}{2r}\sum\limits_{b\leq m}\left(\sup\limits_{\overline{s}\in(s,s^\prime)}\left|\partial_I\left(\bm{u}-\widehat{\bm{u}}\right)^b_t(\overline{s},\cdot,\cdot)\right|^{(\alpha)}_{[0,\overline{s}]\times\mathbb{R}^d}+\left|\partial_I\left(\bm{u}-\widehat{\bm{u}}\right)^b(s^\prime,\cdot,\cdot)\right|^{(\alpha)}_{[0,s^\prime]\times\mathbb{R}^d}\right)       \\
	\leq &~ C_{16}(R)\delta^{\frac{\alpha}{2r}}(s^\prime-s)^{\frac{\alpha}{2r}}\| \bm{u}-\widehat{\bm{u}}\|^{(2r+\alpha)}_{[0,\delta]},  
\end{split}
\end{equation}
which implies the following by noting that $\bm{\eta}(t,0,y)\equiv 0$,
\begin{equation} \label{Boundness of eta}  
|\bm{\eta}^a(t,\cdot,\cdot)|^{\infty}_{[0,t]\times\mathbb{R}^d}\leq C_{16}(R)\delta^\frac{\alpha}{r}\| \bm{u}-\widehat{\bm{u}}\|^{(2r+\alpha)}_{[0,\delta]}.
\end{equation}

In order to estimate $|\bm{\eta}^a(t,s,y)-\bm{\eta}^a(t,s,y^\prime)|$ for $0\leq s\leq t\leq \delta\leq T$ and any $y$, $y^\prime\in\mathbb{R}^d$, it is convenient to add and subtract 
\begin{equation*}
\begin{split}
	&\int^1_0\sum\limits_{|I|\leq 2r,b\leq m}\Bigg[\sum\limits_{|J|\leq 2r,c\leq m}\Bigg(\partial^2_{IJ}\bm{F}^a_{bc}\big(t,s,y^\prime,\bm{\theta}_\sigma(t,s,y^\prime)\big)\cdot\Big(\sigma\partial_J\bm{u}^c_t(t,s,y^\prime)+(1-\sigma)\partial_J\widehat{\bm{u}}^c_t(t,s,y^\prime)\Big) \\
	&\qquad\qquad\qquad\qquad\qquad\qquad
	-\partial^2_{IJ}\bm{F}^a_{bc}\big(t,0,y^\prime,\bm{\theta}_0(t,y^\prime)\big)\cdot\partial_J\bm{g}^c_t(t,y^\prime)\Bigg)\Bigg]\cdot\partial_I\left(\bm{u}-\widehat{\bm{u}}\right)^b(t,s,y)d\sigma \\
	+ & \int^1_0\sum\limits_{|I|\leq 2r,b\leq m}\Bigg[\sum\limits_{|J|\leq 2r,c\leq m}\Bigg(\partial^2_{IJ}\overline{\bm{F}}^a_{bc}\big(t,s,y^\prime,\bm{\theta}_\sigma(t,s,y^\prime)\big)\cdot\Big(\sigma\partial_J\bm{u}^c_t(t,s,y^\prime)+(1-\sigma)\partial_J\widehat{\bm{u}}^c_t(t,s,y^\prime)\Big) \\
	&\qquad\qquad\qquad\qquad\qquad\qquad
	-\partial^2_{IJ}\overline{\bm{F}}^a_{bc}\big(t,0,y^\prime,\bm{\theta}_0(t,y^\prime)\big)\cdot\partial_J\bm{g}^c_t(t,y^\prime)\Bigg)\Bigg]\cdot\partial_I\left(\bm{u}-\widehat{\bm{u}}\right)^b(s,s,y)d\sigma.
\end{split}
\end{equation*}
Then we need to estimate the error (for $\bm{F}$)
\begin{equation} \label{eq:2ndF-1stterm}
\begin{split}
	&\bigg|\partial^2_{IJ}\bm{F}^a_{bc}\big(t,s,y,\bm{\theta}_\sigma(t,s,y)\big)\cdot\Big(\sigma\partial_J\bm{u}^c_t(t,s,y)+(1-\sigma)\partial_J\widehat{\bm{u}}^c_t(t,s,y)\Big)-\partial^2_{IJ}\bm{F}^a_{bc}\big(t,0,y,\bm{\theta}_0(t,y)\big)\cdot\partial_J\bm{g}^c_t(t,y) \\
	&\quad -\partial^2_{IJ}\bm{F}^a_{bc}\big(t,s,y^\prime,\bm{\theta}_\sigma(t,s,y^\prime)\big)\cdot\Big(\sigma\partial_J\bm{u}^c_t(t,s,y^\prime)+(1-\sigma)\partial_J\widehat{\bm{u}}^c_t(t,s,y^\prime)\Big)+\partial^2_{IJ}\bm{F}^a_{bc}\big(t,0,y^\prime,\bm{\theta}_0(t,y^\prime)\big)\cdot\partial_J\bm{g}^c_t(t,y^\prime)\bigg|
\end{split}
\end{equation}
as well as the error (for $\overline{\bm{F}}$)
\begin{equation} \label{eq:2ndbarF-1stterm}
\begin{split}
	&\bigg|\partial^2_{IJ}\overline{\bm{F}}^a_{bc}\big(t,s,y,\bm{\theta}_\sigma(t,s,y)\big)\cdot\Big(\sigma\partial_J\bm{u}^c_t(t,s,y)+(1-\sigma)\partial_J\widehat{\bm{u}}^c_t(t,s,y)\Big)-\partial^2_{IJ}\overline{\bm{F}}^a_{bc}\big(t,0,y,\bm{\theta}_0(t,y)\big)\cdot\partial_J\bm{g}^c_t(t,y) \\
	&\quad -\partial^2_{IJ}\overline{\bm{F}}^a_{bc}\big(t,s,y^\prime,\bm{\theta}_\sigma(t,s,y^\prime)\big)\cdot\Big(\sigma\partial_J\bm{u}^c_t(t,s,y^\prime)+(1-\sigma)\partial_J\widehat{\bm{u}}^c_t(t,s,y^\prime)\Big) +\partial^2_{IJ}\overline{\bm{F}}^a_{bc}\big(t,0,y^\prime,\bm{\theta}_0(t,y^\prime)\big)\cdot\partial_J\bm{g}^c_t(t,y^\prime)\bigg|.
\end{split}
\end{equation}
Moreover, we also need to estimate 
\begin{equation} \label{eq:2ndF-2ndterm}
\begin{split}
	&\bigg|\partial^2_{IJ}\bm{F}^a_{bc}\big(t,s,y,\bm{\theta}_\sigma(t,s,y)\big)\cdot\Big(\sigma\partial_J\bm{u}^c_t(t,s,y)+(1-\sigma)\partial_J\widehat{\bm{u}}^c_t(t,s,y)\Big)-\partial^2_{IJ}\bm{F}^a_{bc}\big(t,0,y,\bm{\theta}_0(t,y)\big)\cdot\partial_J\bm{g}^c_t(t,y)\bigg|
\end{split}
\end{equation}
and 
\begin{equation} \label{eq:2ndbarF-2ndterm}
\begin{split}
	&\bigg|\partial^2_{IJ}\overline{\bm{F}}^a_{bc}\big(t,s,y,\bm{\theta}_\sigma(t,s,y)\big)\cdot\Big(\sigma\partial_J\bm{u}^c_t(t,s,y)+(1-\sigma)\partial_J\widehat{\bm{u}}^c_t(t,s,y)\Big)-\partial^2_{IJ}\overline{\bm{F}}^a_{bc}\big(t,0,y,\bm{\theta}_0(t,y)\big)\cdot\partial_J\bm{g}^c_t(t,y)\bigg|.
\end{split}
\end{equation}
Note that
\begin{equation*}
\begin{split}
	\eqref{eq:2ndF-1stterm} \leq &~ \bigg|\partial^2_{IJ}\bm{F}^a_{bc}\big(t,s,y,\bm{\theta}_\sigma(t,s,y)\big)\cdot\Big(\sigma\partial_J\bm{u}^c_t(t,s,y)+(1-\sigma)\partial_J\widehat{\bm{u}}^c_t(t,s,y)\Big) \\
	&\qquad
	-\partial^2_{IJ}\bm{F}^a_{bc}\big(t,s,y^\prime,\bm{\theta}_\sigma(t,s,y^\prime)\big)\cdot\Big(\sigma\partial_J\bm{u}^c_t(t,s,y^\prime)+(1-\sigma)\partial_J\widehat{\bm{u}}^c_t(t,s,y^\prime)\Big)\bigg| \\
	& + \bigg|\partial^2_{IJ}\bm{F}^a_{bc}\big(t,0,y^\prime,\bm{\theta}_0(t,y^\prime)\big)\cdot\partial_J\bm{g}^c_t(t,y^\prime)-\partial^2_{IJ}\bm{F}^a_{bc}\big(t,0,y,\bm{\theta}_0(t,y)\big)\cdot\partial_J\bm{g}^c_t(t,y)\bigg| =: N_1+N_2. 
\end{split}
\end{equation*}
For $N_1$, it holds that 
\begin{equation*}
\begin{split}
	N_1 \leq &~ \bigg|\partial^2_{IJ}\bm{F}^a_{bc}\big(t,s,y,\bm{\theta}_\sigma(t,s,y)\big)\cdot\Big(\sigma\partial_J\bm{u}^c_t(t,s,y)+(1-\sigma)\partial_J\widehat{\bm{u}}^c_t(t,s,y)\Big) \\
	&\qquad\qquad\qquad
	-\partial^2_{IJ}\bm{F}^a_{bc}\big(t,s,y,\bm{\theta}_\sigma(t,s,y)\big)\cdot\Big(\sigma\partial_J\bm{u}^c_t(t,s,y^\prime)+(1-\sigma)\partial_J\widehat{\bm{u}}^c_t(t,s,y^\prime)\Big)\bigg| \\
	& + \bigg|\partial^2_{IJ}\bm{F}^a_{bc}\big(t,s,y,\bm{\theta}_\sigma(t,s,y)\big)\cdot\Big(\sigma\partial_J\bm{u}^c_t(t,s,y^\prime)+(1-\sigma)\partial_J\widehat{\bm{u}}^c_t(t,s,y^\prime)\Big) \\
	&\qquad\qquad\qquad
	-\partial^2_{IJ}\bm{F}^a_{bc}\big(t,s,y^\prime,\bm{\theta}_\sigma(t,s,y^\prime)\big)\cdot\Big(\sigma\partial_J\bm{u}^c_t(t,s,y^\prime)+(1-\sigma)\partial_J\widehat{\bm{u}}^c_t(t,s,y^\prime)\Big)\bigg| \leq C_{17}(R)|y-y^\prime|^\alpha.
\end{split}
\end{equation*}
For $N_2$, 
\begin{equation*}
\begin{split}
	N_2 \leq &~ \bigg|\partial^2_{IJ}\bm{F}^a_{bc}\big(t,0,y^\prime,\bm{\theta}_0(t,y^\prime)\big)\cdot\partial_J\bm{g}^c_t(t,y^\prime)-\partial^2_{IJ}\bm{F}^a_{bc}\big(t,0,y,\bm{\theta}_0(t,y)\big)\cdot\partial_J\bm{g}^c_t(t,y^\prime)\bigg| \\
	& + \bigg|\partial^2_{IJ}\bm{F}^a_{bc}\big(t,0,y,\bm{\theta}_0(t,y)\big)\cdot\partial_J\bm{g}^c_t(t,y^\prime)-\partial^2_{IJ}\bm{F}^a_{bc}\big(t,0,y,\bm{\theta}_0(t,y)\big)\cdot\partial_J\bm{g}^c_t(t,y)\bigg| \leq C_{18}(R)|y-y^\prime|^\alpha. 
\end{split}
\end{equation*}
From the estimates of $N_1$ and $N_2$, we have $\eqref{eq:2ndF-1stterm} \leq  C_{19}(R)|y-y^\prime|^\alpha$.
Moroever, we have 
\begin{equation*}
\begin{split}
	\eqref{eq:2ndF-2ndterm} \leq &~ \bigg|\partial^2_{IJ}\bm{F}^a_{bc}\big(t,s,y,\bm{\theta}_\sigma(t,s,y)\big)\cdot\Big(\sigma\partial_J\bm{u}^c_t(t,s,y)+(1-\sigma)\partial_J\widehat{\bm{u}}^c_t(t,s,y)\Big)-\partial^2_{IJ}\bm{F}^a_{bc}\big(t,s,y,\bm{\theta}_\sigma(t,s,y)\big)\cdot\partial_J\bm{g}^c_t(t,y)\bigg| \\
	& \qquad\qquad\qquad\qquad\qquad\quad+\bigg|\partial^2_{IJ}\bm{F}^a_{bc}\big(t,s,y,\bm{\theta}_\sigma(t,s,y)\big)\cdot\partial_J\bm{g}^c_t(t,y)-\partial^2_{IJ}\bm{F}^a_{bc}\big(t,0,y,\bm{\theta}_0(t,y)\big)\cdot\partial_J\bm{g}^c_t(t,y)\bigg| \leq C_{20}(R)\delta^\frac{\alpha}{2r}. 
\end{split}
\end{equation*}
Similarly, for $\overline{\bm{F}}$, we have $\eqref{eq:2ndbarF-1stterm} \leq C_{21}(R)|y-y^\prime|^\alpha$ and $\eqref{eq:2ndbarF-2ndterm} \leq C_{22}(R)\delta^\frac{\alpha}{2r}$.
Hence, we have 
\begin{equation}
\begin{split}
	\big|\bm{\eta}^a(t,s,y)-\bm{\eta}^a(t,s,y^\prime)\big|\leq &~ C_{19}(R)|y-y^\prime|^\alpha\delta^\frac{\alpha}{2r}\sum\limits_{b\leq m}\left|\partial_I\left(\bm{u}-\widehat{\bm{u}}\right)^b(t,\cdot,\cdot)\right|^{(\alpha)}_{[0,t]\times\mathbb{R}^d}+C_{20}(R)\delta^\frac{\alpha}{2r}|y-y^\prime|^\alpha\sum\limits_{b\leq m}\left|\partial_I\left(\bm{u}-\widehat{\bm{u}}\right)^b(t,\cdot,\cdot)\right|^{(\alpha)}_{[0,t]\times\mathbb{R}^d} \\
	&+C_{21}(R)|y-y^\prime|^\alpha\delta^\frac{\alpha}{2r}\sum\limits_{b\leq m}\left|\partial_I\left(\bm{u}-\widehat{\bm{u}}\right)^b(s,\cdot,\cdot)\right|^{(\alpha)}_{[0,t]\times\mathbb{R}^d}+C_{22}(R)\delta^\frac{\alpha}{2r}|y-y^\prime|^\alpha\sum\limits_{b\leq m}\left|\partial_I\left(\bm{u}-\widehat{\bm{u}}\right)^b(s,\cdot,\cdot)\right|^{(\alpha)}_{[0,t]\times\mathbb{R}^d} \\
	\leq &~C_{23}(R)\delta^\frac{\alpha}{2r}|y-y^\prime|^\alpha\| \bm{u}-\widehat{\bm{u}}\|^{(2r+\alpha)}_{[0,\delta]}.  
\end{split}
\end{equation}

Therefore, together with \eqref{Boundness of eta}, we have 
\begin{equation}
|\bm{\eta}^a(t,\cdot,\cdot)|^{(\alpha)}_{[0,t]\times\mathbb{R}^d}\leq C_{24}(R)\delta^{\frac{\alpha}{2r}}\| \bm{u}-\widehat{\bm{u}}\|^{(2r+\alpha)}_{[0,\delta]}.
\end{equation}

Since $M_1$, $M_3$ and $M_4$ satisfy the same estimates, it holds that 
\begin{equation}
|\bm{\varphi}^a_t(t,\cdot,\cdot)|^{(\alpha)}_{[0,t]\times\mathbb{R}^d}\leq C_{25}(R)\delta^{\frac{\alpha}{2r}}\| \bm{u}-\widehat{\bm{u}}\|^{(2r+\alpha)}_{[0,\delta]}.
\end{equation}

Finally, we have a contraction
\begin{equation}
\|\bm{\Lambda}(\bm{u})-\bm{\Lambda}(\widehat{\bm{u}})\|^{(2r+\alpha)}_{[0,\delta]}\leq C\|\bm{\varphi}\|^{(\alpha)}_{[0,\delta]}=C\sup\limits_{t\in[0,\delta]}\sum_{a\leq m}\left\{\bm{\varphi}^a(t,\cdot,\cdot)|^{(\alpha)}_{[0,t]\times\mathbb{R}^d}+\bm{\varphi}^a_t(t,\cdot,\cdot)|^{(\alpha)}_{[0,t]\times\mathbb{R}^d}\right\}\leq C(R)\delta^\frac{\alpha}{2r}\|\bm{u}-\widehat{\bm{u}}\|^{(2r+\alpha)}_{[0,\delta]}.
\end{equation}  

\ \ 

\noindent (\textbf{A contraction $\bm{\Lambda}$ mapping $\bm{\mathcal{U}}$ into itself.}) To show the contraction, we need to choose a suitably large $R$ such that $\bm{\Lambda}$ maps $\bm{\mathcal{U}}$ into itself. If $\delta$ and $R$ satisfy 
\begin{equation*}
C(R)\delta^{\frac{\alpha}{2r}}\leq \frac{1}{2},
\end{equation*} 
then $\bm{\Lambda}$ is a $\frac{1}{2}$-contraction and for any $\bm{u}\in\bm{\mathcal{U}}$, we have
\begin{equation*}
\|\bm{\Lambda}(\bm{u})-\bm{g}\|^{(2r+\alpha)}_{[0,\delta]}\leq \frac{R}{2}+\| \bm{\Lambda}(\bm{g})-\bm{g}\|^{(2r+\alpha)}_{[0,\delta]}. 
\end{equation*} 
Define the function $\bm{G}:=\bm{\Lambda}(\bm{g})-\bm{g}$ as the solution of the equation 
\begin{equation*}
\left\{
\begin{array}{lr}
	\bm{G}_s(t,s,y)=\bm{L}_0\bm{G}+\bm{F}\big(t,s,y,\left(\partial_I\bm{g}\right)_{|I|\leq 2r}(t,y),  \left(\partial_I\bm{g}\right)_{|I|\leq 2r}(s,y)\big), \\
	\bm{G}(t,0,y)=\bm{0},\hfill 0\leq s\leq t\leq \delta,\quad y\in\mathbb{R}^d.
\end{array}
\right. 
\end{equation*}
By \eqref{Estimates of solutions of nonlocal system}, there is $C>0$, independent of $\delta$, such that 
\begin{equation*}
\| \bm{G}\|^{(2r+\alpha)}_{[0,\delta]}\leq C\sup\limits_{t\in[0,\delta]}\sum\limits_{a\leq m}\left\{\left|\bm{\psi}^a(t,\cdot,\cdot)\right|^{(\alpha)}_{[0,t]\times\mathbb{R}^d}+\left|\bm{\psi}^a_t(t,\cdot,\cdot)\right|^{(\alpha)}_{[0,t]\times\mathbb{R}^d}\right\}=:C^\prime, 
\end{equation*}
where $\bm{\psi}^a(t,s,y)=\bm{F}^a\big(t,s,y,\left(\partial_I\bm{g}\right)_{|I|\leq 2r}(t,y),  \left(\partial_I\bm{g}\right)_{|I|\leq 2r}(s,y)\big)$. 

To conclude, we have 
\begin{equation*}
\|\bm{\Lambda}(\bm{u})-\bm{g}\|^{(2r+\alpha)}_{[0,\delta]}\leq \frac{R}{2}+C^\prime.
\end{equation*}
Therefore for a suitably large $R$, $\bm{\Lambda}$ is a contraction mapping $\bm{\mathcal{U}}$ into itself and it has a unique fixed point $\bm{u}$ in $\bm{\mathcal{U}}$ satisfying  
\begin{equation} \label{Nonlocal fully nonlinear system from 0 to delta} 
\left\{
\begin{array}{lr}
	\bm{u}_s(t,s,y)=\bm{F}\big(t,s,y,\left(\partial_I\bm{u}\right)_{|I|\leq 2r}(t,s,y),  \left(\partial_I\bm{u}\right)_{|I|\leq 2r}(s,s,y)\big), \\
	\bm{u}(t,0,y)=\bm{g}(t,y),\hfill 0\leq s\leq t\leq \delta,\quad y\in\mathbb{R}^d.  
\end{array}
\right. 
\end{equation} 

\ \ 

\noindent (\textbf{Uniqueness}) To complete the proof, we have to show that $\bm{u}$ is the unique solution of \eqref{Nonlocal fully nonlinear system} in $\bm{\Omega}^{(2r+\alpha)}_{[0,\delta]}$. It follows the Schauder-type estimate (Theorem \ref{Schauder estimates}) for the nonlocal, homogeneous,
linear, and strongly parabolic system with initial value zero, which is satisfied
by the difference of any two solutions $\bm{u}$, and $\overline{\bm{u}}$ in $\bm{\Omega}^{(2r+\alpha)}_{[0,\delta]}$ to the system \eqref{Nonlocal fully nonlinear system from 0 to delta}. On the other hand, since $\bm{\Lambda}$ is a contraction, it can be done with some standard arguments.

If \eqref{Nonlocal fully nonlinear system from 0 to delta} admits two fixed points $\bm{u}$ and $\overline{\bm{u}}$, let
$$
\overline{t}=\sup\left\{t\in[0,\delta]:~\bm{u}(t,s,y)=\overline{\bm{u}}(t,s,y),~(t,s,y)\in\Delta[0,t]\times\mathbb{R}^d\right\}.
$$
We shall focus only on the case when $\overline{t}<\delta$ because if $\overline{t}=\delta$, then $\bm{u}=\overline{\bm{u}}$ in the whole $\Delta[0,\delta]\times\mathbb{R}^d$ and the proof is completed. According to the definition of $\overline{t}(<\delta)$, we know that $\bm{u}(t,s,y)=\overline{\bm{u}}(t,s,y)$ in $R_1\times\mathbb{R}^d$ in Figure \ref{fig:uniqueness}. Hence, we obtain diagonal conditions, namely $\left(\partial_I\bm{u}\right)_{|I|\leq 2r}(s,s,y)=\left(\partial_I\overline{\bm{u}}\right)_{|I|\leq 2r}(s,s,y)$ for any $s\in[0,\overline{t}]$ and $y\in\mathbb{R}^d$. By observing \eqref{Nonlocal fully nonlinear system from 0 to delta} provided that the same initial and diagonal conditions (i.e., the initial condition 1 and the diagonal condition in Figure \ref{fig:uniqueness}), the classical PDE theory promises that $\bm{u}$ and $\overline{\bm{u}}$ coincide in $\left(R_1\cup R_2\right)\times\mathbb{R}^d$.
\begin{figure}[!ht]
\centering
\includegraphics[width=0.3\textwidth]{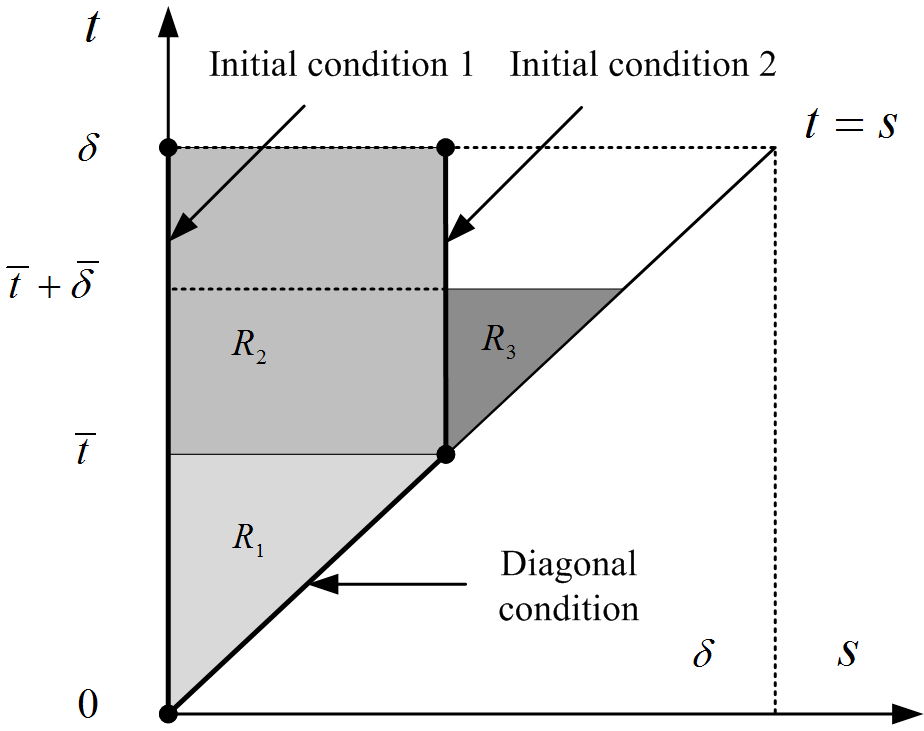}
\caption{Uniqueness of the solution in $\Theta[0,\delta]$}
\label{fig:uniqueness}
\end{figure}

Next, let $\bm{u}(t,\overline{t},y)=\overline{\bm{u}}(t,\overline{t},y)=\overline{\bm{g}}(t,y)$ for $(t,y)\in[\overline{t},T]\times\mathbb{R}^d$. Based on the new initial condition (i.e., the initial condition 2 in Figure \ref{fig:uniqueness}), we consider the following initial value problem: 
\begin{equation} \label{Nonlocal fully nonlinear system with a new initial value condition} 
\left\{
\begin{array}{lr}
	\bm{u}_s(t,s,y)=\bm{F}\big(t,s,y,\left(\partial_I\bm{u}\right)_{|I|\leq 2r}(t,s,y),  \left(\partial_I\bm{u}\right)_{|I|\leq 2r}(s,s,y)\big), \\
	\bm{u}(t,\overline{t},y)=\overline{\bm{g}}(t,y),\hfill \overline{t}\leq s\leq t\leq T,\quad y\in\mathbb{R}^d.  
\end{array}
\right. 
\end{equation} 
Our previous proof shows that \eqref{Nonlocal fully nonlinear system with a new initial value condition} admits a unique solution in the set $$\overline{\bm{\mathcal{U}}}=\left\{\bm{u}\in\bm{\Omega}^{(2r+\alpha)}_{[\overline{t},\overline{t}+\overline{\delta}]}:\bm{u}(t,\overline{t},y)=\overline{\bm{g}}(t,y),\|\bm{u}-\overline{\bm{g}}]^{(2r+\alpha)}_{[\overline{t},\overline{t}+\overline{\delta}]}\leq \overline{R}\right\}$$ 
provided that $\overline{R}$ is large enough and $\overline{\delta}$ is small enough. Considering $\overline{R}$ larger than $[\bm{u}-\overline{\bm{g}}]^{(2r+\alpha)}_{[\overline{t},\overline{t}+\overline{\delta}]}$ and $[\overline{\bm{u}}-\overline{\bm{g}}]^{(2r+\alpha)}_{[\overline{t},\overline{t}+\overline{\delta}]}$, we have $\bm{u}=\overline{\bm{u}}$ in $R_3\times\mathbb{R}^d$. Hence, for any $y\in\mathbb{R}^d$, $\bm{u}$ equals to $\overline{\bm{u}}$ in $\{(t,s):\overline{t}\leq t\leq \overline{t}+\overline{\delta},0\leq s\leq \overline{t}\}\cup R_3$, which contradicts the definition of $\overline{t}$. Consequently, $\overline{t}=\delta$ and $\bm{u}=\overline{\bm{u}}$. This completes the proof.   	
\end{proof}

\begin{proof}[Proof of Theorem \ref{Global existence of fully nonlineaity}]
Assume that $\lim_{s\to\tau}\bm{u}(t,s,y)\in\bm{\mathcal{O}}$. To obtain the global solvability, the maximally defined solution in $[0,\tau)$ has to be extended continuously to a closed interval $[0,\tau]$ such that we can update the initial data with $\bm{u}(\cdot,\tau,\cdot)\in\bm{\Omega}^{(2r+\alpha)}_{[\tau,T]}$.

According to the definition of $\bm{\Omega}^{(2r+\alpha)}_{[a,b]\times[c,d]}$ and the extension procedure in Figure \ref{fig:extension}, for each fixed $t\in[\tau,T]$, it requires that the mapping $\bm{u}:s\mapsto \bm{u}(t,s,y)$ from $[0,\tau)$ to $C^{2r+\alpha}(\mathbb{R}^d;\mathbb{R}^m)$ is at least uniformly continuous such that the limit value $\lim_{s\to\tau}\bm{u}(t,s,y)$ exists. Similar to the proof of Theorem \ref{Schauder estimates}, by the estimate \eqref{Upper bounded for global existence} and extension arguments for the triangle area $(t,s)\in\Delta[0,\sigma]$ to the corresponding trapezoid region $(t,s)\in[0,T]\times[0,t\wedge\sigma]$ for any fixed $t\in[\tau,T]$, we have 
\begin{equation*}
\bm{u}(t,\cdot,\cdot)\in B\big([0,\tau);C^{2r+\alpha^\prime}(\mathbb{R}^d;\mathbb{R}^m)\big), \quad \bm{u}_s(t,\cdot,\cdot)\in B\big([0,\tau);C^{\alpha^\prime}(\mathbb{R}^d;\mathbb{R}^m)\big), 
\end{equation*}
where $B([a,b);X)$ denotes the space of bounded function defined in $[a,b)$ and valued in the Banach space $X$. By an interpolation result (see \cite[Proposition 2.7]{Sinestrari1985} {\color{black}or \cite[Chapter 1]{Lunardi1995}}), it follows that $\bm{u}(t,\cdot,\cdot)\in C^{1-\theta}\big([0,\sigma];C^{\alpha^\prime+2r\theta}(\mathbb{R}^d;\mathbb{R}^m)\big)$ for every $\sigma\in(0,\tau)$ with H\"{o}lder constant independent of $\sigma$. By choosing $\theta=1-\frac{\alpha^\prime-\alpha}{2r}$, it follows 
\begin{equation*}
\bm{u}(t,\cdot,\cdot)\in C^{\frac{\alpha^\prime-\alpha}{2r}}\big([0,\sigma];C^{2r+\alpha}(\mathbb{R}^d;\mathbb{R}^m)\big). 
\end{equation*}
Consequently, for each $t\in[\tau,T]$, $\bm{u}(t,\cdot,\cdot)$ can be continued at $s=\tau$ in such a way that the extension $\bm{u}(\cdot,\tau,\cdot)$ belongs to $\bm{\Omega}^{(2r+\alpha)}_{[\tau,T]}$. After updating with $\bm{u}(\cdot,\tau,\cdot)$ as a new initial condition at $s=\tau$, by Theorem \ref{Well-posedness of fully nonlinear system}, the nonlocal system \eqref{Nonlocal fully nonlinear system} restricted in $(t,s,y)\in[\tau,T]\times[\tau,t]\times\mathbb{R}^d$ admits a unique solution $\bm{u}\in\bm{\Omega}^{(2r+\alpha)}_{[\tau,\tau+\tau_1]}$ for some $\tau_1>0$, which contradicts the definition of $\tau$. Therefore, we have $\tau=T$ or $\lim_{s\to\tau}\bm{u}(\cdot,s,\cdot)\in\partial\bm{\mathcal{O}}$.  
\end{proof}

\begin{proof}[Proof of Theorem \ref{Wellposedness of Quasilinear systems}]
Thanks to Theorem \ref{Well-posedness of fully nonlinear system} and Remark \ref{Maximal interval}, \eqref{Nonlocal quasilinear system} admits a unique maximally defined solution $\bm{u}\in\bm{\Omega}^{{(2r+\alpha)}}_{[0,\tau)}$ in the maximal interval $\Delta[0,\tau)$. We need to prove that the solution can be extended uniquely into $\Delta[0,T]$. 

According to the proof of Theorem \ref{Well-posedness of fully nonlinear system} and the formulation of $\bm{L}_0$, we replace the nonlinearity $\bm{F}$ of \eqref{Definition of Lamda} with the right side of \eqref{Nonlocal quasilinear system}, i.e. $\sum A\partial\bm{u}+\bm{Q}(\bm{u})$.
In the quasilinear case of \eqref{Nonlocal quasilinear system}, it is clear that the radius $R$ of $C(R)\delta^\frac{\alpha}{2r}$ in the proof of Theorem \ref{Well-posedness of fully nonlinear system} only depends on $\|\bm{g}\|^{(2r-1+\alpha)}_{[0,\delta]}$ and $\|\bm{u}\|^{(2r-1+\alpha)}_{[0,\delta]}$ instead of $\|\bm{g}\|^{(2r+\alpha)}_{[0,\delta]}$ and $\|\bm{u}\|^{(2r+\alpha)}_{[0,\delta]}$ as in the fully nonlinear case. Consequently, in order to establish the existence in the large time interval, we only need to investigate and control the solutions of \eqref{Nonlocal quasilinear system} under the norm $\|\cdot\|^{(2r-1+\alpha)}_{[0,\delta]}$. It suffices to show that the mapping $\bm{u}:s\mapsto \bm{u}(t,s,y)$ from $[0,\tau)$ to $C^{2r-1+\alpha}(\mathbb{R}^d;\mathbb{R}^m)$ is uniformly continuous under the conditions in Theorem \ref{Wellposedness of Quasilinear systems}.


To this end, we note from the nonlocal quasilinear system of \eqref{Nonlocal quasilinear system} that 
\begin{equation}  \label{Nonlocal quasilinear system 1}
\left\{
\begin{array}{rcl}
	\left(\frac{\partial\bm{u}}{\partial t}\right)^a_s(t,s,y) & = & \sum\limits_{|I|= 2r,b\leq m}A^{aI}_b(s,y)\partial_I\left(\frac{\partial\bm{u}}{\partial t}\right)^b(t,s,y)+\sum\limits_{|I|\leq 2r-1,b\leq m}\partial_I\bm{Q}^a_b\big(\bm{u}\big)\partial_I\left(\frac{\partial \bm{u}}{\partial t}\right)^b(t,s,y)+\bm{Q}^a_t(\bm{u}), \\
	\frac{\partial\bm{u}}{\partial t}(t,0,y) & = & \bm{g}_t(t,y),\qquad \hfill 0\leq s\leq t<\tau,\quad y\in\mathbb{R}^d, \qquad a=1,\ldots,m. 
\end{array}
\right. 
\end{equation}
where $\partial_I\bm{Q}^a_b\big(\bm{u}\big)$ and $\bm{Q}^a_t(\bm{u})$ represent the first-order partial derivatives of the nonlinearity $\bm{Q}^a$ with respect to its argument $\partial_I\bm{u}^b(t,s,y)$ and $t$, respectively, while both of them are evaluated at $\big(t,s,y,\left(\partial_I\bm{u}\right)_{|I|\leq 2r-1}(t,s,y),  \left(\partial_I\bm{u}\right)_{|I|\leq 2r-1}(s,s,y)\big)$. According to the linear growth condition and bounds of $\bm{Q}$ and the Gr\"{o}nwall--Bellman inequality, it is clear from \eqref{Nonlocal quasilinear system} and \eqref{Nonlocal quasilinear system 1} that there exists a constant $K$ depending only on the given coefficients and data of \eqref{Nonlocal quasilinear system} such that $\|\bm{u}\|^{(2r-1)}_{[0,\tau)}\leq K$. By the classical theory of PDEs, it further implies that $\bm{u}\in\bm{\Omega}^{(2r-1+\alpha)}_{[0,\tau)}$ and $\|\bm{u}\|^{(2r-1+\alpha)}_{[0,\tau)}\leq K$. Consequently, the nonlocal and nonlinear term $\bm{Q}^a(\bm{u})$ of \eqref{Nonlocal quasilinear system} belongs to $\bm{\Omega}^{(\alpha)}_{[0,\tau)}$. Thanks to Theorem \ref{Schauder estimates}, the nonlocal quasilinear system \eqref{Nonlocal quasilinear system} admits a unique solution $\bm{u}\in\bm{\Omega}^{{(2r+\alpha)}}_{[0,T]}$ in $\Delta[0,T]\times\mathbb{R}^d$ and $\|\bm{u}\|^{(2r+\alpha)}_{[0,\tau)}\leq K$, where $K$ could vary from line to line. With similar arguments as the proof of Theorem \ref{Global existence of fully nonlineaity}, for each $t\in[\tau,T]$, we can extend $\bm{u}:s\mapsto \bm{u}(t,s,y)$ from $[0,\tau)$ to $C^{2r-1+\alpha}(\mathbb{R}^d;\mathbb{R}^m)$ at $s=\tau$. With the achieved limit point $\bm{u}(\cdot,\tau,\cdot)\in\bm{\Omega}^{{(2r-1+\alpha)}}_{[\tau,T]}$, by Lemma 8.5.5 in \cite{Lunardi1995}, it follows that $\bm{u}(\cdot,\tau,\cdot)\in\bm{\Omega}^{{(2r+\alpha)}}_{[\tau,T]}$ and $\|\bm{u}\|^{(2r+\alpha)}_{[0,\tau]}\leq K$. Hence, by updating the initial condition, we can extend the maximally defined solution up to the whole time region $\Delta[0,T]$. 
\end{proof} 

{\color{black}
\begin{proof}[Proof of Lemma \ref{EquivalentNorms}]
For the equivalence between \eqref{WeightedNorm1} and \eqref{WeightedNorm2}, we refer the readers to Lemma 2.5 of \cite{Lorenzi2000}. It comes directly from the basic properties of the exponential weight $\varrho(y)$. Next, we will show the equivalence between \eqref{WeightedNorm2} and \eqref{WeightedNorm3}. Let us consider $f(y)\in C^{\alpha}(\mathbb{R}^d;\mathbb{R})$ and $0<|y-y^\prime|\leq 1$. Without loss of generality, we assume that $\langle Sy,y\rangle^{1/2}<\langle Sy^\prime,y^\prime\rangle^{1/2}$. Then,
\begin{equation} \label{equivalenceN3N2}
	\begin{split}
		\left|\frac{f(y)}{\varrho(y^\prime)}-\frac{f(y^\prime)}{\varrho(y^\prime)}\right| & =\left|\frac{f(y)}{\varrho(y^\prime)}-\frac{f(y)}{\varrho({y})}+\frac{f(y)}{\varrho(y)}-\frac{f(y^\prime)}{\varrho(y^\prime)}\right| \leq \left|\frac{f(y)}{\varrho(y)}\right|\left|\langle Sy,y\rangle^\frac{1}{2}-\langle Sy^\prime,y^\prime\rangle^\frac{1}{2}\right|+\left|\frac{f(y)}{\varrho(y)}-\frac{f(y^\prime)}{\varrho(y^\prime)}\right| \\
		& \leq C\Big|\frac{f}{\varrho}\Big|^\infty|y-y^\prime|^\alpha+\Big\langle \frac{f}{\varrho}\Big\rangle^{(\alpha)}_y|y-y^\prime|^\alpha,
	\end{split}
\end{equation}
where $C$ depends on the maximum eigenvalue $\overline{\lambda}$ of $S$. Similarly, it holds that 
\begin{equation} \label{equivalenceN2N3}
	\left|\frac{f(y)}{\varrho(y)}-\frac{f(y^\prime)}{\varrho(y^\prime)}\right| =\left|\frac{f(y)}{\varrho(y)}-\frac{f(y)}{\varrho({y^\prime})}+\frac{f(y)}{\varrho(y^\prime)}-\frac{f(y^\prime)}{\varrho(y^\prime)}\right|\leq C\Big|\frac{f}{\varrho}\Big|^\infty|y-y^\prime|^\alpha+\langle f\rangle^{(\alpha)}_y|y-y^\prime|^\alpha \min\left\{\varrho^{-1}(y),\varrho^{-1}(y^\prime)\right\}.
\end{equation}
Consequently, thanks to \eqref{equivalenceN3N2} and \eqref{equivalenceN2N3}, it is easy to see the equivalence between \eqref{WeightedNorm2} and \eqref{WeightedNorm3}.    
\end{proof}
}

{\color{black}
\begin{proof}[Proof of Theorem \ref{WellposednessWeighted}]
In order to show the well-posedness result and the estimate \eqref{Weighted: Estimates of solutions of nonlocal system} of solutions of the nonlocal linear systems \eqref{Nonlocal linear system} in the weighted space $\bm{\Omega}^{{(l)}}_{\varrho,[0,T]}$, we firstly consider a simplified case of \eqref{Nonlocal linear system} where $A^{aI}_b(t,s,y)=A^{aI}_b(s,y)$, $B^{aI}_b(t,s,y)=0$, $\bm{f}(t,s,y)=\bm{f}(s,y)$, and $\bm{g}(t,s,y)=0$. Then, according to the classical theory of PDE systems (see \cite[Chapter 9]{Friedman1964} or \cite[Chapter 1,3]{Eidelman1969}), \eqref{Nonlocal linear system} admits a unique solution $\bm{u}(s,y)$ of the form 
\begin{equation*}
	\bm{u}(s,y)=\int^s_0d\tau\int_{\mathbb{R}^d}Z(s,\tau,y,\xi)\bm{f}(\tau,\xi)d\xi
\end{equation*}
and for $|I|\leq 2r$, its derivatives are expressed as
\begin{equation*} 
	\partial_I\bm{u}(s,y)= \int_0^sd\tau\int_{\mathbb{R}^d}\partial_IZ(s,\tau,y,\xi)\Big[\bm{f}(\tau,\xi)-\bm{f}(\tau,y)\Big]d\xi+\int^s_0\Big(\partial_I\int_{\mathbb{R}^d}Z(s,\tau,y,\xi)d\xi\Big)\bm{f}(\tau,y)d\xi,
\end{equation*}
where $Z(s,\tau,y,\xi)$ is the fundamental solution of \eqref{Nonlocal linear system} with $A^{aI}_b(t,s,y)=A^{aI}_b(s,y)$ and $B^{aI}_b(t,s,y)=0$. With the upper bound of $Z$ (see \cite[Chapter 9]{Friedman1964} or \cite[Chapter 1,3]{Eidelman1969}), it is clear that
\begin{equation} \label{WeightedB1}
	\left|\frac{\bm{u}(s,y)}{\varrho(y)}\right|\leq\int^s_0d\tau\int_{\mathbb{R}^d}\big|Z(s,\tau,y,\xi)\big|\left|\frac{\bm{f}(\tau,\xi)}{\varrho(\xi)}\right|\left|\frac{\varrho(\xi)}{\varrho(y)}\right| d\xi\leq Cs| \bm{f}|^{(\alpha)}_{\varrho,[0,T]\times\mathbb{R}^d}, 
\end{equation}
where $C$ only depends on the coefficients $A^{aI}_b$ and $S$. The inequality \eqref{WeightedB1} makes full use of the facts that the upper bound of $|Z|$ contains $\exp\{-c|y-\xi|^{\frac{2r}{2r-1}}\}$ while $|\varrho(\xi)/\varrho(y)|$ is bounded by $\exp\{\overline{\lambda}|y-\xi|\}$. Similarly, we also have 
\begin{equation} \label{WeightedB2}
	\begin{split}
		|\partial_I\bm{u}(s,y)| & \leq \int_0^sd\tau\int_{\mathbb{R}^d}\big|\partial_IZ(s,\tau,y,\xi)\big|\big|\bm{f}(\tau,\xi)-\bm{f}(\tau,y)\big|d\xi+\int^s_0\Big|\partial_I\int_{\mathbb{R}^d}Z(s,\tau,y,\xi)d\xi\Big|\Big|\bm{f}(\tau,y)\Big|d\xi \\
		& \leq \int^s_0d\tau\int_{\mathbb{R}^d}\big(|\partial_IZ(s,\tau,y,\xi)|| \bm{f}|^{(\alpha)}_{\varrho,[0,T]\times\mathbb{R}^d}\big(\varrho(\xi)+\varrho(y)\big)|y-\xi|^\alpha\big) d\xi + \int^s_0\Big|\partial_I\int_{\mathbb{R}^d}Z(s,\tau,y,\xi)d\xi\Big|| \bm{f}|^{(\alpha)}_{\varrho,[0,T]\times\mathbb{R}^d}\varrho(y) d\xi \\
		& \leq Cs^\frac{2r-|I|+\alpha}{2r}| \bm{f}|^{(\alpha)}_{\varrho,[0,T]\times\mathbb{R}^d}\varrho(y), 
	\end{split}
\end{equation}
the second inequality of which holds thanks to $\bm{f}\in C^{\frac{\alpha}{2r},\alpha}_\varrho([0,T]\times\mathbb{R}^d;\mathbb{R}^m)$ and the last one is shown by a similar argument of \eqref{WeightedB1} as well as the upper bound of $\partial_IZ$ illustrated in \cite{Friedman1964,Eidelman1969}. 

Furthermore, in order to estimate the H\"{o}lder continuity of $\partial_I\bm{u}$ in $y$, we need to estimate the difference between $\partial_I\bm{u}(s,y)$ and $\partial_I\bm{u}(s,y^\prime)$, denoted by $\Delta\partial_I\bm{u}(s,y)$. First, we consider the case where $s\leq|y-y^\prime|^{2r}$. From \eqref{WeightedB2}, we have   
\begin{equation*}
	|\Delta\partial_I\bm{u}(s,y)|\leq C|y-y^\prime|^{2r-|I|+\alpha}| \bm{f}|^{(\alpha)}_{\varrho,[0,T]\times\mathbb{R}^d}\big(\varrho(y)+\varrho(y^\prime)\big).
\end{equation*}
In particular for $|I|=2r$, we obtain $\frac{|\partial_I\bm{u}(s,y)-\partial_I\bm{u}(s,y^\prime)|}{|y-y^\prime|^\alpha}\min\left\{\varrho^{-1}(y),\varrho^{-1}(y^\prime)\right\}\leq C| \bm{f}|^{(\alpha)}_{\varrho,[0,T]\times\mathbb{R}^d}$. Next, for the case where $|y-y^\prime|^{2r}<s$, we take advantage of the representation of $\Delta\partial_I\bm{u}(s,y)$ (see also page 110 of \cite{Eidelman1969}): 
\begin{equation*}
	\begin{split}
		\Delta\partial_I\bm{u}(s,y)= & \int^{s-\eta}_0\int_{\mathbb{R}^d}\Delta\partial_IZ(s,\tau,y,\xi)\big[\bm{f}(\tau,\xi)-\bm{f}(\tau,y)\big]d\xi+\int^s_{s-\eta}d\tau\int_{\mathbb{R}^d}\partial_IZ\big[\bm{f}(\tau,\xi)-\bm{f}(\tau,y)\big]d\xi \\
		& + \int^{s-\eta}_0\Big(\Delta\partial_I\int_{\mathbb{R}^d}Zd\xi\Big)\bm{f}(\tau,y)d\tau+\int^s_{s-\eta}\partial_I\int_{\mathbb{R}^d}Zd\xi\bm{f}(\tau,y)d\tau \\
		& -\int^s_{s-\eta}d\tau\int_{\mathbb{R}^d}\partial_{I,y^\prime}Z(s,\tau,y^\prime,\xi)\big[\bm{f}(\tau,\xi)-\bm{f}(\tau,y^\prime)\big]d\xi-\int^s_{s-\eta}\partial_{I,y^\prime}\int_{\mathbb{R}^d}Z(s,\tau,y^\prime,\xi)d\xi\bm{f}(\tau,y^\prime)d\tau,
	\end{split}
\end{equation*}
where $\partial_{I,y^\prime}$ is the differential operator in $y^\prime$. With similar arguments in \eqref{WeightedB1} and \eqref{WeightedB2}, we also have $|\Delta\partial_I\bm{u}(s,y)|\leq C|y-y^\prime|^{\alpha}| \bm{f}|^{(\alpha)}_{\varrho,[0,T]\times\mathbb{R}^d}\big(\varrho(y)+\varrho(y^\prime)\big)$ for $|I|=2r$. Consequently, we have a prior estimate of the solution of the simplified system 
\begin{equation}
	\sup\limits_{s\in[0,T]}|\bm{u}(s,\cdot)|^{(2r+\alpha)}_{\varrho,\mathbb{R}^d}\leq C|\bm{f}|^{(\alpha)}_{\varrho,[0,T]\times\mathbb{R}^d}.  
\end{equation}
Furthermore, thanks to the regularities of $A^{aI}_b$ and $\bm{f}$, we have $\sup\limits_{s\in[0,T]}|\bm{u}_s(s,\cdot)|^{(\alpha)}_{\varrho,\mathbb{R}^d}\leq C|\bm{f}|^{(\alpha)}_{\varrho,[0,T]\times\mathbb{R}^d} $ as well. According to the interpolation theory (see \cite[Proposition 1.1.4 or Lemma 5.1.1]{Lunardi1995} or \cite[Proposition 2.7]{Sinestrari1985}), it holds that 
\begin{equation} \label{WeightedSchauderEst}
	|\bm{u}|^{(2r+\alpha)}_{\varrho,[0,T]\times\mathbb{R}^d}\leq C|\bm{f}|^{(\alpha)}_{\varrho,[0,T]\times\mathbb{R}^d}. 
\end{equation}

For the general setting of nonlocal linear system \eqref{Nonlocal linear system}, its global well-posedness and the Schauder's estimate \eqref{Weighted: Estimates of solutions of nonlocal system} can be both proven with the same arguments in Theorem \ref{Well-posedness of u and v} and Theorem \ref{Schauder estimates} by replacing the estimate used in \eqref{Contraction 1} with the weighted one \eqref{WeightedSchauderEst}. After verifying the claims for the nonlocal linear system, it is clear that the conditions of Definition \ref{Def: AppropriatePair} and the updated Schauder's estimate \eqref{Weighted: Estimates of solutions of nonlocal system} suffice to guarantee the local solvability of fully nonlinear systems in the weighted space. The proof is the same as that of Theorem \ref{Well-posedness of fully nonlinear system}. In the same spirit of Subsection \ref{Large Well-posedness of Nonlocal Nonlinear Systems}, the last claim of Theorem \ref{WellposednessWeighted} can be proven as well.  
\end{proof}
}

\begin{proof}[Proof of Lemma \ref{More regularities of solutions}]
It is clear that the claim holds when $k=0$. Let us consider $k=1$. Suppose that $\bm{u}$ is the solution of \eqref{Nonlocal fully nonlinear system} in $\bm{\Omega}^{(2r+\alpha)}_{[0,\delta]}$, then the family of partial derivatives $\frac{\partial\bm{u}}{\partial y_i}$ ($i=1,\ldots,d$) satisfy (for $a=1,\ldots,m$)   
\begin{equation} 
\left\{
\begin{array}{lr}
	\left(\frac{\partial\bm{u}}{\partial y_i}\right)^a_s(t,s,y)=\sum\limits_{|I|\leq 2r,b\leq m}\left(\partial_I\bm{F}^a_b\big(\bm{u}\big)\cdot\partial_I\left(\frac{\partial\bm{u}}{\partial y_i}\right)^b(t,s,y)+\partial_I\overline{\bm{F}}^a_b\big(\bm{u}\big)\cdot\partial_I\left(\frac{\partial\bm{u}}{\partial y_i}\right)^b(s,s,y)\right)+\partial_{y_i}\bm{F}^a\big(\bm{u}\big), \\
	\left(\frac{\partial\bm{u}}{\partial y_i}\right)(t,0,y)=\bm{g}_{y_i}(t,y),\hfill 0\leq s\leq t\leq \delta,\quad y\in\mathbb{R}^d,  
\end{array}
\right. 
\end{equation} 
where $\partial_I\bm{F}^a_b(\bm{u})$, $\partial_I\overline{\bm{F}}^a_b(\bm{u})$ and $\partial_{y_i}\bm{F}^a(\bm{u})$ are the derivatives of $\bm{F}$ all evaluated at
$$\big(t,s,y,\left(\partial_I\bm{u}\right)_{|I|\leq 2r}(t,s,y),  \left(\partial_I\bm{u}\right)_{|I|\leq 2r}(s,s,y)\big).$$
Given $\bm{u}\in\bm{\Omega}^{(2r+\alpha)}_{[0,\delta]}$, all coefficients ($\partial_I\bm{F}^a_b(\bm{u})$ and $\partial_I\overline{\bm{F}}^a_b(\bm{u})$) and the inhomogeneous term $\partial_{y_i}\bm{F}^a(\bm{u})$ are all in $\bm{\Omega}^{(\alpha)}_{[0,\delta]}$. Moreover, the regularity of $\bm{g}\in\bm{\Omega}^{2r+K+\alpha}_{[0,T]}$ ensures $\bm{g}_{y_i}\in\bm{\Omega}^{(2r+\alpha)}_{[0,\delta]}$. Therefore, the well-posedness of nonlocal linear higher-order systems promises $D_y\bm{u}\in\bm{\Omega}^{(2r+\alpha)}_{[0,\delta]}$. Similarly, we can also show iteratively the cases for $k\leq K$. 
\end{proof}

\begin{proof}[Proof of Theorem \ref{F-K formula}]
First, under the regularity assumptions of $\bm{F}$ and $\bm{g}$, Proposition \ref{Equivalance between forward probelms and backward problems} and Lemma \ref{More regularities of solutions} guarantee that there exists a unique solution $\bm{u}(t,s,y)$ of \eqref{Backward nonlocal fully nonlinear equation}, which is first-order continuously differentiable in $s$ and third-order continuously differentiable with respect to $y$. Consequently, the family of random fields $\left(X,Y,Z,\Gamma,A\right)$ defined by \eqref{F-K formula for 2FBSVIE} is well-defined (adapted).

Next, we show that the random field solves the flow of 2FBSDEs, i.e. \eqref{Flow of 2FBSDEs}. For any fixed $(t,s)\in \nabla[t_0,T]$, we apply the It\^{o}'s formula to the map $\tau\to \bm{u}^a(t,\tau,\bm{X}(\tau))$ on $[s,T]$. Then we have 
\begin{equation*}
\begin{split}
	d\bm{u}^a(t,\tau,\bm{X}(\tau))=& \Big[\bm{u}^a_s(t,\tau,\bm{X}(\tau))+\sum^d_{i=1}b_i(\tau,\bm{X}(\tau))\left(\frac{\partial \bm{u}}{\partial y_i}\right)^a(t,\tau,\bm{X}(\tau))+\frac{1}{2}\sum^d_{i,j=1}\left(\sigma\sigma^\top\right)_{ij}(\tau,\bm{X}(\tau))\left(\frac{\partial^2 \bm{u}}{\partial y_i\partial y_j}\right)^a(t,\tau,\bm{X}(\tau))\Big]d\tau \\
	& ~~+\left(\bm{u}^a_y\right)^\top(t,\tau,\bm{X}(\tau))\sigma(\tau,\bm{X}(\tau))d\bm{W}(\tau) \\
	=& \Big[-\bm{F}^a\big(t,\tau,\bm{X}(\tau),\bm{u}(t,\tau,\bm{X}(\tau)),\bm{u}_y(t,\tau,\bm{X}(\tau)),\bm{u}_{yy}(t,\tau,\bm{X}(\tau)),\bm{u}(\tau,\tau,\bm{X}(\tau)),\bm{u}_y(\tau,\tau,\bm{X}(\tau)),\bm{u}_{yy}(\tau,\tau,\bm{X}(\tau)\big) \\
	& ~~ +\frac{1}{2}\sum^d_{i,j=1}\left(\sigma\sigma^\top\right)_{ij}(\tau,\bm{X}(\tau))\left(\frac{\partial^2 \bm{u}}{\partial y_i\partial y_j}\right)^a(t,\tau,\bm{X}(\tau))-\sum^d_{i=1}b_i(\tau,\bm{X}(\tau))\left(\frac{\partial \bm{u}}{\partial y_i}\right)^a(t,\tau,\bm{X}(\tau))\Big]d\tau \\
	& ~~+\left(\bm{u}^a_y\right)^\top(t,\tau,\bm{X}(\tau))\sigma(\tau,\bm{X}(\tau))d\bm{W}(\tau) \\
	=& -\mathbb{F}^a\big(t,\tau,\bm{X}(\tau),\bm{Y}(t,\tau),\bm{Y}(\tau,\tau),\bm{Z}(t,\tau),\bm{Z}(\tau,\tau),\bm{\Gamma}(t,\tau),\bm{\Gamma}(\tau,\tau)\big)d\tau+\left(\bm{Z}^a\right)^\top(t,\tau)d\bm{W}(\tau),
\end{split}
\end{equation*}
which indicates $d\bm{Y}(t,\tau)=-\mathbb{F}ds+\bm{Z}^\top(t,\tau)d\bm{W}(\tau)$. Similarly, for any fixed $(t,s)\in \nabla[t_0,T]$, by applying the It\^{o}'s formula to $\tau\to \left(\sigma^\top \bm{u}^a_y\right)(t,\tau,\bm{X}(\tau))$ on $[t_0,s]$, we can also verify that $d\bm{Z}^a(t,\tau)=\bm{A}^a(t,\tau)d\tau+\bm{\Gamma}^a(t,\tau)d\bm{W}(\tau)$ for $a=1,\ldots,m$. Hence, \eqref{F-K formula for 2FBSVIE} is an adapted solution of \eqref{Flow of 2FBSDEs}.
\end{proof}

\begin{proof}[Proof of Proposition \ref{Solvability of the financial example}]
By the classical theory of ODE systems, we could use the conventional contraction mapping arguments to obtain the local well-posedness of \eqref{General integral equation}. Next, we consider a special case of \eqref{ODE system in a matrix form}, where $\bm{w}$ is a diagonal matrix, i.e. $\bm{w}=\mathrm{diag}\{w^{11},w^{22},\cdots,w^{mm}\}$. Under this condition, the system of ODEs \eqref{ODE system in a matrix form} has the following form  
\begin{equation} \label{Special ODE system} 
\left\{
\begin{array}{lr}
	\bm{\varphi}^a_s(t,s)        
	+\left[\bm{k}^a(t,s)-\sum\limits_{1\leq b\leq m}\beta\left(\frac{\bm{\varphi}^b(s,s)}{\bm{v}^{bb}(s,s)}\right)^\frac{1}{\beta-1}\right]\bm{\varphi}^a(t,s) +\sum\limits_{1\leq b\leq m}\bm{v}^{ab}(t,s) \left(\frac{\bm{\varphi}^b(s,s)}{\bm{v}^{bb}(s,s)}\right)^\frac{\beta}{\beta-1} = 0, \\
	\bm{\varphi}(t,T)=\bm{g}(t),\quad 0\leq t\leq s \leq T, \quad a=1,\ldots,m,
\end{array}
\right. 
\end{equation} 
where $\bm{k}^a(t,s)=k-\bm{w}^{aa}(t,s)$. In this special case, for $a=1,\ldots,m$, we have 
\begin{equation*}
\begin{split}
	\bm{\varphi}^a(t,s) & = \exp\bigg\{\int^T_s\bigg[\bm{k}^a(t,\tau)-\sum\limits_{1\leq b\leq m}\beta\left(\frac{\bm{\varphi}^b(\tau,\tau)}{\bm{v}^{bb}(\tau,\tau)}\right)^\frac{1}{\beta-1}\bigg] d\tau \bigg\}\bm{g}^a(t) \\
	& \qquad\qquad  
	+ \int^T_s \exp\bigg\{\int^\sigma_s\bigg[\bm{k}^a(t,\tau)-\sum\limits_{1\leq b\leq m}\beta\left(\frac{\bm{\varphi}^b(\tau,\tau)}{\bm{v}^{bb}(\tau,\tau)}\right)^\frac{1}{\beta-1}\bigg]d\tau\bigg\}\left(\sum\limits_{1\leq b\leq m}\bm{v}^{ab}(t,\sigma) \left(\frac{\bm{\varphi}^b(\sigma,\sigma)}{\bm{v}^{bb}(\sigma,\sigma)}\right)^\frac{\beta}{\beta-1} \right) d\sigma.
\end{split}
\end{equation*} 
Denoting by $\overline{\bm{\varphi}}^a(s)=\frac{\bm{\varphi}^a(s,s)}{\bm{v}^{aa}(s,s)}$, $\overline{\bm{g}}^a(t)=\frac{\bm{g}^a(t)}{v^{aa}(t,t)}$, and $\overline{\bm{v}}^{ab}(t,s)=\frac{\bm{v}^{ab}(t,s)}{\bm{v}^{aa}(t,t)}$, we obtain 
\begin{equation}
\begin{split}
	\overline{\bm{\varphi}}^a(s) & = \exp\bigg\{\int^T_s\bigg[\bm{k}^a(s,\tau)-\sum\limits_{1\leq b\leq m}\beta \overline{\bm{\varphi}}^b(\tau)^\frac{1}{\beta-1}\bigg] d\tau \bigg\}\overline{\bm{g}}^a(s) \\
	& \qquad\qquad  
	+ \int^T_s \exp\bigg\{\int^\sigma_s\bigg[\bm{k}^a(s,\tau)-\sum\limits_{1\leq b\leq m}\beta \overline{\bm{\varphi}}^b(\tau)^\frac{1}{\beta-1}\bigg]d\tau\bigg\}\left(\sum\limits_{1\leq b\leq m}\overline{\bm{v}}^{ab}(s,\sigma)  \overline{\bm{\varphi}}^b(\sigma)^\frac{\beta}{\beta-1} \right) d\sigma  
\end{split}
\end{equation}

Let
\begin{eqnarray*}
\widehat{\bm{\varphi}}^a(s) & = & \overline{\bm{\varphi}}^a(s)\prod\limits_{1\leq b\leq m}\exp\bigg\{\beta \int^T_s \overline{\bm{\varphi}}^b(\tau)^\frac{1}{\beta-1} d\tau \bigg\}, \quad
\widehat{\bm{g}}^a(s) = \overline{\bm{g}}^a(s)\exp\bigg\{\int^T_s \bm{k}^a(s,\tau) d\tau \bigg\}, \\
\widehat{\bm{v}}^{ab}(s,\sigma) & = & \overline{\bm{v}}^{ab}(s,\sigma)\exp\bigg\{\int^\sigma_s \bm{k}^a(s,\tau) d\tau \bigg\}.
\end{eqnarray*}
Then, we have 
\begin{equation}
\begin{split}
	\widehat{\bm{\varphi}}^a(s) & = \widehat{\bm{g}}^a(s) + \int^T_s \left(\sum\limits_{1\leq b\leq m}\widehat{\bm{v}}^{ab}(s,\sigma)  \widehat{\bm{\varphi}}^b(\sigma)\overline{\bm{\varphi}}^b(\sigma)^\frac{1}{\beta-1} \right) d\sigma.
\end{split}
\end{equation}
We impose the following conditions: there exist some constants $g_0>0$ and $\gamma>0$ such that 
\begin{equation} \label{Condition: g_0}
\widehat{\bm{g}}^a(s)=\overline{\bm{g}}^a(s)\exp\bigg\{\int^T_s \bm{k}^a(s,\tau) d\tau \bigg\}=\frac{\bm{g}^a(s)}{\bm{v}^{aa}(s,s)}\exp\bigg\{\int^T_s \bm{k}^a(s,\tau) d\tau \bigg\}\geq g_0 
\end{equation} 
and 
\begin{equation} \label{Condition: gamma}
\widehat{\bm{v}}^{ab}(s,\sigma)=\overline{\bm{v}}^{ab}(s,\sigma)\exp\bigg\{\int^\sigma_s \bm{k}^a(s,\tau) d\tau \bigg\}=\frac{\bm{v}^{ab}(s,\sigma)}{\bm{v}^{aa}(s,s)}\exp\bigg\{\int^\sigma_s \bm{k}^a(s,\tau) d\tau \bigg\}\geq e^{-\gamma(\sigma-s)} 
\end{equation} 
hold for $a,b=1,\ldots,m$. 
With conditions \eqref{Condition: g_0} and \eqref{Condition: gamma}, we have 
\begin{equation*}
\begin{split}
	\widehat{\bm{\varphi}}^a(s)e^{-\gamma s} & \geq g_0e^{-\gamma s} + \sum\limits_{1\leq b\leq m} \int^T_s   \Big[\widehat{\bm{\varphi}}^b(\sigma)e^{-\gamma\sigma}\Big]\overline{\bm{\varphi}}^b(\sigma)^\frac{1}{\beta-1}d\sigma =: \omega(s), \quad a=1,\ldots,m.   
\end{split}
\end{equation*} 
Note that 
\begin{equation} \label{System for overline varphi} 
\frac{d\omega(s)}{ds} =-\gamma g_0e^{-\gamma s}-\sum\limits_{1\leq b\leq m}\Big[\widehat{\bm{\varphi}}^b(s)e^{-\gamma s}\Big]\overline{\bm{\varphi}}^b(s)^\frac{1}{\beta-1}\leq -\gamma g_0e^{-\gamma s}-\sum\limits_{1\leq b\leq m}\bm{\xi}^b(s)\overline{\bm{\varphi}}^b(s)^\frac{1}{\beta-1},
\end{equation}
\begin{equation*}
\frac{d\left(\omega(s)\prod\limits_{1\leq b\leq m}\exp\bigg\{- \int^T_s \overline{\bm{\varphi}}^b(\tau)^\frac{1}{\beta-1} d\tau \bigg\}\right)}{ds}\leq -\gamma g_0e^{-\gamma s}\prod\limits_{1\leq b\leq m}\exp\bigg\{- \int^T_s \overline{\bm{\varphi}}^b(\tau)^\frac{1}{\beta-1} d\tau \bigg\}.
\end{equation*} 
By integrating both sides above over $[s,T]$, it follows that  
\begin{equation*}
g_0e^{-\gamma T}-\omega(s)\prod\limits_{1\leq b\leq m}\exp\bigg\{- \int^T_s \overline{\bm{\varphi}}^b(\tau)^\frac{1}{\beta-1} d\tau \bigg\}\leq -\gamma g_0\int^T_se^{-\gamma \sigma}\prod\limits_{1\leq b\leq m}\exp\bigg\{- \int^T_\sigma \overline{\bm{\varphi}}^b(\tau)^\frac{1}{\beta-1} d\tau \bigg\} d\sigma.
\end{equation*} 
Hence, for $a=1,\ldots,m$, 
\begin{equation*}
\omega(s) \geq \prod\limits_{1\leq b\leq m}\exp\bigg\{\int^T_s \overline{\bm{\varphi}}^b(\tau)^\frac{1}{\beta-1} d\tau \bigg\}\left(g_0e^{-\gamma T}+\gamma g_0\int^T_se^{-\gamma \sigma}\prod\limits_{1\leq b\leq m}\exp\bigg\{- \int^T_\sigma \overline{\bm{\varphi}}^b(\tau)^\frac{1}{\beta-1} d\tau \bigg\} d\sigma\right).
\end{equation*} 
Thus, for $a=1,\ldots,m$, we have 
\begin{eqnarray}
\overline{\bm{\varphi}}^a(s) & = &\widehat{\bm{\varphi}}^a(s)\prod\limits_{1\leq b\leq m}\exp\bigg\{-\beta \int^T_s \overline{\bm{\varphi}}^b(\tau)^\frac{1}{\beta-1} d\tau \bigg\}\geq e^{\gamma s} \omega(s)\prod\limits_{1\leq b\leq m}\exp\bigg\{-\beta \int^T_s \overline{\bm{\varphi}}^b(\tau)^\frac{1}{\beta-1} d\tau \bigg\} \cr
& \geq & e^{\gamma s}\prod\limits_{1\leq b\leq m}\exp\bigg\{(1-\beta) \int^T_s \overline{\bm{\varphi}}^b(\tau)^\frac{1}{\beta-1} d\tau \bigg\}\left(g_0e^{-\gamma T}+\gamma g_0\int^T_se^{-\gamma \sigma}\prod\limits_{1\leq b\leq m}\exp\bigg\{- \int^T_\sigma \overline{\bm{\varphi}}^b(\tau)^\frac{1}{\beta-1} d\tau \bigg\} d\sigma\right) \cr
& \geq & g_0 e^{-\gamma(T-s)}\geq c>0. \nonumber
\end{eqnarray}
Moreover, we can also obtain the upper bounds:
\begin{eqnarray}
\overline{\bm{\varphi}}^a(s) & = & \exp\bigg\{\int^T_s\bigg[\bm{k}^a(s,\tau)-\sum\limits_{1\leq b\leq m}\beta \overline{\bm{\varphi}}^b(\tau)^\frac{1}{\beta-1}\bigg] d\tau \bigg\}\overline{\bm{g}}^a(s) \cr
& & + \int^T_s \exp\bigg\{\int^\sigma_s\bigg[\bm{k}^a(s,\tau)-\sum\limits_{1\leq b\leq m}\beta \overline{\bm{\varphi}}^b(\tau)^\frac{1}{\beta-1}\bigg]d\tau\bigg\} \left(\sum\limits_{1\leq b\leq m}\overline{\bm{v}}^{ab}(s,\sigma)  \overline{\bm{\varphi}}^b(\sigma)^\frac{\beta}{\beta-1} \right) d\sigma \cr
& \leq & \exp\bigg\{\int^T_s \bm{k}^a(s,\tau) d\tau \bigg\}\overline{\bm{g}}^a(s) + \int^T_s \exp\bigg\{\int^\sigma_s \bm{k}^a(s,\tau) d\tau\bigg\} \left(\sum\limits_{1\leq b\leq m}\overline{\bm{v}}^{ab}(s,\sigma)  \frac{1}{c^\frac{\beta}{1-\beta}} \right) d\sigma \leq C. \nonumber
\end{eqnarray}

After showing $\overline{\bm{\varphi}}^a(s)\in[c,C]$ for any $s\in[0,T]$ and $a=1,\ldots,m$, we are ready to prove the global well-posedness of \eqref{System for overline varphi}. Specifically, by choosing a suitably small $s\in[0,T]$, we can first obtain a small-time solvability of \eqref{System for overline varphi} with the Banach fixed-point arguments. Next, the bounds of $\overline{\bm{\varphi}}(s)$ guarantee the extension from the local solution to an arbitrary large time interval by standard continuation arguments.
\end{proof} 

{\color{black}
\section{Partial derivatives of the nonlinearity} \label{PDNonLinearity}
This appendix presents the partial derivatives of the nonlinearity of \eqref{HJB Exponential utility} $\mathbb{H}:=\mathbb{H}_\gamma(t,s,y,z)$ with respect to its arguments. For its first-order derivative, we have 
\begin{equation} \label{PDNonlinearity1}
\partial_{I}\mathbb{H}^a_b=\left\{
\begin{array}{ll}
	-\bm{w}^{ab}_3(T-t,T-s), & \hbox{if } |I|=0, \\
	\sum\limits_{1\leq b\leq m}\left(\frac{(\mu_b-r)\widehat{\bm{w}}^{b}_1(T-s,T-s,y)+(\mu_b-r)^2 \bm{U}^b_y(s,s,y)}{\widehat{\bm{w}}^{b}_2(T-s,T-s,y)-\sigma_b^2 \bm{U}^b_{yy}(s,s,y)}\right), & \hbox{if } a=b,~|I|=1, \\
	\frac{1}{2}\sum\limits_{1\leq b\leq m}\left(\frac{\sigma_b\widehat{\bm{w}}^{b}_1(T-s,T-s,y)+\sigma_b(\mu_b-r) \bm{U}^b_y(s,s,y)}{\widehat{\bm{w}}^{b}_2(T-s,T-s,y)-\sigma_b^2 \bm{U}^b_{yy}(s,s,y)}\right)^2, & \hbox{if } a=b,~|I|=2, \\
	0, & \hbox{if } a\neq b,~|I|=1,2,
\end{array}
\right. 
\end{equation}
\begin{equation} \label{PDNonlinearity4}
\partial_{I}\overline{\mathbb{H}}^a_b=\left\{
\begin{array}{ll}
	0, & \hbox{if } |I|=0, \\
	\frac{\sigma^2_b(\mu_b-r)\widehat{\bm{w}}^{b}_1(T-s,T-s,y)+\sigma^2_b(\mu_b-r)^2 \bm{U}^b_y(s,s,y)}{\left(\widehat{\bm{w}}^{b}_2(T-s,T-s,y)-\sigma_b^2 \bm{U}^b_{yy}(s,s,y)\right)^2} \bm{U}^a_{yy}(t,s,y) +\frac{(\mu_b-r)^2 }{\widehat{\bm{w}}^{b}_2(T-s,T-s,y)-\sigma_b^2 \bm{U}^b_{yy}(s,s,y)}\bm{U}^a_y(t,s,y), & \\
	+\frac{\gamma\exp\{-rs\}\bm{w}^{b}_1(T-t,T-s,y)(\mu_b-r)}{\widehat{\bm{w}}^{b}_2(T-s,T-s,y)-\sigma_b^2 \bm{U}^b_{yy}(s,s,y)}-\frac{\exp\{-2rs\}\bm{w}^{b}_2(T-t,T-s,y)[2(\mu_b-r)\widehat{\bm{w}}^{b}_1(T-s,T-s,y)+2(\mu_b-r)^2 \bm{U}^b_y(s,s,y)]}{\left(\widehat{\bm{w}}^{b}_2(T-s,T-s,y)-\sigma_b^2 \bm{U}^b_{yy}(s,s,y)\right)^2}, & \hbox{if } |I|=1,
	\\
	\frac{\left(\sigma^2_b\widehat{\bm{w}}^{b}_1(T-s,T-s,y)+\sigma^2_b(\mu_b-r) \bm{U}^b_y(s,s,y)\right)^2}{\left(\widehat{\bm{w}}^{b}_2(T-s,T-s,y)-\sigma_b^2 \bm{U}^b_{yy}(s,s,y)\right)^3} \bm{U}^a_{yy}(t,s,y) + \frac{\sigma_b^2(\mu_b-r)\widehat{\bm{w}}^{b}_1(T-s,T-s,y)+\sigma_b^2(\mu_b-r)^2 \bm{U}^b_y(s,s,y)}{\left(\widehat{\bm{w}}^{b}_2(T-s,T-s,y)-\sigma_b^2 \bm{U}^b_{yy}(s,s,y)\right)^2} \bm{U}^a_y(t,s,y) & \\
	+\frac{\gamma\exp\{-rs\}\bm{w}^{b}_1(T-t,T-s,y)[\sigma_b^2\widehat{\bm{w}}^{b}_1(T-s,T-s,y)+\sigma_b^2(\mu_b-r) \bm{U}^b_y(s,s,y)]}{\left(\widehat{\bm{w}}^{b}_2(T-s,T-s,y)-\sigma_b^2 \bm{U}^b_{yy}(s,s,y)\right)^2}-\frac{\exp\{-2rs\}\bm{w}^{b}_2(T-t,T-s,y)\left(\sigma_b\widehat{\bm{w}}^{b}_1(T-s,T-s,y)+\sigma_b(\mu_b-r) \bm{U}^b_y(s,s,y)\right)^2}{\left(\widehat{\bm{w}}^{b}_2(T-s,T-s,y)-\sigma_b^2 \bm{U}^b_{yy}(s,s,y)\right)^3}, & \hbox{if } |I|=2,
\end{array}
\right. 
\end{equation} 
Furthermore, the second-order derivatives can be derived from \eqref{PDNonlinearity1}-\eqref{PDNonlinearity4}:
$$\partial^2_{It}\mathbb{H}^a_b=\left\{
\begin{array}{lll}
(\bm{w}^{ab}_3)_{T-t}(T-t,T-s), & \hbox{if } |I|=0, \\ 
0, & \hbox{if } |I|=1,2, 
\end{array}
\right. \qquad \partial^2_{IJ}\mathbb{H}^a_{bc}=0 \text{\text{~~for all~~}}|I|=0,1,2,~~ |J|=0,1,2,$$  
$$
\partial^2_{It}\overline{\mathbb{H}}^a_b=\left\{
\begin{array}{ll}
0, & \hbox{if } |I|=0, \\
\frac{-\gamma\exp\{-rs\}(\bm{w}^{b}_1)_{T-t}(T-t,T-s,y)(\mu_b-r)}{\widehat{\bm{w}}^{b}_2(T-s,T-s,y)-\sigma_b^2 \bm{U}^b_{yy}(s,s,y)} & \\
+\frac{\exp\{-2rs\}(\bm{w}^{b}_2)_{T-t}(T-t,T-s,y)[2(\mu_b-r)\widehat{\bm{w}}^{b}_1(T-s,T-s,y)+2(\mu_b-r)^2 \bm{U}^b_y(s,s,y)]}{\left(\widehat{\bm{w}}^{b}_2(T-s,T-s,y)-\sigma_b^2 \bm{U}^b_{yy}(s,s,y)\right)^2}, & \hbox{if } |I|=1, \\
\frac{-\gamma\exp\{-rs\}(\bm{w}^{b}_1)_{T-t}(T-t,T-s,y)[\sigma_b^2\widehat{\bm{w}}^{b}_1(T-s,T-s,y)+\sigma_b^2(\mu_b-r) \bm{U}^b_y(s,s,y)]}{\left(\widehat{\bm{w}}^{b}_2(T-s,T-s,y)-\sigma_b^2 \bm{U}^b_{yy}(s,s,y)\right)^2} & \\
+\frac{\exp\{-2rs\}(\bm{w}^{b}_2)_{T-t}(T-t,T-s,y)\left(\sigma_b\widehat{\bm{w}}^{b}_1(T-s,T-s,y)+\sigma_b(\mu_b-r) \bm{U}^b_y(s,s,y)\right)^2}{\left(\widehat{\bm{w}}^{b}_2(T-s,T-s,y)-\sigma_b^2 \bm{U}^b_{yy}(s,s,y)\right)^3}, & \hbox{if } |I|=2, 
\end{array}
\right.
$$
$$
\partial^2_{IJ}\overline{\mathbb{H}}^a_{bc}=\left\{
\begin{array}{ll}
0, & \hbox{if } |I||J|=0 \hbox{ or }a\not=c, \\
\frac{(\mu_b-r)^2 }{\widehat{\bm{w}}^{b}_2(T-s,T-s,y)-\sigma_b^2 \bm{U}^b_{yy}(s,s,y)} & \hbox{if } |I|=|J|=1,~a=c, \\
\frac{\sigma_b^2(\mu_b-r)\widehat{\bm{w}}^{b}_1(T-s,T-s,y)+\sigma_b^2(\mu_b-r)^2 \bm{U}^b_y(s,s,y)}{\left(\widehat{\bm{w}}^{b}_2(T-s,T-s,y)-\sigma_b^2 \bm{U}^b_{yy}(s,s,y)\right)^2}, & \hbox{if } |I|=2,~|J|=1,~a=c, \\
\frac{\sigma^2_b(\mu_b-r)\widehat{\bm{w}}^{b}_1(T-s,T-s,y)+\sigma^2_b(\mu_b-r)^2 \bm{U}^b_y(s,s,y)}{\left(\widehat{\bm{w}}^{b}_2(T-s,T-s,y)-\sigma_b^2 \bm{U}^b_{yy}(s,s,y)\right)^2} & \hbox{if } |I|=1,~|J|=2,~a=c, \\
\frac{\left(\sigma^2_b\widehat{\bm{w}}^{b}_1(T-s,T-s,y)+\sigma^2_b(\mu_b-r) \bm{U}^b_y(s,s,y)\right)^2}{\left(\widehat{\bm{w}}^{b}_2(T-s,T-s,y)-\sigma_b^2 \bm{U}^b_{yy}(s,s,y)\right)^3}, & \hbox{if } |I|=|J|=2, ~a=c. 
\end{array}
\right.
$$
}

\end{document}